\documentclass[11pt]{amsart}
\usepackage{amsmath, amsmath,amssymb,latexsym,dsfont}
\usepackage[left=2.6cm,right=2.6cm,top=2.7cm,bottom=2.7cm]{geometry}
\usepackage{amsfonts,epsfig, textcomp,esint}
\usepackage{graphicx}
\usepackage{cases}
\usepackage{color}
\usepackage{tikz}
\usepackage{cancel}
\usepackage{amscd}
\usetikzlibrary{positioning}
\usepackage[french,english]{babel}
\usepackage{subfig}
\usepackage[OT1]{fontenc}
\usepackage{amsthm}
\usepackage{enumerate}
\usepackage{hyperref}
\usepackage{cleveref}

\theoremstyle{plain}
\newtheorem{theorem}{Theorem}[section]
\newtheorem{lemma}[theorem]{Lemma}

\newtheorem{proposition}[theorem]{Proposition}
\newtheorem{corollary}[theorem]{Corollary}

\newtheorem{question}[theorem]{Open problem}
\newtheorem{definition}[theorem]{Definition}

\theoremstyle{definition}
\newtheorem{remark}[theorem]{Remark}

\newcommand{\tmu}{\widetilde \mu}

\DeclareMathOperator{\Sp}{Sp}

\newcommand{\hy}{\hat y}
\newcommand{\hu}{\hat u}

\newcommand{\hw}{\hat w}
\newcommand{\hf}{\hat f}

\newcommand{\tvarphi}{\widetilde \varphi}
\newcommand{\tu}{\widetilde u}

\newcommand{\hhy}{\hat{\hat y}}
\newcommand{\hhf}{\hat{\hat f}}

\newcommand{\hh}{\hat h}

\newcommand{\cH}{\mathcal H}
\newcommand{\cL}{\mathcal L}

\newcommand{\cG}{\mathcal G}
\newcommand{\cP}{\mathbb{P}}

\newcommand{\tlambda}{\widetilde \lambda}

\newcommand{\ty}{\widetilde y}
\newcommand{\bu}{{\bf u}}
\newcommand{\by}{{\bf y}}
\newcommand{\tbu}{\widetilde {\bf u}}
\newcommand{\tby}{\widetilde {\bf y}}

\newcommand{\cN}{{\mathcal N}}
\newcommand{\R}{{\mathbb{R}}}
\newcommand{\C}{{\mathbb{C}}}

\newcommand{\N}{\mathbb{N}}

\newcommand{\loc}{\operatorname{loc}}

\newcommand{\eps}{\varepsilon}
\newcommand{\mR}{\mathbb{R}}
\newcommand{\mZ}{\mathbb{Z}}
\newcommand{\cD}{{\mathcal D}}

\newcommand{\mC}{\mathbb{C}}

\newcommand{\supp}{\operatorname{supp}}
\newcommand{\dsp}{\displaystyle}

\newcommand{\M}{{\mathcal M}}

\newcommand{\hcG}{\mathbb{G}}

\newcommand{\diff}[1][-3]{\mathop{}\mkern#1mu{d}}
\numberwithin{equation}{section}

\newcounter{proofstep}
\newcommand{\step}[1]{%
\refstepcounter{proofstep}%
\vskip-\lastskip\medskip\noindent\textit{Step \arabic{proofstep}: #1.}}
\let\oldproof\proof
\renewcommand{\proof}{\oldproof\setcounter{proofstep}{0}}

      \makeatletter
      \def\@setcopyright{}
      \def\serieslogo@{}
      \makeatother
   \newcommand{\tr}{^\mathsf{T}}

\usepackage[normalem]{ ulem }
\usepackage{soul}

\definecolor{darkgreen}{rgb}{0,0.7,0}

\title[On the small-time local controllability of KdV equations]{On the small-time local
controllability of a KdV system for critical lengths}

\author{Jean-Michel Coron}
\address[Jean-Michel Coron]{Sorbonne Universit\'{e}, Universit\'{e} de Paris, CNRS,
INRIA,
	\newline  \indent Laboratoire Jacques-Louis Lions, \'{e}quipe Cage, Paris,
France.}
\email{coron@ann.jussieu.fr.}

\author{Armand Koenig}
\address[Armand Koenig]{Universit\'e C\^ote d’Azur, CNRS, LJAD, France.}
\email{Armand.KOENIG@univ-cotedazur.fr}

\author{Hoai-Minh Nguyen}
\address[Hoai-Minh Nguyen]{Ecole Polytechnique F\'ed\'erale de Lausanne, EPFL,
\newline \indent CAMA, Station 8,  CH-1015 Lausanne, SwitzerSwitzerland.}
\email{hoai-minh.nguyen@epfl.ch}

\begin{document}

\begin{abstract} This paper is devoted to the local null-controllability
of  the nonlinear  KdV equation equipped the Dirichlet boundary conditions using the
Neumann boundary control on the right.  Rosier proved that this KdV system is
small-time locally controllable for all non-critical lengths and that the
uncontrollable space of the linearized system is of finite dimension when the length is
critical. Concerning critical lengths, Coron and Cr\'{e}peau showed that the same result
holds when the uncontrollable space of the linearized system is of
dimension 1, and later Cerpa, and then Cerpa and Cr\'epeau established that the local
controllability holds at a finite time for all other critical
lengths. In this paper, we prove that, for a class of critical lengths, the nonlinear KdV
system is {\it not} small-time locally controllable.
\end{abstract}

\maketitle



\tableofcontents

\section{Introduction}

We are concerned about the local null-controllability of the (nonlinear) KdV
equation equipped the Dirichlet boundary conditions using the Neumann boundary
control on the right.  More precisely, given $L > 0$ and $T>0$, we consider the
following  control system
\begin{equation}\label{intro-sys-KdV}\left\{
\begin{array}{cl}
y_t (t, x) + y_x (t, x) + y_{xxx} (t, x) + y (t, x) y_x(t, x) = 0 &  \mbox{ for } t
\in (0, T), \, x \in (0, L), \\[6pt]
y(t, x=0) = y(t, x=L) = 0 & \mbox{ for } t \in (0, T), \\[6pt]
y_x(t , x= L) = u(t) & \mbox{ for } t \in (0, T),
\end{array}\right.
\end{equation}
and
\begin{equation}\label{intro-IC-KdV}
y(t = 0, x)  = y_0 (x) \mbox{ for } x \in  (0, L).
\end{equation}

Here $y$ is the state, $y_0$ is the initial data,  and $u$ is
the control. More precisely, we are interested in the \emph{small-time local
controllability} property of this system.


The KdV equation has been introduced by Boussinesq \cite{1877-Boussinesq} and Korteweg
and de Vries \cite{KdV} as a model for propagation of surface water waves along a
channel.
This equation also furnishes a very useful nonlinear approximation model
including a balance between a weak nonlinearity and weak dispersive effects. The KdV
equation has been intensively studied from various aspects of mathematics, including the
well-posedness, the existence and stability of solitary waves, the integrability, the
long-time behavior, etc., see e.g.~\cite{Whitham74, Miura76, Kato83, Tao06, LP15}.

\subsection{Bibliography}\label{sec:biblio}

The controllability properties of system \eqref{intro-sys-KdV} and
\eqref{intro-IC-KdV} (or of its variants) has been studied
intensively, see e.g. the surveys \cite{RZ09, Cerpa14} and the references therein. Let us briefly review the
existing results on \eqref{intro-sys-KdV} and
\eqref{intro-IC-KdV}. For initial and final datum in $L^2(0, L)$ and controls in $L^2(0, T)$,
Rosier~\cite{Rosier97} proved that the system
is small-time locally controllable around 0 provided that the length $L$ is not critical,
i.e., $L \notin \cN$, where
\begin{equation}\label{def-cN}
\cN : = \left\{ 2 \pi \sqrt{\frac{k^2 + kl + l^2}{3}}; \, k, l \in \N_*\right\}.
\end{equation}
To this end, he studied the controllability of the linearized system using the Hilbert
Uniqueness Method  and compactness-uniqueness arguments. Rosier also  showed that the
linearized system is controllable if $L \not \in \cN$. He as well established
that when $L \in \cN$,  the linearized system is not controllable. More precisely, he showed
that there exists a non-trivial finite-dimensional subspace $\M$ of $L^2(0, L)$  such that its orthogonal space is reachable from $0$ whereas  $\M$ is not.

To tackle the control problem for the critical length $L \in \cN$ with  initial and final  datum in $L^2(0, L)$ and controls in $L^2(0, T)$, Coron and Cr\'epeau
introduced the
power series expansion method \cite{CC04}. The idea is to take into account the effect of
 the nonlinear term $y y_x$  absent in  the linearized system. Using this
method, they showed \cite{CC04} (see also \cite[section 8.2]{Coron07}) that  system \eqref{intro-sys-KdV} and
\eqref{intro-IC-KdV}  is
small-time locally controllable if $L = m 2 \pi$ for $m \in \N_*$ satisfying
\begin{equation}
\nexists (k, l) \in \N_* \times \N_* \mbox{ with } k^2 + kl + l^2 = 3 m^2 \mbox{ and }
k \neq l.
\end{equation}
In this case,  $\dim \M = 1$ and $\M$ is spanned by $1 - \cos x$. Cerpa \cite{Cerpa07}
developed the analysis in \cite{CC04} to prove that system \eqref{intro-sys-KdV} and
\eqref{intro-IC-KdV}  is locally
controllable at \emph{a finite time} in the case $\dim \M = 2$. This corresponds to
the
case where
\[
L = 2 \pi \sqrt{\frac{k^2 + kl + l^2}{3}}
\]
for some $k, \,  l \in \N_*$ with   $k>l$, and there is no $m, n \in N_*$ with $m>n$
and $m^2 + mn + n^2 = k^2 + kl + l^2$. Later, Cr\'epeau and Cerpa \cite{CC09}
succeeded to extend  the ideas in \cite{Cerpa07} to obtain the local
controllability for all other critical lengths at {\it a finite time}. To summarize,
concerning the critical lengths
with  initial and final  datum in $L^2(0, L)$ and controls in $L^2(0, T)$, the small-time local controllability is valid when  $\dim \M
= 1$ and local controllability in a large enough time holds when $\dim \M \ge 2$.

\subsection{Statement of the result}

The control properties of the KdV equations have been intensively studied previously  but the following natural question remains open (see
\cite[Open problem 10]{Coron07-Survey}, \cite[Remark 1.7]{Cerpa07}):

\begin{question} \label{OQ} Is system \eqref{intro-sys-KdV} and \eqref{intro-IC-KdV} small-time
locally controllable for all $L \in \cN$?
\end{question}

In this paper we give a negative answer to this question. We show that
system~\eqref{intro-sys-KdV} and~\eqref{intro-IC-KdV} is not small-time locally
controllable for a class of critical lengths. More precisely, we have

\begin{theorem}\label{thm-main} Let $k, \,  l \in \N_*$ be such that $2 k + l \not \in  3
\N_*$. Assume that
\[
L = 2 \pi \sqrt{\frac{k^2 + k l + l^2}{3}}.
\]
Then system \eqref{intro-sys-KdV} and \eqref{intro-IC-KdV} is not
small-time locally null-controllable with controls in $H^1$ and initial and final datum in $H^3(0, L) \cap H^1_0(0, L)$, i.e.,  there exist $T_0>0$ and $\eps_0 > 0$ such
that,  for all $\delta >0$, there is $y_0 \in H^3(0, L) \cap H^1_0(0, L)$ with $\| y_0\|_{H^3(0, L)} <
\delta$ such that for all $u \in H^1(0, T_0)$ with $\| u\|_{H^1(0, T_0)} < \eps_0$ and $u(0) = y_0'(L)$, we
have
\[
y(T_0, \cdot) \not \equiv 0,
\]
where $y \in C\big([0, T_0]; H^3(0, L) \big)  \cap L^2\big([0, T_0]; H^4(0, L)\big)$ is the unique solution of \eqref{intro-sys-KdV} and \eqref{intro-IC-KdV}.
\end{theorem}

\begin{question} \label{op1}  We are not able to establish that the control system  \eqref{intro-sys-KdV} and \eqref{intro-IC-KdV} is not small-time  locally controllable with  initial and final  datum in $L^2(0, L)$ and control in $L^2(0, T)$ for a critical length as in \Cref{thm-main}.
It would be interesting to extend the method in the paper to deal with this problem. 
It would be also interesting to know what is the smallest $s$ such that system \eqref{intro-sys-KdV} and \eqref{intro-IC-KdV} is not small-time locally controllable with controls in $H^s(0,T)$,  and initial and final  datum in $D(\mathcal{A}^s)$, $\mathcal{A}$ being defined in \Cref{lem-EU} below.  
\end{question}

\begin{remark} \rm Concerning  \Cref{op1}, may be the smallest $s$  is not an integer, as in the nonlinear parabolic equation studied in \cite{BM20}, a phenomenon which is specific to the  infinite dimension as shown in \cite{BM18}. Note that in \cite{M18} a non integer $s$ already appears  for an obstruction to small-time local controllability; however it is not known if this $s$ is the optimal one.

\end{remark}

\begin{question} It would be also interesting to know what is the optimal time for the local null controllability. In particular one may ask if
$T \leq T^>$, with $T^>$ defined in \cite[p. 463]{CC09}, then the control system  \eqref{intro-sys-KdV} and \eqref{intro-IC-KdV} is not locally null controllable in time $T$ (for example with  initial and final  datum in $H^3(0, L) \cap H^1_0(0, L)$ and control in $H^1(0, T)$) for  critical lengths $L$ as in the above theorem.
\end{question}

\begin{question}
Finally, it would be interesting to know if the assumption $2 k + l \not \in  3
\N_*$ can be replaced by the weaker assumption $\dim \M > 1$. In other words, is it true that the control system  \eqref{intro-sys-KdV} and \eqref{intro-IC-KdV} is not small time locally controllable when $\dim \M > 1$? 
\end{question}

In \Cref{thm-main}, we deal with  controls in $H^1(0, T_0)$,  and initial and final datum in $H^3(0, L) \cap H^1_0(0, L)$ instead of controls in $L^2(0, T_0)$, and initial and final datum in $L^2(0, L)$ as considered in \cite{Rosier97, CC04, Cerpa07, CC09}.
For a subclass of the critical lengths considered in \Cref{thm-main}, we prove later (see \Cref{thm-CP} in \Cref{sect-CP}) that system \eqref{intro-sys-KdV} and \eqref{intro-IC-KdV} is locally controllable with  initial and final  datum in $H^3(0, L) \cap H^1_0(0, L)$ and  controls in $H^1(0, T)$.  It is worth noting that even though the  propagation speed of the KdV equation is
infinite, some time is needed to reach the zero state.

We emphasize that there are other types of boundary controls for the KdV equations
for which there is no critical length, see   \cite{Rosier97, Rosier04, GG08, Cerpa14}.
There are also results on  internal controllability for the KdV equations, see
\cite{RZ96}, \cite{CPR15} and references therein.

A minimal time of the null-controllability is also required for some linear partial differential equations.
This is obviously the case for equations with a finite speed of propagation, such as
the transport equation~\cite[Theorem.~2.6]{Coron07},  or the wave equation~\cite{BLR92,
BG97},  or hyperbolic systems \cite{CoronNg19}. But this can also happen for equations with infinite speed of propagation, such as
some parabolic systems~\cite{ABG11, BAM18}, Grushin-type equations \cite{BMM15, BDE20,
DK20}, Kolmogorov-type equations~\cite{BHHR15} or parabolic-transport coupled
systems~\cite{BKL20},  and the  references therein. Nevertheless,  a minimal time required for the KdV equations using boundary controls is  observed and established for  the first time in this work to our knowledge. This fact is surprising when compared with known results on internal controls for KdV system \eqref{intro-sys-KdV} with $u=0$. It is known, see \cite{CPR15, PVZ05, Pazoto05}, that
the KdV system \eqref{intro-sys-KdV} with $u=0$ is local controllable using internal controls {\it whenever} the  control region contains  an {\it arbitrary} open subset of $(0, L)$.

However our obstruction to small-time local controllability of our KdV control system is of a different nature than these obstructions to small-time null controllability for linear partial differential equations. It comes from a phenomena which already appears
in finite dimension for \textit{nonlinear} control systems. Note that in finite dimension, in contrast to the case of partial differential equations as just pointed above, a linear control system which is controllable in large time is controllable in arbitrary small time. This is no longer the case for nonlinear control systems in finite dimension: There are nonlinear control systems in finite dimension which are locally controllable in large enough time  but are not locally controllable in small time. A typical example is the control system
\begin{equation}\label{eq-nonlinear}
\dot y_1=u,\quad  \dot y_2=y_3,  \quad \dot y_3=-y_2+2y_1u,
\end{equation}
where the state is $(y_1,y_2,y_3)\tr \in \R^3$ and the control is $u\in\R$. There are many powerful necessary conditions for  small-time local controllability of nonlinear control systems in finite dimension. Let us mention in particular the Sussmann condition \cite[Proposition 6.3]{1987-Sussmann-SICON}. See also \cite{BM18} by Beauchard and Marbach for further results, in particular for controls in the Sobolev spaces $H^k(0,T)$, and a different approach. The Sussmann condition \cite[Proposition 6.3]{1987-Sussmann-SICON} tells us that the nonlinear control system \eqref{eq-nonlinear} is not small-time locally controllable (see \cite[Example 3.38]{Coron07}): it gives a precise direction, given by an explicit iterated Lie bracket, in which one cannot move in small time. For partial differential equations iterated Lie brackets can sometimes be defined, at least heuristically, for interior controls but are not well understood  for boundary controls (see \cite[Chapter 5]{Coron07}), which is the type of controls considered here. However, for the simple control system \eqref{eq-nonlinear}, an obstruction to small-time local controllability can be obtained by pointing out that
if $(y,u):[0,T]\to \R^3\times\R$ is a trajectory of the control system \eqref{eq-nonlinear} such that $y(0)=0$, then
\begin{gather}
\label{y2}
y_2(T)=\int_0^T\cos(T-t)y_1^2(t)\diff t,
\\
\label{y4}
y_3(T)=y_1(T)^2-\int_0^T\sin(T-t)y_1^2(t)\diff t.
\end{gather}
Hence,
\begin{gather}
\label{sens-deplacement-0}
y_2(T)\geq 0 \text{ if } T\in[0,\pi/2]
\\
\label{sens-deplacement}
y_3(T)\leq 0 \text{ if } T\in[0,\pi] \text{ and } y_1(T)=0,
\end{gather}
which also show that the control system \eqref{eq-nonlinear} is not small-time locally controllable and more precisely, using \eqref{sens-deplacement}, is not locally controllable in time $T\in[0,\pi]$ (\eqref{sens-deplacement-0} gives only an obstruction for $T\in [0,\pi/2]$). Note that condition \eqref{sens-deplacement-0}, at least for $T>0$ small enough, is the obstruction  to small-time local controllability given by  \cite[Proposition 6.3]{1987-Sussmann-SICON}, while \eqref{sens-deplacement} is not related to this proposition. For the control system \eqref{eq-nonlinear} one knows that it is locally controllable in a large enough time and the optimal time for local controllability is also known: this control system is locally controllable in time $T$ if and only if $T > \pi$; see \cite[Example 6.4]{Coron07}. Moreover, if there are higher order perturbations (with respect to
the weight $(r_1,r_2,r_3)= (1,2,2)$ for the state and $1$ for the control; see \cite[Section 12.3]{Coron07}) one can still get an obstruction to small-time local controllability by pointing out that \eqref{y2} and \eqref{y4} respectively imply
\begin{gather}
\label{y2-coerciv}
\text{for every $T\in (0,\pi/2)$ there exists $\delta>0$ such that $y_2(T)\geq \delta |u|_{H^{-1}(0,T)}^2$},
\\
\label{y3-coerciv}
\text{for every $T\in (0,\pi]$ there exists $\delta>0$ such that if $y_1(T)=0$, then $y_3(T)\leq -\delta |u|_{H^{-2}(0,T)}^2$.}
\end{gather}
Assertion \eqref{y3-coerciv} follows from the following facts:  
$$
\int_0^T\Big(\int_0^t y_1 (s) \, ds \Big)^2 \, dt \leq \int_0^T t\int_0^t y_1(s)^2 \, ds \, dt  \leq T \int_0^T (T-s)  y_1(s)^2 \, ds, 
$$
$$
\int_0^T \Big(\int_t^Ty(s) \, ds \Big)^2 \, dt \leq \int_0^T(T-s) y(s)^2 \, ds, 
$$
and, since $y_1' = u $ and $y_1(0) = 0$,  
$$
\| u \|_{H^{-2} (0, T)}^2  \le C \int_0^T\Big(\int_0^t y_1 (s) \, ds \Big)^2 \, dt + C \Big( \int_0^T y_1(s) \, ds \Big)^2. 
$$
Note that inequality \eqref{y2-coerciv} does not require any condition on the control, while \eqref{y3-coerciv} requires that the control is such that $y_1(T)=0$. On the other hand it is \eqref{y3-coerciv} which gives the largest time for the obstruction to local controllability in time $T$: \eqref{y2-coerciv} gives an obstruction  for $T\in[0,\pi/2)$, while \eqref{y3-coerciv} gives an obstruction  for $T\in[0,\pi]$, which in fact optimal as mentioned above.

 There are nonlinear partial differential equations where related inequalities giving an obstruction to small-time local controllability were already proved, namely  nonlinear Schr\"{o}dinger control systems considered by Coron in \cite{Coron06} and by Beauchard and Morancey in \cite{2014-Beauchard-Morancey-MCRF}, a viscous Burgers equation considered by Marbach in \cite{M18}, and  a nonlinear parabolic equation considered by Beauchard and Marbach in \cite{BM20}. Our obstruction to small-time local controllability is also in the same spirit (see in particular \Cref{corollary-dir}). Let us briefly explain some of the main ingredients of these previous works.
\begin{itemize}
\item In \cite{Coron06} and \cite{2014-Beauchard-Morancey-MCRF}, the control is interior and one can compute, at least formally, the iterated Lie bracket \cite{1987-Sussmann-SICON} in which one could not move in small time (see \cite[Section 9.3.1]{Coron07}) if the control systems were in finite dimension. Then one checks by suitable computations that it is indeed not possible to move in small time in this direction by proving an inequality analogous to \eqref{y3-coerciv}. The computations are rather explicit due to the fact that the drift\footnote{If the linearized control system is written in the form $\dot y =Ay+Bu$, the drift term is the map $y\mapsto Ay$} of the linearized control system is skew-adjoint with explicit and simple eigenvalues and eigenfunctions.
\item In \cite{M18} the control is again interior. However the iterated Lie bracket 
\cite{1987-Sussmann-SICON} in the direction of which one could not move in small time 
turns out to be $0$. Hence it does not produce any obstruction to small-time local 
controllability. However an inequality analogous to \eqref{y2-coerciv} is proved, but with 
a fractional (non integer) Sobolev norm. An important tool of the proof is a change of 
time-scale which allows to do an expansion with respect to a new parameter. In the 
framework of \eqref{eq-nonlinear}, this leads to a boundary layer which is analyzed thanks 
to the maximum principle. Here the drift term of the linearized control system is 
self-adjoint with explicit and simple eigenvalues and eigenfunctions.
\item In \cite{BM20} the control is again an interior control.  Two cases are considered, a case \cite[Theorem 3]{BM20} related to \cite{Coron06} and \cite{2014-Beauchard-Morancey-MCRF} (already analyzed above) and a case \cite[Theorem 4]{BM20} where classical obstructions relying on iterated Lie brackets fail. Concerning \cite[Theorem 4]{BM20} the proof relies on an inequality of type \eqref{y3-coerciv}. The proof of the inequality of type \eqref{y3-coerciv} can be performed by explicit computations due to some special structure of the quadratic form one wants to analyze: roughly speaking it corresponds to the case (see \cite[(4.17)]{BM20}) where \eqref{identity-1} below would be replaced by
\begin{equation}\label{identity-1-new}
\int_0^L \int_{0}^{+ \infty} |y(t, x)|^2 \varphi_x(x) e^{- i p t} \diff t \diff x  =
\int_{\mR}  \hu(z) \overline{\hu(z)} \int_0^L B(z, x) \diff x \diff z, 
\end{equation}
which simplifies the analysis the left hand side of \eqref{identity-1-new}   (in \eqref{identity-1} one has $\hu(z)\overline{\hu(z-p)}$ instead of $\hu(z)\overline{\hu(z)}$). The computations are also simplified by the fact that the drift term of the linearized control system is self-adjoint with, again, explicit eigenvalues and eigenfunctions.
\end{itemize}
In this article we prove an estimate of type \eqref{y3-coerciv}, instead of \eqref{y2-coerciv}, expecting that with more precise estimates one might get the optimal time for local controllability as for the control system \eqref{eq-nonlinear}. The main differences of our study compare with those of  these previous articles are the following ones.
\begin{itemize}
\item This is the first case dealing with boundary controls. In our case one does not know what are the
iterated Lie brackets even heuristically. Let us take this opportunity  to point out that, even if they are expected to not leave in the state space (see \cite[pages 181--182]{Coron07}), that would be very interesting to understand what are these iterated Lie brackets.
\item It sounds difficult to perform the change of time-scale introduced in \cite{M18} in our situation. Indeed this change will also lead to a boundary layer. However one can no longer use the maximum principle to study this boundary layer. Moreover if the change of time-scale, if justified, allows simpler computations\footnote{This is in particular due to the fact that for the limit problem one has again \eqref{identity-1-new}}, the advantage for not using it might be to get better or more explicit time for the obstruction to small-time local controllability.
\item The linear drift term of the linearized control system (i.e. the operator $\mathcal{A}$ defined in Lemma \ref{lem-EU}) is neither self-adjoint nor skew-adjoint. Moreover its eigenvalues and eigenfunctions are not explicit.
\item Finally,  \eqref{identity-1-new} does not hold.
\end{itemize}

\subsection{Ideas of the analysis} Our approach is inspired by the power
series expansion method introduced by Coron and Cr\'epeau \cite{CC04}.
The idea of this method is to search/understand a control $u$ of the form
\[
u = \eps u_1 + \eps^2 u_2 +  \cdots.
\]
The corresponding solution then formally has the form
\[
y = \eps y_1 + \eps^2 y_2  + \cdots,
\]
and the non-linear term $y y_x$ can be written as
\[
y y_x = \eps^2 y_1 y_{1, x} + \cdots.
\]
One then obtains the following systems
\begin{equation}\label{eq:first_order}\left\{
\begin{array}{cl}
y_{1, t} (t, x) + y_{1, x} (t, x) + y_{1, xxx} (t, x) = 0 &  \mbox{ for } t \in
(0, T), \, x \in (0, L), \\[6pt]
y_1(t, x=0) = y_1(t, x=L) = 0 & \mbox{ for } t \in (0, T), \\[6pt]
y_{1, x}(t , x= L) = u_1(t) & \mbox{ for } t \in (0, T),
\end{array}\right.
\end{equation}
\begin{equation} \label{eq:second_order}\left\{
\begin{array}{cl}
y_{2, t} (t, x) + y_{2, x} (t, x) + y_{2, xxx} (t, x) + y_1 (t, x) y_{1, x} (t, x)  = 0 &
\mbox{ for } t \in (0, T), \, x \in (0, L), \\[6pt]
y_2(t, x=0) = y_2(t, x=L)  = 0 & \mbox{ for } t \in (0, T), \\[6pt]
y_{2, x}(t , x= L) = u_2(t) & \mbox{ for } t \in (0, T). 
\end{array}\right.
\end{equation}
The idea in \cite{Cerpa07, CC09} with its root in \cite{CC04}  is then to find $u_1$ and $u_2$ such that, if
$y_1(0, \cdot)=y_2(0, \cdot)=0$, then $y_1(T, \cdot) = 0$ and the $L^2(0, L)$-orthogonal projection of
$y_2(T)$ on $\M$ is a given (non-zero) element in $\M$. In \cite{CC04}, the authors  needed to make
an expansion up to the order $3$ since $y_2 $ belongs to the orthogonal space of $\M$  in
this case. To this end,  in \cite{CC04, Cerpa07, CC09}, the authors
used delicate contradiction arguments to capture the structure of the KdV systems.

The analysis in this paper has the same root as the ones mentioned above. Nevertheless,
instead of using a contradiction argument, our strategy is to characterize all possible
$u_1$ which steers 0 at time 0 to $0$ at time $T$ (see \Cref{pro-Gen}).  This is done by
taking the  Fourier transform with respect to time  of the solution $y_1$  and applying
Paley-Wiener's theorem.  Surprisingly, in the case $2k + l \neq 3 \N_*$, if the time $T$
is sufficiently small, there are directions in $\M$ which cannot be reached via $y_2$
(see \Cref{corollary-dir} and \Cref{lem-E}). This is one of the crucial observations in this
paper. Using this observation, we then implement a method to prove
the obstruction for the small-time local null-controllability of the KdV system, see
\Cref{thm-NL}. The idea is to bring the nonlinear context to the one, based on the power series expansion approach,  where the new
phenomenon  is observed (the context of \Cref{corollary-dir}).  To be able to reach the
result as stated in \Cref{thm-main}, we establish several new estimates for the
linear and nonlinear KdV systems using low regularity data (see \Cref{sect-LKdV} for the
linear  and \Cref{lem-kdvNL} for  the nonlinear settings). Their proofs partly involve a
connection between the linear KdV equation and the linear KdV-Burgers equation as previously
used by Bona et al. \cite{Bona09} and inspired by the work of
Bourgain \cite{Bourgain93}, and  Molinet and Ribaud \cite{MR02}. To establish the local controllability for a subclass of critical lengths in  a finite time (\Cref{thm-CP}), we apply again the power series method and use a fixed point argument. The key point here is first to obtain  controls in $H^1(0, T)$ to control directions  which can be reached via the linearized system and second to obtain controls in $H^1(0, T)$ for $y_1$ and $y_2$ mentioned above. The analysis of the first part is based on a modification of the Hilbert Uniqueness Method and the analysis of the second part is again based on the information obtained in  \Cref{corollary-dir} and \Cref{lem-E}. Our fixed point argument is inspired by \cite{CC04, Cerpa07} but is different, somehow simpler, and, more importantly, relies on  the usual Banach fixed point theorem instead of the Brouwer fixed point theorem, which might be interesting to handle nonlinear partial  differential equations such that $\M$ is of infinite dimension, as, for example, in \cite{M18}.

\subsection{Structure of the paper} The paper is organized as follows. \Cref{sect-0-0} is devoted to the study of controls
which steers 0 to 0 (motivated by the system of $y_1$). In \Cref{sect-dir}, we study
attainable directions for small time via the power series approach (motivated by the
system of $y_2$).
The main result in this section is \Cref{pro-monotone} whose consequence (\Cref{corollary-dir}) is crucial in the proof of \Cref{thm-main}. In \Cref{sect-L-KdV}, we established several useful estimates for linear KdV systems. In \Cref{sect-NL-KdV}, we give the proof of \Cref{thm-main}. In fact, we will establish a result (\Cref{thm-NL}),  which implies \Cref{thm-main} and reveals a connection with unreachable directions via the power series expansion method. In \Cref{sect-CP}, we establish the local controllability for the nonlinear KdV system \eqref{intro-sys-KdV} with initial and final  datum in $H^3(0, L) \cap H^1_0(0, L)$  and  controls in $H^1(0, 1)$ for some critical lengths (\Cref{thm-CP}).  In the appendix, we establish various results used  in  \Cref{sect-0-0,sect-dir,sect-L-KdV}.

\section{Properties of  controls steering \texorpdfstring{$0$ at time $0$ to $0$ at
time $T$}{0 at time 0 to 0 at time T}} \label{sect-0-0}

In this section,  we characterize the controls that steer $0$ to $0$ for the linearized KdV system at a given time. This is done by considering the Fourier transform in the $t$-variable and these conditions are written
in terms of Paley-Wiener's conditions.  The
resolvent of $\partial_x^3 + \partial_x$ hence naturally appears during this analysis. We  begin with the discrete property on the spectrum of this operator.

\begin{lemma}\label{lem-EU} Set $D(\mathcal A) = \Big\{v\in H^3(0,L), v(0) = v(L) = v'(L)
= 0 \Big\}$ and let $\mathcal A$ be the unbounded operator on $L^2(0,L)$ with domain
$D(\mathcal A)$ and defined by $\mathcal A v = v''' +v'$ for $v \in
D(\mathcal A) $. The spectrum of $\mathcal A$ is discrete.
\end{lemma}

\begin{proof}
Since $\mathcal A$ is closed, we only have to prove that there exists a discrete set $\cD
\subset \mC$ such that for $z \in \mC \setminus \cD$ and for $f \in L^2(0, L)$, there
exists a unique solution $v \in H^3(0, L)$ of the system
\begin{equation}
\label{eqvf}
\left\{
\begin{array}{l}
v''' + v' +zv = f \mbox{ in } (0, L),
\\[6pt]
v(0) = v(L) = v'(L) = 0.
\end{array}
\right.
\end{equation}

\step{An auxiliary shooting problem} For  each $z\in \C$, let
$U_{(z)} \in
C^3(\R;\C)$ be the unique solution of the Cauchy problem
\begin{equation}\label{defUz}
U_{(z)}''' + U_{(z)}' + zU_{(z)}=0 \mbox{ in } (0, L), \quad U_{(z)}'(L)=U_{(z)}(L)=0, \quad  U_{(z)}''(L)=1.
\end{equation}
Let $\theta \colon  \C \to \C$ be defined by $\theta(z)=U_{(z)}(0)$. Then
$\theta$ is an entire function. We claim that this function does not vanish
identically, and   $\cD: = \theta^{-1}(0)$ is therefore a discrete set. Indeed,
let us assume that $U_{(1)}(0) = \theta(1) =0$. Multiplying \eqref{defUz} with $z=1$ (the equation of $U_{(1)}$) by
the (real) function $U_{(1)}$ and integrating by parts on $[0,L]$, one gets
\begin{equation}\label{U1byparts}
  \frac{1}{2} U_{(1)}'(0)^2 + \int_0^L U_{(1)}^2 (x) \diff x =0,
\end{equation}
which implies $U_{(1)} = 0$ in $[0, L]$.  This is in contradiction with $U_{(1)}''(L)=1$.

\step{Uniqueness}
Let $z\notin \cD$, i.e., $\theta(z) = U_{(z)} (0) \neq 0$.
Assume that  $v_1,v_2\in H^3(0,L)$ are two solutions of~\eqref{eqvf}. Set $U = v_1-v_2$. Then  $U''' +U' +zU = 0$ and  $U(L) = U'(L) =
0$. It follows that $U = U''(L) U_{(z)}$ in $[0, L]$. So, $U(0) = U''(L) U_{(z)} (0) =
U''(L)\theta(z)$. Since $\theta(z)\neq 0$ and $U(0) = v_1(0) - v_2(0) = 0$, we conclude that $U''(L) =
0$. Hence $U = 0$ in $[0, L]$, which implies the uniqueness.

\step{Existence} Let $z\notin \cD$ and $f \in L^2(0, L)$. Let $V\in H^3(0,L)$ be the
unique solution of the Cauchy problem
\begin{equation}
\label{eqVfCauchy}
\left\{
\begin{array}{l}
V''' + V' + zV = f \text{ in } (0, L), \\[6pt]
V(L) = V'(L) = V''(L) = 0.
\end{array}
\right.
\end{equation}
Set  $v = V-V(0)(\theta(z))^{-1}U_{(z)}$ in $[0, L]$. Then $v$ belongs to $H^3(0,L)$ and
satisfies the differential equation $v'''+v'+zv =f$,  and the boundary conditions $ v(L)
=  0$, $ v'(L) =  0$,  and $v(0) = V(0) - V(0) =0$. Thus $v$ is a solution
of~\eqref{eqvf}.
\end{proof}

Before characterizing controls steering $0$ at time $0$ to $0$ at time $T$, we  introduce

\begin{definition}\label{def:Q-P}
For $z \in \mC$,  let $(\lambda_j)_{1\leq j \leq 3} =  \big(\lambda_j(z) \big)_{1\leq j \leq 3}$ be the three solutions repeated with the multiplicity of
 \begin{equation}\label{eq-lambda}
  \lambda^3 + \lambda + i z = 0.
 \end{equation}
Set
 \begin{equation}\label{eq-defQ}
  Q = Q (z) : = \sum_{j=1}^3 (\lambda_{j+1} - \lambda_j) e^{\lambda_{j} L  + \lambda_{j+1} L
} = 
 \begin{pmatrix}
 1 & 1 & 1\\
 e^{\lambda_1 L } & e^{\lambda_2 L } & e^{\lambda_3 L }\\
 \lambda_1 e^{\lambda_1 L } & \lambda_2 e^{\lambda_2 L } & \lambda_3 e^{\lambda_3 L }
 \end{pmatrix},
\end{equation}
\begin{equation}\label{def-P}
P =  P(z) : =  \sum_{j=1}^3  \lambda_j(e^{\lambda_{j+2} L } - e^{\lambda_{j+1} L})
 = \det
 \begin{pmatrix}
  1&1&1 \\
  e^{\lambda_1L}&e^{\lambda_2L}&e^{\lambda_3L}\\
  \lambda_1& \lambda_2& \lambda_3
 \end{pmatrix},
 \end{equation}
 and
 \begin{equation}\label{def-Xi}
\Xi = \Xi(z) := -  (\lambda_2 - \lambda_1) (\lambda_3 - \lambda_2) (\lambda_1 - \lambda_3) =
\det \begin{pmatrix}1&1&1\\ \lambda_1&\lambda_2&\lambda_3\\
\lambda_1^2&\lambda_2^2&\lambda_3^2\end{pmatrix}, 
\end{equation}
with the convention $\lambda_{j+3} = \lambda_{j}$ for $j \ge 1$. 
\end{definition}

\begin{remark}\label{rk:PG_holomorphic}
The matrix $Q$ and the quantities $P$ and $\Xi$  are antisymmetric  with respect
to $\lambda_j$ ($j=1, 2, 3$), and their definitions depend on a choice of  the order of  $(\lambda_1, \lambda_2, \lambda_3)$. Nevertheless, we later consider a product  of either $P$, $\Xi$,  or $\det Q$ with another antisymmetric function of
$(\lambda_j)$,  or deal with $|\det Q|$,  and these quantities therefore make sense (see e.g. \eqref{def-hy}, \eqref{def-dxhy}). The definitions of $P$, $\Xi$,  and $Q$ are only  understood in these contexts.
\end{remark}

In what follows, for an appropriate function $v$ defined
on $\mR_+ \times (0, L)$, we extend $v$ by $0$ on $\mR_-\times (0,L)$ and we denote by
$\hat v$ its Fourier transform with respect to $t$, i.e., for $z \in \mC$,
\begin{equation*}
\hat v(z, x) = \frac{1}{\sqrt{2 \pi} }\int_0^{+\infty} v(t, x) e^{- i z t} \diff t.
\end{equation*}

We  have

\begin{lemma} \label{lem-form-sol} Let $u \in L^2(0, + \infty)$  and let $y \in C
\big([0,
+ \infty); L^2(0, L) \big) \cap L^2_{\loc}\big( [0, + \infty); H^1(0, L) \big)$ be the
unique solution of
\begin{equation}\label{sys-y}\left\{
\begin{array}{cl}
y_t (t, x) + y_x (t, x) + y_{xxx} (t, x) = 0 &  \mbox{ in } (0, +\infty) \times (0, L),
\\[6pt]
y(t, x=0) = y(t, x=L) = 0 & \mbox{ in }  (0, +\infty), \\[6pt]
y_x(t , x= L) = u(t) & \mbox{ in } (0, +\infty),
\end{array}\right.
\end{equation}
with
\begin{equation}\label{IC-y}
y(t = 0, \cdot) =  0 \mbox{ in } (0, L).
\end{equation}
Then, outside of a discrete set $z\in \mR$, we have
\begin{equation}\label{def-hy}
\hy (z, x) =  \frac{\hu }{\det Q}  \sum_{j=1}^3 (e^{\lambda_{j+2} L } - e^{\lambda_{j+1}
L
}) e^{\lambda_j x}\mbox{ for a.e. } x \in (0, L),
\end{equation}
and in particular,
\begin{equation}\label{def-dxhy}
 \partial_x\hy(z,0)  = \frac{\hu (z) P(z)}{\det Q(z)}.
\end{equation}
\end{lemma}

\begin{remark}\label{rem-form-sol} Assume that $\hu(z, \cdot)$ is well-defined for $z
\in \mC$ (e.g. when $u$ has a compact support). Then the conclusions of
\Cref{lem-form-sol} hold  outside of a discrete set $z\in \mC$.
\end{remark}

\begin{proof} From the system of $y$, we have
\begin{equation}\label{sys-hy}
\left\{\begin{array}{cl}
i z \hy(z, x) + \hy_x (z, x) + \hy_{xxx} (z, x) = 0 &  \mbox{ in } \mR \times  (0, L),
\\[6pt]
\hy(z, x=0) = \hy(z, x=L) = 0 & \mbox{ in }  \mR, \\[6pt]
\hy_x(z, x = L) = \hu(z) & \mbox{ in } \mR.
\end{array}\right.
\end{equation}
Taking into account the equation of $\hy$, we search the solution of the form
\begin{equation*}
\hy(z, \cdot) = \sum_{j=1}^3 a_j e^{\lambda_j x},
\end{equation*}
where $\lambda_j = \lambda_j(z) $ with $j=1, 2, 3$ are defined in \Cref{def:Q-P}.

According to the theory of
ordinary differential equations with constant coefficients, this is possible if the
equation $\lambda^3 + \lambda + iz = 0$ has three distinct solutions, i.e., if the
discriminant $-4+27z^2$ is not $0$. Moreover, if $- iz\notin \Sp(\mathcal A)$, this solution
is unique. Thus, by \Cref{lem-EU}, outside a discrete set in $\mR$, $\hy(z,\cdot)$ can be written in this form
in a unique way.
Using the boundary conditions for $\hy$,  we require that
\begin{equation*}
\left\{\begin{array}{cl}
\sum_{j=1}^3 a_j = 0, \\[6pt]
\sum_{j=1}^3 e^{\lambda_j L } a_j = 0, \\[6pt]
\sum_{j=1}^3 \lambda_j e^{\lambda_j L } a_j = \hu.
\end{array}\right.
\end{equation*}
This implies, with $Q = Q(z)$ defined in \Cref{def:Q-P},
\begin{equation}\label{eq-systQ}
Q (a_1, a_2, a_3)\tr = (0, 0, \hu)\tr.
\end{equation}
It follows that
\begin{equation*}
a_j = \frac{\hu }{\det Q} \big(e^{\lambda_{j+2} L } - e^{\lambda_{j+1} L } \big).
\end{equation*}
This yields
\begin{equation}\label{form-hy}
\hy(z, x) =  \frac{\hu }{\det Q}  \sum_{j=1}^3 (e^{\lambda_{j+2} L } -
e^{\lambda_{j+1} L
}) e^{\lambda_j x}.
\end{equation}
We thus obtain
\begin{equation}
\partial_x \hy(z, 0) =  \frac{\hu (z) P(z) }{\det Q (z)}.\qedhere
\end{equation}
\end{proof}

As mentioned in \Cref{rk:PG_holomorphic}, the maps $P$ and $\det Q$ are antisymmetric functions with respect to $\lambda_j$. It is hence convenient to consider $\partial_x
\hy(z, 0)$ under the form
\begin{equation}\label{eq:hyx_meromorphic}
\partial_x \hy(z, 0) = \frac{\hu (z) G(z) }{H(z)},
\end{equation}
where, with $\Xi$ defined in \eqref{def-Xi},
\begin{equation}\label{def-GH}
G(z) =  P(z)/\Xi(z) \quad \mbox{ and } \quad H(z) = \det Q(z)/\Xi(z).
\end{equation}

Concerning the functions  $G$ and $H$, we have

\begin{lemma}\label{lem-hol} The functions $G$ and $H$ defined in \eqref{def-GH}  are entire functions.
\end{lemma}

\begin{proof} Note that the maps $z\mapsto \Xi(z)
P(z)$, $z\mapsto \Xi(z) \det Q(z)$ and $z\mapsto \Xi(z)^2$ are symmetric functions of the
$\lambda_j$ and are thus well-defined, and even entire functions (see \Cref{pro-S} in
\Cref{sec:symmetric_holomorphic}). According to the definition of $\Xi$, $\Xi(z_0) = 0$ if and only if $X^3 + X +i z_0$ has a
double root, i.e.  $z_0 = \pm 2/ (3\sqrt 3)$. Simple computations prove that when
$\epsilon$ is small,
\begin{equation}\label{lem-hol-lambda}
\left\{\begin{aligned}
 \lambda_1(z_0+ \eps) &=
 \mp \frac{i}{\sqrt 3} + \frac{\sqrt{\mp i}}{3^{1/4}} \sqrt \epsilon + O(\eps),  \\[6pt]
 \lambda_2(z_0 + \eps) &=
 \mp \frac{i}{\sqrt 3}  - \frac{\sqrt{\mp i}}{3^{1/4}} \sqrt \eps + O(\eps), \\[6pt]
 \lambda_3(z_0+\eps) &=  \pm \frac{2i}{\sqrt 3} + \frac{ \eps}{3} + O(\eps^2).
 \end{aligned}\right.
\end{equation}
Indeed, the behavior of $\lambda_3$ follows immediately from the expansion of $\lambda_3$ near $ \pm \frac{2i}{\sqrt 3} $. The behavior of $\lambda_1$ and $\lambda_2$ can be then verified using, with $\Delta = - 3 \lambda_3^2 - 4$,
\begin{equation*}
\lambda_1 = \frac{-\lambda_3 + \sqrt{\Delta}}{2} \quad \mbox{ and } \quad \lambda_2 = \frac{-\lambda_3 - \sqrt{\Delta}}{2}.
\end{equation*}
It follows that  that $\Xi^2(z_0+\eps) =
c_\pm \eps + O(\eps^2)$ for some $c_\pm \neq 0$. This in turn implies that
$z_0 = \pm 2/(3\sqrt 3)$ are simple zeros of $\Xi^2$. When $X^3+X+iz$ has a double
root, the definitions of $P$ and $\det Q$ (Eq.~\eqref{eq-defQ} and~\eqref{def-P}) imply
\[
|P(z_0)| = \lvert\det Q (z_0)\rvert = 0 \mbox{ for } z_0 = \pm 2/ (3\sqrt 3).
\]
The conclusion follows.
\end{proof}

\begin{remark} \label{rem-detQ-realroots} It is interesting to note  that
\begin{enumerate}
\item  \big($H(z) = 0$ and $z \neq  \pm 2 / (3 \sqrt{3}) \big)$ if and only if $- iz\in \Sp(\mathcal A)$.

\item  $ i z \in \Sp(\mathcal A)$ and $z$ is real  if and only if $L = 2 \pi \sqrt{\frac{k^2 + k l + l^2}{3}}$, and
\begin{equation}\label{coucou}
z =  \frac{(2k + l)(k-l)(2 l + k)}{3 \sqrt{3}(k^2 + kl + l^2)^{3/2}},
\end{equation}
for some $k, l \in \N$ with $1 \le l \le k$.
\end{enumerate}
Indeed, if $L = 2 \pi \sqrt{\frac{k^2 + k l + l^2}{3}}$ and $z$ is given by the RHS of
\eqref{coucou}, then, from \cite{Rosier97}, $i z \in \Sp(\mathcal A)$. On the other hand,
if $z$ is real and $i z \in \Sp(\mathcal A)$, then, by an integration by parts, the
corresponding eigenfunction $w$ also satisfies the condition $w_x(0) = 0$. It follows
from \cite{Rosier97} that  $L = 2 \pi \sqrt{\frac{k^2 + k l + l^2}{3}}$ and  $z$ is given
by \eqref{coucou} for some $k, l \in \N$ with $1 \le l \le k$.  We finally note that for
$z \neq \pm 2/ (3\sqrt 3)$,
the solutions of the ordinary differential equation $u''' +u' +iz u = 0$
are of the form $u(x) = \sum_{j=1}^3 a_j e^{\lambda_j x}$. This implies that $Q (a_1, a_2, a_3)\tr = (0, 0, 0)\tr$ if $u(0) = u(L) = u'(L) = 0$. Therefore, for $z\neq \pm 2/ (3\sqrt 3)$, $- i z$ is an eigenvalue of
$\mathcal A$ if and only if  $\lvert\det Q(z)\rvert = 0$, i.e., $H (z)= 0$. We finally
note that, $\pm 2i / (3\sqrt 3)$ is not a pure imaginary eigenvalue of ${\mathcal A}$
since, for $k \ge l \ge 1$,
\[
0 \le \frac{(2k + l)(k-l)(2 l + k)}{3 \sqrt{3}(k^2 + kl + l^2)^{3/2}} = \frac{(2k + l)(k^2  + k l - 2 l^2)}{3 \sqrt{3}(k^2 + kl + l^2)^{3/2}} <  \frac{(2k + l)}{3 \sqrt{3}(k^2 + kl + l^2)^{1/2}} < \frac{2}{3 \sqrt{3}}.
\]
\end{remark}

We are ready to  give the characterization of the controls
steering $0$ to $0$, which is the starting point of our analysis.

\begin{proposition} \label{pro-Gen} Let $L>0$, $T>0$, and $u \in L^2(0, + \infty)$. Assume that   $u$ has a compact support in $[0, T]$,  and $u$  steers $0$ at the time $0$ to $0$ at the time $T$, i.e., the unique solution $y$ of \eqref{sys-y} and \eqref{IC-y} satisfies $y(T, \cdot) = 0$ in $(0, L)$.  Then $\hu$ and  $\hu G/ H$  satisfy the assumptions of Paley-Wiener's theorem concerning the support in $[-T, T]$, i.e.,
\[
\hu \mbox{ and } \hu G/H \mbox{ are entire functions},
\]
and
\[
|\hu(z)|  + \left|\frac{\hu G(z)}{H(z)} \right| \le C e^{T| \Im(z)|},
\]
for some positive constant $C$.
\end{proposition}

Here and in what follows, for a complex number $z$, $\Re(z)$, $\Im(z)$, and $\bar z$ denote the real part, the imaginary part, and the conjugate of $z$, respectively.

\begin{proof}\Cref{pro-Gen} is a consequence of \Cref{lem-form-sol} and Paley-Wiener's theorem, see e.g. \cite[19.3 Theorem]{Rudin-RC}. The proof is clear from the analysis above in this section and left to the reader.
\end{proof}

\section{Attainable directions for small time}\label{sect-dir}

In this section, we investigate controls which steer the linear KdV equation from $0$ to
$0$ in some time $T$, and a quantity related to  the quadratic order in the power expansion of the nonlinear
KdV equation behaves. Let  $u \in L^2(0,  + \infty)$ and denote $y$ the corresponding solution
of the linear KdV equation~\eqref{sys-y}. We assume the initial condition to be $0$ and
that $y$ satisfies $y(t, \cdot) = 0$ in $(0, L)$ for $t \ge T$. We have, by
\Cref{lem-form-sol} (and also \Cref{rem-form-sol}), for $z \in \mC$ outside  a discrete
set,
\begin{equation}\label{def-y}
\hat y(z, x) = \hat u(z) \frac{\sum_{j=1}^3(e^{\lambda_{j+1} L } - e^{\lambda_{j} L })e^{\lambda_{j+2} x} }{\sum_{j=1}^3 (\lambda_{j+1} - \lambda_{j}) e^{-\lambda_{j+2} L }}.
\end{equation}
Recall that  $\lambda_j = \lambda_j(z)$ for $j=1, \, 2, \, 3$  are the three solutions of
the equation
\begin{equation}\label{eq-lambda-j}
x^3 + x = -i z \mbox{ for } z \in \mC.
\end{equation}
Let $\eta_1, \, \eta_2, \,  \eta_3 \in i \mR$, i.e.,  $\eta_j \in \mC$ with $\Re(\eta_j) = 0$ for $j=1, \, 2, \, 3$. Define
\begin{equation}\label{def-varphi}
\varphi(x) = \sum_{j=1}^3 (\eta_{j+1} - \eta_{j}) e^{\eta_{j+2} x} \mbox{ for } x \in [0, L],
\end{equation}
with the convention $\eta_{j+3} = \eta_j$ for $j \ge 1$. The following assumption on $\eta_j$ is used repeatedly  throughout the paper:
\begin{equation}\label{pro-eta}
e^{\eta_1  L} = e^{\eta_2 L} = e^{\eta_3 L},
\end{equation}
which is equivalent to $\eta_3 - \eta_2, \eta_2 - \eta_1 \in \dsp  \frac{2 \pi i}{L} \mZ$. The definition of $\varphi$ in \eqref{def-varphi} and  the assumption on $\eta_j$  in \eqref{pro-eta} are motivated by the structure of $\M$ \cite{Cerpa07, CC09} and will be clear in \Cref{sect-NL-KdV}.

\medskip 

We have 


\begin{lemma}\label{lem-1} Let   $p \in \mR$ and let $\varphi$ be defined by \eqref{def-varphi}.  Set, for $(z, x) \in  \mC \times [0, L]$,
\begin{equation}\label{def-B}
B(z, x) = \frac{\sum_{j=1}^3(e^{\lambda_{j+1} L } - e^{\lambda_{j} L })e^{\lambda_{j+2} x} }{\sum_{j=1}^3 (\lambda_{j+1} - \lambda_{j}) e^{-\lambda_{j+2} L }} \cdot  \frac{\sum_{j=1}^3(e^{\tlambda_{j+1} L } - e^{\tlambda_{j} L })e^{\tlambda_{j+2} x} }{\sum_{j=1}^3 (\tlambda_{j+1} - \tlambda_{j}) e^{-\tlambda_{j+2} L }}  \cdot \varphi_x(x),
\end{equation}
where $\tlambda_j = \tlambda_j (z)$ ($j=1, \, 2, \, 3$) denotes the conjugate of the roots of \eqref{eq-lambda-j}  with $z$ replaced by $z - p$ and with the use of convention $\tlambda_{j+3} = \tlambda_j$ for $j \ge 1$.
Let $u \in L^2(0, + \infty)$ and let $y \in C \big([0, + \infty); L^2(0, L) \big) \cap L^2_{\loc}\big( [0, + \infty); H^1(0, L) \big)$ be the unique solution of \eqref{sys-y} and \eqref{IC-y}.  Then
\begin{equation}\label{identity-1}
\int_0^L \int_{0}^{+ \infty} |y(t, x)|^2 \varphi_x(x) e^{- i p t} \diff t \diff x  =
\int_{\mR}  \hu(z) \overline{\hu(z - p )} \int_0^L B(z, x) \diff x \diff z .
\end{equation}
\end{lemma}

\begin{remark} \rm The LHS of \eqref{identity-1} is a multiple of the  $L^2(0, L)$-projection of the solution $y(T, \cdot)$ into the space spanned by the conjugate of the vector $\varphi(x) e^{-i p T}$ whose  real and imaginary parts are in $\M$ for appropriate choices of $\eta_j$ and $p$ when the initial data is orthogonal to $\M$ (see~\cite{CC04, Cerpa07, CC09},  and also \eqref{thm-NL-B}).  
\end{remark}

\begin{proof}  We have
\begin{align*}
\int_0^L \int_{0}^{\infty} |y(t, x)|^2 \varphi_x(x) e^{ - i p t} \diff t \diff x = &
\sqrt{2 \pi} \int_0^L  \varphi_x(x)  \widehat{|y|^2}(p, x) \diff  x=   \int_0^L
\varphi_x(x)  \hat y * \widehat{ \bar{y}} (p, x) \diff x \\[6pt]
= &  \int_0^L \varphi_x(x) \int_{\mR}\hat y (z, x)  \widehat{ \bar{y}} (p - z, x) \diff z
\diff x \\[6pt]
= &   \int_0^L  \varphi_x(x)\int_{\mR}\hat y (z, x)  \overline{\hat y} (z -   p, x) \diff
z  \diff x.
\end{align*}
Using Fubini's theorem, we derive from \eqref{def-y} that
\begin{equation*}
\int_0^L \int_{0}^{\infty} |y(t, x)|^2 \varphi_x(x) e^{- i p t} \diff t \diff x  =
\int_{\mR}  \hu(z) \overline{\hu(z - p )} \int_0^L B(z, x) \diff x \diff z ,
\end{equation*}
which is \eqref{identity-1}.
\end{proof}

We next state  the behaviors of $\lambda_j$ and $ \tlambda_j$ given in \Cref{lem-1} for
``large positive" $z$, which will be used repeatedly in this section and \Cref{sect-L-KdV}.
These asymptotics are direct consequence of the equation~\eqref{eq-lambda} satisfied by
the $\lambda_j$.

\begin{lemma}\label{th:lambda_asym}
For  $p \in \mR$ and $z$ in a small enough conic neighborhood of $\R_+$, let $\lambda_j$ and $\tlambda_j$ with $j=1, \, 2, \, 3$
be given in  \Cref{lem-1}.
Consider the convention $\Re(\lambda_1) < \Re(\lambda_2) < \Re(\lambda_3)$ and similarly
for $\tlambda_j$. We have,  in the
limit $|z|\to \infty$,
\begin{equation}
\lambda_j = \mu_jz^{1/3} - \frac1{3\mu_j}z^{-1/3} + O(z^{-2/3}) \quad \text{ with } \mu_j
= e^{-i\pi/6-2ji\pi/3},
\end{equation}
\begin{equation}
\tlambda_j = \tmu_j z^{1/3} - \frac1{3\tmu_j} z^{-1/3} +O(z^{-2/3})  \quad  \text{ with }
\tmu_j = e^{i\pi/6+2ij\pi/3}
\end{equation}
(see Figure~\ref{fig:mu} for the geometry of $\mu_j$ and $\tmu_j$). Here $z^{1/3}$ denotes the cube root of $z$ with the real part positive.  
\end{lemma}


\begin{figure}[htp]
 \begin{minipage}[c]{0.5\textwidth}
  {\centering
  \begin{tikzpicture}[scale=0.8]
  \draw[->] (-4,0) -- (4,0);
  \draw[->] (0,-4) -- (0,4);
  
  \draw (0,0) circle[radius=3];
  
  \foreach \j in {1,2,3} {
   \fill (-30-120*\j:3) circle[fill, radius=0.15] 
     ++(0.15,0) node[right, fill=white, fill opacity=0.5, text opacity = 1]{$\mu_\j$};
   \fill (30+120*\j:3) circle[fill, radius=0.15] 
     ++(0.15,0) node[right, fill=white, fill opacity=0.5, text opacity = 1]{$\tmu_\j$}; 
         }
\end{tikzpicture}}
 \end{minipage}\hfill%
 \begin{minipage}[c]{0.5\textwidth}
 \caption{The roots $\lambda_j$ of $\lambda^3 + \lambda + i z = 0$ satisfy, when
$z>0$ is large, $\lambda_j \sim \mu_j z^{1/3}$ where $\mu_j^3 = -i$. When
$z<0$ and $|z|$ is large, then the corresponding roots $\hat \lambda_j$ satisfy
$\hat \lambda_j \sim \tmu_j |z|^{1/3}$ with $\tmu_j = \overline{\mu_j}$. We also have $\tlambda_j \sim \hat \lambda_j$.}
\label{fig:mu}
 \end{minipage}
\end{figure}
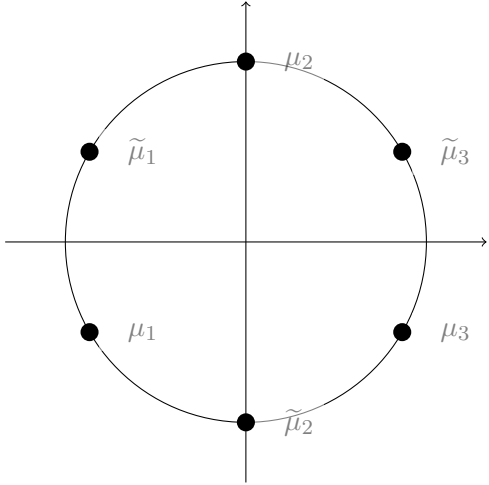

We are ready to establish the behavior of
\[
\int_0^L B(z, x) \diff x
\]
for $z \in \mR$ with  large $|z|$, which is one of the main ingredients for the analysis in this section.

\begin{lemma} \label{lem-B}  Let   $p \in \mR$,  and let $\varphi$ be defined by \eqref{def-varphi}.  Assume that \eqref{pro-eta} holds and  $\eta_j \neq 0$ for $j=1, \,  2, \,  3$.   Let $B$ be defined by \eqref{def-B}.  We have
\begin{equation}
 \int_0^L B(z, x) \diff x  =
\frac{E}{|z|^{4/3}}    + O(|z|^{- 5/3}) \mbox{ for $z\in \mR$ with large $|z|$},
\end{equation}
where $E$ is  defined by
\begin{equation}\label{def-D}
E = \frac{1}{3} (e^{\eta_1 L} -1) \left(- \frac{2}{3}  \sum_{j=1}^3 \eta_{j+2}^2(\eta_{j+1} -\eta_{j})  - i p      \sum_{j=1}^3 \frac{\eta_{j+1}-\eta_{j}} { \eta_{j+2}} \right).
\end{equation}
\end{lemma}

\begin{proof} We first deal with the case where $z$ is positive and large. We use the convention in \Cref{th:lambda_asym} for $\lambda_j$ and $\tlambda_j$.
Consider the denumerator of $B(z, x)$. We have, by \Cref{th:lambda_asym},
\begin{multline}\label{den-B}
\frac{1}{\sum_{j=1}^3 (\lambda_{j+1} - \lambda_{j}) e^{-\lambda_{j+2} L }} \cdot
\frac{1}{ \sum_{j=1}^3 (\tlambda_{j+1} - \tlambda_{j}) e^{-\tlambda_{j+2} L }} \\[6pt]
= \frac{e^{\lambda_1 L } e^{\tlambda_1 L}}{(\lambda_3 - \lambda_2)  (\tlambda_3 -
\tlambda_2)  } \Big( 1 + O \big(e^{-C |z|^{1/3}} \big) \Big).
\end{multline}

We next deal with the numerator of $B(z, x)$. Set, for $(z, x) \in \mR \times (0, L)$,
\begin{equation}\label{def-fg}
f(z, x) = \sum_{j=1}^3(e^{\lambda_{j+1} L } - e^{\lambda_{j} L })e^{\lambda_{j+2} x},
\quad g (z, x) = \sum_{j=1}^3(e^{\tlambda_{j+1} L } - e^{\tlambda_{j} L
})e^{\tlambda_{j+2} x},
\end{equation}
\footnote{The index $m$ stands the
main part.}
\[
f_m(z, x) =  - e^{\lambda_3 L} e^{\lambda_2 x} + e^{\lambda_2 L}  e^{\lambda_3 x} +
e^{\lambda_3 L } e^{\lambda_1 x}, \quad
g_m (z, x) =  - e^{\tlambda_3 L} e^{\tlambda_2 x} + e^{\tlambda_2 L}  e^{\tlambda_3 x}
+ e^{\tlambda_3 L } e^{\tlambda_1 x}.
\]
We have
\begin{multline*}
 \int_0^L f(z, x) g(z, x) \varphi_x(x) \diff x
 = \int_0^L f_m(z, x) g_m(z, x) \varphi_x(x) \diff x + \int_0^L (f- f_m)(z, x) g_m(z, x)
\varphi_x(x) \diff x \\[6pt]  + \int_0^L f_m(z, x) (g -g_m) (z, x) \varphi_x(x) \diff x +
\int_0^L (f- f_m)(z, x) (g -g_m) (z, x) \varphi_x(x) \diff x.
\end{multline*}
It is clear from \Cref{th:lambda_asym} that
\begin{multline}\label{B-p0}
 \int_0^L |(f- f_m)(z, x) g_m(z, x) \varphi_x(x)| \diff x + \int_0^L |(f- f_m)(z, x) (g
-g_m) (z, x) \varphi_x(x)| \diff x  \\[6pt]  + \int_0^L |f_m(z, x) (g -g_m) (z, x)
\varphi_x(x)| \diff x \le C |e^{(\lambda_3 + \tlambda_3) L }| e^{- C |z|^{1/3}}.
\end{multline}

We next estimate
\begin{equation}\label{eq:int_fm_gm}
\int_0^L f_m(x, z) g_m(x, z) \varphi_x(x) = \int_0^L f_m(x, z) g_m(x, z) \left(
\sum_{j=1}^3 \eta_{j+2} (\eta_{j+1} - \eta_{j}) e^{\eta_{j+2} x} \right) \diff x.
\end{equation}
We first have, by  \eqref{pro-eta} and \Cref{th:lambda_asym},
\begin{multline}\label{B-p1}
\int_0^L \Big(  - e^{\lambda_3 L} e^{\lambda_2 x}  e^{\tlambda_2 L}  e^{\tlambda_3 x}
- e^{\lambda_2 L}  e^{\lambda_3 x}e^{\tlambda_3 L} e^{\tlambda_2 x}
+ e^{\lambda_2 L}  e^{\lambda_3 x}e^{\tlambda_2 L}  e^{\tlambda_3 x} \Big) \\[6pt]
\times \left( \sum_{j=1}^3 \eta_{j+2} (\eta_{j+1} - \eta_{j}) e^{\eta_{j+2} x} \right)
\diff x
=  e^{ (\lambda_3 +   \tlambda_3 + \lambda_2 +  \tlambda_2) L } \Big(e^{\eta_1 L
}T_1(z) + O\big(e^{-C|z|^{1/3}} \big) \Big),
\end{multline}
where
\begin{equation}\label{def-T1}
T_1(z) : =  \sum_{j=1}^3 \eta_{j+2}(\eta_{j+1}-\eta_{j}) \left( \frac{1}{  \lambda_3
+   \tlambda_3   + \eta_{j+2}} - \frac{1}{ \lambda_3  +   \tlambda_2   + \eta_{j+2}} -
\frac{1}{ \lambda_2  + \tlambda_3  + \eta_{j+2}}  \right).
\end{equation}

Let us now deal with the terms of~\eqref{eq:int_fm_gm} that contain both $e^{\lambda_3 L + \tlambda_3 L}$ and (either $e^{\lambda_1 x}$ or $e^{\tlambda_1 x}$).  We obtain, by  \eqref{pro-eta} and \Cref{th:lambda_asym},
\begin{multline}\label{B-p2}
\int_0^L \Big( e^{\lambda_3 L } e^{\lambda_1 x}e^{\tlambda_3 L } e^{\tlambda_1 x}
- e^{\lambda_3 L } e^{\lambda_1 x} e^{\tlambda_3 L} e^{\tlambda_2 x} - e^{\lambda_3 L}
e^{\lambda_2 x} e^{\tlambda_3 L } e^{\tlambda_1 x} \Big) \\[6pt]
\times \left( \sum_{j=1}^3 \eta_{j+2} (\eta_{j+1} - \eta_{j}) e^{\eta_{j+2} x} \right)
\diff x
=  e^{(\lambda_3 + \tlambda_3) L }\Big( T_2(z) + O(e^{-C|z|^{1/3}}) \Big),
\end{multline}
where
\begin{equation}\label{def-T2}
T_2 (z): = \sum_{j=1}^3 \eta_{j+2}(\eta_{j+1}-\eta_{j}) \left( - \frac{1}{\lambda_1
+  \tlambda_1  + \eta_{j+2}} + \frac{1}{\lambda_1   +  \tlambda_2   + \eta_{j+2}} +
\frac{1}{ \lambda_2  +   \tlambda_1   + \eta_{j+2}}\right).
\end{equation}
We  have, by \eqref{pro-eta},
\begin{equation}\label{B-p3}
\int_0^L e^{\lambda_3 L} e^{\lambda_2 x} e^{\tlambda_3 L} e^{\tlambda_2 x}  \left(
\sum_{j=1}^3 \eta_{j+2} (\eta_{j+1} - \eta_{j}) e^{\eta_{j+2} x} \right) \diff x
=  e^{(\lambda_3 + \tlambda_3) L } T_3(z),
\end{equation}
where
\begin{equation}\label{def-T3}
T_3 (z): = \Big( e^{\lambda_2 L  + \tlambda_2 L  + \eta_{1} L } - 1\Big) \sum_{j=1}^3
\frac{\eta_{j+2}(\eta_{j+1}-\eta_{j})  }{\lambda_2   + \tlambda_2   + \eta_{j+2}}.
\end{equation}
The other terms of~\eqref{eq:int_fm_gm} are negligible, because we have
\begin{multline}\label{B-p3-1}
\left| \int_0^L \Big(  e^{\lambda_3 L } e^{\lambda_1 x} e^{\tlambda_2 L} e^{\tlambda_3
x} + e^{\lambda_2 L} e^{\lambda_3 x} e^{\tlambda_3 L} e^{\tlambda_1 x}  \Big) \Big(
\sum_{j=1}^3 \eta_{j+2} (\eta_{j+1} - \eta_{j}) e^{\eta_{j+2} x} \Big) \diff x \right|
\\[6pt]
= |e^{(\lambda_3 + \tlambda_3)L} | O(e^{-C z^{1/3}}).
\end{multline}

Using \Cref{th:lambda_asym}, we have
\begin{equation}\label{B-p3-2}
\left\{\begin{array}{c}
\lambda_1 + \tlambda_1 + \lambda_2 + \tlambda_2 + \lambda_3 + \tlambda_3 =
O(z^{-1/3}),  \\[6pt]
\lambda_1 + \tlambda_1 + \lambda_3 + \tlambda_3 = O(z^{-1/3}), \\[6pt]
(\lambda_3 - \lambda_2)(\tlambda_3 - \tlambda_2) = 3 z^{2/3} ( 1 + O(z^{-1/3}) ).
\end{array}\right.
\end{equation}

We claim that
\begin{equation}\label{claim-T}
|T_1(z)| + |T_2(z)| + |T_3(z)| = O(z^{-2/3}) \mbox{ for large positive $z$}.
\end{equation}

Assuming \eqref{claim-T}, and combining \eqref{den-B}, \eqref{B-p1}, \eqref{B-p2},
\eqref{B-p3}, \eqref{B-p3-1}, and \eqref{B-p3-2} yields
\begin{equation}\label{B-p4}
\int_{0}^L B(z, x) \diff z  \\[6pt]= \frac{1}{3 |z|^{2/3}} \Big( e^{\eta_1 L }T_1(z) +
T_2 (z) + T_3 (z) + O(z^{-1}) \Big).
\end{equation}

We next derive the asymptotic behaviors of $T_1(z)$, $T_2(z)$, and $T_3(z)$, which in
particular imply \eqref{claim-T}. We first deal with $T_1(z)$ given in
\eqref{def-T1}. Since
\begin{equation}\label{annulation}
\sum_{j=1}^3 \eta_{j+2}(\eta_{j+1}-\eta_{j}) = 0,
\end{equation}
we obtain
\begin{align*}
T_1 (z)=  & \sum_{j=1}^3 \eta_{j+2}(\eta_{j+1}-\eta_{j}) \left( \frac{1}{\lambda_3 +
\tlambda_3 + \eta_{j+2}} -  \frac{1}{\lambda_3 + \tlambda_3 } \right) \\[6pt]
& +
 \sum_{j=1}^3 \eta_{j+2}(\eta_{j+1}-\eta_{j})  \left(  - \frac{1}{\lambda_3 +
\tlambda_2 + \eta_{j+2}}
 +  \frac{1}{\lambda_3 + \tlambda_2}  \right) \\[6pt]
&  + \sum_{j=1}^3 \eta_{j+2}(\eta_{j+1}-\eta_{j})  \left(- \frac{1}{\lambda_2 +
\tlambda_3 + \eta_{j+2}}  + \frac{1}{\lambda_2 + \tlambda_3}  \right).
\end{align*}
Using \Cref{th:lambda_asym},
 we get
\begin{equation*}
T_1(z) =  - \sum_{j=1}^3 \eta_{j+2}^2(\eta_{j+1} -\eta_{j}) \left( \frac{1}{(\lambda_3
+ \tlambda_3)^2} - \frac{1}{(\lambda_3 + \tlambda_2)^2} - \frac{1}{(\lambda_2 +
\tlambda_3)^2} \right) +  O(z^{-1}).
\end{equation*}
Moreover, we derive from \Cref{th:lambda_asym} that
\begin{align}
 \frac{1}{(\lambda_3 + \tlambda_3)^2} - \frac{1}{(\lambda_3 + \tlambda_2)^2} -
\frac{1}{(\lambda_2 + \tlambda_3)^2}
 &= z^{-2/3} \Big((\mu_3+\tmu_3)^{-2}-(\mu_3+\tmu_2)^{-2} -(\mu_2+\tmu_3)^{-2}
\Big) + O(z^{-1})\nonumber\\[6pt]
&= z^{-2/3} \left( \frac13 - \frac{-1+i\sqrt 3}6 - \frac{-1-i\sqrt
3}6\right) + O(z^{-1})\nonumber\\[6pt]
&= \frac23z^{-2/3} + O(z^{-1}). \label{eq:T1_asym}
\end{align}
We derive that
\begin{equation}\label{T1-final}
T_1(z)  = - \frac{2}{3} z^{-2/3}  \sum_{j=1}^3 \eta_{j+2}^2(\eta_{j+1} -\eta_{j}) +
O(z^{-1}).
\end{equation}

We next consider  $T_2(z)$ given in \eqref{def-T2}. We have, by \eqref{annulation},
\begin{align*}
T_2 (z)=  & \sum_{j=1}^3 \eta_{j+2}(\eta_{j+1}-\eta_{j}) \left( -  \frac{1}{\lambda_1
+ \tlambda_1 + \eta_{j+2}} +  \frac{1}{\lambda_1 + \tlambda_1 } \right) \\[6pt]
& +
 \sum_{j=1}^3 \eta_{j+2}(\eta_{j+1}-\eta_{j})  \left(   \frac{1}{\lambda_1 +
\tlambda_2 + \eta_{j+2}} -
   \frac{1}{\lambda_1 + \tlambda_2}  \right) \\[6pt]
&  + \sum_{j=1}^3 \eta_{j+2}(\eta_{j+1}-\eta_{j})  \left( \frac{1}{\lambda_2 +
\tlambda_1 + \eta_{j+2}}  - \frac{1}{\lambda_2 + \tlambda_1}  \right).
\end{align*}
Using \Cref{th:lambda_asym}, we obtain
\begin{equation*}
T_2 (z)=   \sum_{j=1}^3 \eta_{j+2}^2(\eta_{j+1} -\eta_{j}) \left( \frac{1}{(\lambda_1
+ \tlambda_1)^2} - \frac{1}{(\lambda_1 + \tlambda_2)^2} - \frac{1}{(\lambda_2 +
\tlambda_1)^2} \right) +  O(z^{-1}),
\end{equation*}
and
\begin{equation*}
 \frac{1}{(\lambda_1
+ \tlambda_1)^2} - \frac{1}{(\lambda_1 + \tlambda_2)^2} - \frac{1}{(\lambda_2 +
\tlambda_1)^2} = z^{-2/3} \Big((\mu_1+\tmu_1)^{-2}-(\mu_1+\tmu_2)^{-2} -(\mu_2+\tmu_1)^{-2}
\Big) + O(z^{-1}).
\end{equation*}
By \Cref{th:lambda_asym}, we have
\begin{align*}
(\mu_1 + \tmu_1)^2 &= (\mu_3 + \tmu_3)^2&
(\mu_1 + \tmu_2)^2 &= (\tmu_3 + \mu_2)^2&
(\tmu_1 + \mu_2)^2 &= (\mu_3 + \tmu_2)^2.
\end{align*}
Combining this with \eqref{eq:T1_asym}, we then have
\begin{equation}\label{T2-final}
T_2 (z) =  \frac{2}{3} z^{-2/3} \sum_{j=1}^3 \eta_{j+2}^2(\eta_{j+1} -\eta_{j})  +
O(z^{-1}).
\end{equation}

We finally  consider  $T_3(z)$ given in \eqref{def-T3}. We have, by \eqref{eq-lambda},
\[
\lambda_2^3 + \tlambda_2^3 + \lambda_2 + \tlambda_2 = - i z + i (z -p) =
- i p.
\]
This yields
\[
\lambda_2 + \tlambda_2 = - \frac{i p}{\lambda_2^2 + \tlambda_2^2 + \lambda_2
\tlambda_2}.
\]
From \Cref{th:lambda_asym}, we have
\[
\lambda_2 + \tlambda_2 = i p  z^{-2/3}  + O(z^{-1}).
\]
It follows that
\begin{align}\label{T3-p2}
\sum_{j=1}^3 \frac{\eta_{j+2}(\eta_{j+1}-\eta_{j})  }{\lambda_2 + \tlambda_2  +
\eta_{j+2}}  &=  \sum_{j=1}^3 \frac{\eta_{j+2}(\eta_{j+1}-\eta_{j})  }{i p
z^{-2/3} + \eta_{j+2}}  + O(|z|^{-1}) \nonumber \\
&=   \sum_{j=1}^3 (\eta_{j+1}-\eta_{j})  \left( 1 -  \frac{ip z^{-2/3}}{ \eta_{j+2} }
\right) + O(|z|^{-1})\nonumber\\
&=   - i p   \sum_{j=1}^3 \frac{  \eta_{j+1}-\eta_{j}} {
\eta_{j+2} } z^{-2/3} +  O(z^{-1}).
\end{align}
We derive from  \eqref{T3-p2} and \Cref{th:lambda_asym} that
\begin{equation}\label{T3-final}
T_3 =  - i p   \Big( e^{ \eta_1 L } - 1 \Big) \sum_{j=1}^3 \frac{\eta_{j+1}-\eta_{j}}
{  \eta_{j+2} } z^{-2/3}   + O (z^{-1}).
\end{equation}

Using \eqref{T1-final}, \eqref{T2-final}, and \eqref{T3-final}, we derive from
\eqref{B-p4} that
\begin{equation*}
\int_{0}^L B(z, x) \diff x  = E z^{-4/3}   + O(z^{-5/3}),
\end{equation*}
which is the conclusion for large positive $z$.

\medskip
The conclusion in the case where $z$ is large and negative can be derived from the
case where $z$ is positive and large as follows.
Define, for $(z, x) \in  \mR \times (0, L)$, with large $|z|$,
\[
M(z, x) =  \frac{\sum_{j=1}^3(e^{\lambda_{j+1} L } - e^{\lambda_{j} L
})e^{\lambda_{j+2} x} }{\sum_{j=1}^3 (\lambda_{j+1} - \lambda_{j}) e^{-\lambda_{j+2} L
}}.
\]
Then
\[
B(z, x) = M(z, x) \overline{M (z-p, x)} \varphi_x(x).
\]
It is clear from the definition of $M$ that
\[
M(-z, x) = \overline{M(z, x)}.
\]
We then have
\[
B(-z, x) =M(-z, x) \overline{M (- z-p, x)} \varphi_x(x) =  \overline{ M(z, x)
\overline{M (z+ p, x)}  \, \overline{\varphi_x(x)}} .
\]
We thus obtain the result in the case where $z$ is negative and large  by taking the
conjugate of the corresponding expression for  large positive $z$ in which $\eta_j$
and $p$ are replaced by $-\eta_j$ and $-p$. The conclusion follows.
\end{proof}

As a consequence of  \Cref{lem-1,lem-B}, we obtain

\begin{lemma}\label{lem-dir} Let   $p \in \mR$ and let $\varphi$ be defined by
\eqref{def-varphi}.  Assume that \eqref{pro-eta} holds and  $\eta_j \neq 0$ for $j=1,
2, 3$.     Let $u \in L^2(0, + \infty)$ and let $y \in C([0, + \infty); L^2(0, L))
\cap L^2_{\loc}\big([0, +\infty); H^1(0, L) \big)$ be the unique solution of
\eqref{sys-y} and \eqref{IC-y}.  We have
\begin{equation}
\int_{0}^{+\infty} \int_0^L  |y(t, x)|^2 \varphi_x(x) e^{- i p t} \diff x \diff t  =
\int_{\mR}  \hu(z) \overline{\hu(z - p )} \Big(\frac{E}{|z|^{4/3}}  + O(|z|^{-5/3})
\Big) \diff z.
\end{equation}
\end{lemma}

Using \Cref{lem-dir}, we will establish the following result which is the key
ingredient for the analysis of the non-null-controllability for small time of the KdV system \eqref{intro-sys-KdV}.

\begin{proposition} \label{pro-monotone} Let   $p \in \mR$ and let $\varphi$ be
defined by \eqref{def-varphi}.  Assume that \eqref{pro-eta} holds and  $\eta_j \neq
0$
for $j=1, \, 2, \,  3$.     Let $u \in L^2(0, + \infty)$ and let $y \in C([0, + \infty);
L^2(0, L)) \cap L^2_{\loc}\big([0, +\infty); H^1(0, L) \big)$ be the unique solution
of \eqref{sys-y} and \eqref{IC-y}.  Assume that $u \not \equiv 0$,  $u(t) = 0$ for $t
> T$, and  $y(t, \cdot) = 0$ for large $t$. Then, there exists a real number $N(u) \ge 0$  such that $C^{-1}\|u\|_{H^{-2/3}} \leq N(u) \leq C\|u\|_{H^{-2/3}}$ for some  constant $C \ge 1$ depending only on $L$,  and \footnote{The map $u\mapsto N(u)$ is actually a norm, which
is (somewhat) explicitly given in the proof, by $N(u)^2 = \|\hat w\|_{L^2}^2$, where $w$
is defined in Eq~\eqref{eq-def-w}.}
\begin{equation}\label{pro-monotone-dir}
\int_{0}^{\infty} \int_0^L |y (t, x)|^2 e^{-ipt }\varphi_x(x) \diff x \diff t =
N(u)^2\big(E + O(1) T^{1/4}\big). 
\end{equation}
\end{proposition}

Here we use the following definition, for $s < 0$ and for $u \in L^2(\mR_+)$,
\[
\| u\|_{H^{s}(\mR)}^2 = \int_{\mR} |\hu|^2(1 + |\xi|^2)^{s} \diff
\xi,
\]
where $\hu$ is the Fourier transform of the extension of $u$ by $0$ for $t < 0$.

Before giving the proof of \Cref{pro-monotone}, we present one of its direct
consequences.  Denote $\xi_1 (t, x) = \Re \{ \varphi (x) e^{-ipt}\}$ and $\xi_2 (t, x)
= \Im \{ \varphi (x) e^{-ipt}\}$. Then
\begin{equation}
\xi_1(t, x) + i \xi_2 (t, x) = \varphi (x) e^{-ipt}.
\end{equation}
Denote $E_1 = \Re(E)$ and $E_2 = \Im (E)$,  and set
\begin{equation}\label{def-Psi}
\Psi(t, x) = E_1 \xi_1(t, x) + E_2 \xi_2(t, x).
\end{equation}

Multiplying \eqref{pro-monotone-dir} by  $\overline{E}$ and normalizing appropriately,
 we have

\begin{corollary} \label{corollary-dir} Let   $p \in \mR$ and let $\varphi$ be defined
by \eqref{def-varphi}.  Assume that \eqref{pro-eta} holds,  $\eta_j \neq 0$ for $j=1,
2, 3$,  and $E \neq 0$. There exists $T_* > 0$ such that, for any  (real) $u \in
L^2(0, + \infty)$ with $u(t) = 0$ for $t > T_*$ and $y(t,
\cdot) = 0$ for large $t$ where $y$ is the unique solution of \eqref{sys-y} and \eqref{IC-y}, 
we have
\begin{equation}\label{pro-monotone-CSQ1}
\int_{0}^\infty \int_0^{+\infty} y^2(t, x) \Psi_x(t, x) \diff x \diff t  \ge C \| u
\|_{H^{-2/3}(\mR)}^2. 
\end{equation}
\end{corollary}

We are ready to give the

\begin{proof}[Proof of \Cref{pro-monotone}] By \Cref{pro-Gen},
\[
\hu G/ H \mbox{ is an entire function}.
\]
By \Cref{lem-hol},  $G$ and $H$ are entire functions. The same holds for $\hu$ since $u(t
) = 0$ for large $t$.
One can show that the number of  common roots of $G$ and $H$ in $\mC$
is finite, see \Cref{lem-HG} in \Cref{B}. Let $z_1, \dots, z_k$ be the distinct common roots of $G$ and $H$ in $\mC$. There exist $m_1, \dots, m_k \in \N$ such that \footnote{One can prove that $m_j = 1$ for $1 \le j \le k$ by \Cref{lem-detQ} in \Cref{B}, but this is not important at this stage.}, with
\[
\Gamma(z) = \prod_{j=1}^k (z- z_j)^{m_j} \quad  \mbox{ in } \mC,
\]
the following two functions are entire
\begin{equation}
\quad \cG(z) : = \frac{G(z)}{\Gamma(z)} \quad \mbox{ and } \quad \cH(z) : = \frac{H(z)}{\Gamma(z)},
\end{equation}
and $\cG$ and $\cH$ have no common roots. Since
\[
 \hu \cG / \cH = \hu G/ H
\]
which is an  entire function, it follows that the function $v$ defined by
\begin{equation}
v(z) = \hu(z)/ \cH(z) = \hu(z) \frac{\Gamma(z)\Xi(z)}{\det Q(z)} \mbox{ in } \mC
\end{equation}
is also an entire function.

It is clear that
\begin{equation}
\hu(z) = v(z) \cH(z) \mbox{ in } \mC.
\end{equation}

We consider the holomorphic function $v$  restricted to  ${\mathcal L}_m: = \Big\{z \in
\mC;  |\Re(z)| \le c m , \; -\big( (2 m+1) /
(\sqrt{3} L) \big)^3 \le  \Im (z) \le \big( (2 m+1) / (\sqrt{3} L) \big)^3 \Big\}$ with
large $m \in \N$. Using \Cref{pro-Gen} to bound $\hat u$,  and  \Cref{lem-bh-detQ} in \Cref{B} to bound $(\det Q(z))^{-1}$, we can bound 
$v$ on $\partial \mathcal L_m$ (and thus also in the interior of $\mathcal L_m$) by
\begin{equation}\label{vvv-p0}
|v(z)| \le  C_\eps e^{(T+ \eps/2) \big( (2 m+1) / (\sqrt{3} L) \big)^3} \mbox{ in
} {\mathcal L}_m,
\end{equation}
for all $\eps >0$, since, for large $|z|$,  
\[
|\Xi (z)| \le C |z|.  
\]
Note that the constant $C_\eps$ can be chosen independently of $m$. Here we used the
fact
\[
|\hu (z)| \le C e^{T |\Im(z)|} \mbox{ for } z \in \mC.
\]
On the other hand, applying \Cref{th:lambda_asym} and item 2 of \Cref{lem-bh-detQ}, we
have
\begin{equation}\label{vvv-p2}
|v(z)| \le  C_\eps e^{(T+  \eps) |z|} \mbox{ in }
 \Big\{z \in \mC;  |\Re(z)| \ge c m , \; -\big( (2 m+1) /
(\sqrt{3} L) \big)^3 \le  \Im (z) \le \big( (2 m+1) / (\sqrt{3} L) \big)^3 \Big\}.
\end{equation}
Combining \eqref{vvv-p0} and \eqref{vvv-p2} yields
\begin{equation}\label{vvv-p3}
|v(z)| \le C_\eps e^{(T+   \eps) |z|} \mbox{ in } \mC.
\end{equation}

Since $\cH$ is a non-constant entire function, there exists $\gamma > 0$ such that
\begin{equation}\label{pro-monotone-H'}
\cH'(z + i \gamma) \neq 0 \mbox{ for all } z \in \mR.
\end{equation}
Fix such an $\gamma$ and denote $\cH_\gamma (z) = \cH(z + i \gamma) $ for $z \in \mC$.

Let us prove some asymptotics for $\cH_\gamma$. Since $\sum_{j=1}^3 \lambda_j = 0$, it follows from \eqref{eq-defQ} that 
\[
 \det Q =
(\lambda_2-\lambda_1)e^{ - \lambda_3 L} +  (\lambda_3-\lambda_2)e^{ - \lambda_1 L} +   (\lambda_1-\lambda_3)e^{ - \lambda_2 L}.
\]
We use the convention in  \Cref{th:lambda_asym}.  Thus, by \Cref{th:lambda_asym}, for fixed $\beta \ge 0$, 
\begin{equation}\label{eq-H-asym}
\cH(z + i \beta) = \frac{\det Q(z +  i \gamma)}{\Xi(z + i \gamma)\Gamma(z + i \gamma)} = \kappa 
z^{-2/3- \sum_{i=1}^k m_j}e^{ - \mu_1 L z^{1/3}}\big(1+O(z^{-1/3})\big), 
\end{equation}
where 
\[
\kappa = -  \frac{1}{ (\mu_2 - \mu_1) (\mu_1 - \mu_3)}. 
\]
We can also compute the asymptotic expansion of $\cH'(z + i \beta) $, either  by explicitly computing the asymptotic behavior of $\lambda_j'(z + i \beta)$ for large positive $z$ (formally, one just needs to take the derivative of \eqref{eq-H-asym} with respect to $z$), or by using the Cauchy
integral formula on the contour $\partial D(z,r)$ for some fixed $r$ to justify
differentiating Eq.~\eqref{eq-H-asym}. We get:
\[
 \cH'(z + i \beta) = -  \frac{\mu_1 L }{3} z^{-2/3}  \kappa 
z^{-2/3- \sum_{i=1}^k m_j}e^{ - \mu_1 L z^{1/3}}\big(1+O(z^{-1/3})\big). 
\]
We then get 
\[
\lim_{z \in \mR, z \to + \infty} \cH(z) |z|^{-2/3} / \cH_\gamma'(z)  = \alpha: = 3 e^{ - i
\pi/6} / L.
\]
Similarly, we obtain
\[
\lim_{z \in \mR, z \to - \infty} \cH(z) |z|^{-2/3} / \cH_\gamma'(z)  = - \bar \alpha.
\]
Moreover, we have
\begin{equation}\label{pro-monotone-H}
\big |\cH(z) |z|^{-2/3} -   \alpha  \cH_\gamma'(z) \big| \le C |\cH(z)|
|z|^{-1} \le C |\cH_\gamma'(z)| |z|^{-1/3} \mbox{ for large positive $z$},
\end{equation}
and
\begin{equation}\label{pro-monotone-H-*}
\big |\cH(z) |z|^{-2/3} + \bar \alpha  \cH_\gamma'(z) \big| \le C |\cH(z)|
|z|^{-1} \le C |\cH_\gamma'(z)| |z|^{-1/3} \mbox{ for large negative $z$}.
\end{equation}

Set
\begin{equation}\label{eq-def-w}
\hw(z) =  v(z) \cH_\gamma'(z) = \hu(z) \cH'_\gamma(z) \cH(z)^{-1}.
\end{equation}
Then $\hw$ is an entire function and satisfies Paley-Wiener's conditions for the
interval $(-T - \eps, T + \eps)$ for all $\eps >0$, see e.g. \cite[19.3 Theorem]{Rudin-RC}. Indeed, this follows from the
facts  $|\hw(z)| \le C_\eps  |v(z)| e^{\eps |z|}$ for $z \in \mC$ by \Cref{th:lambda_asym}, $|v(z)| \le C_\eps e^{(T + \eps)
|z|}$ for $z \in \mC$ by \eqref{vvv-p3}, $|\cH_\gamma '(z) v(z)| =
|\cH_\gamma'(z)\cH(z)^{-1} \hu(z)| \le |\hat u(z)|$ for  real $z$ with large $|z|$,
so that $\int_{\mR} |\hu|^2 < + \infty$.

We claim that\footnote{Recall that $B$ was defined in Eq.\eqref{def-B}.}
\begin{equation}\label{pro-monotone-BB}
\left| \int_0^L B(z, x) \diff x \right| \le \frac{C}{(|z| + 1)^{4/3}} \mbox{ for } z \in
\mR.
\end{equation}
In fact, this inequality follows from  \Cref{lem-B} for large $z$,  and from 
\Cref{lem-detQ} in \Cref{B} otherwise since, for if $z$ is a real solution of the 
equation $H(z) = 0$, which is simple by \Cref{lem-detQ}, it holds,  by \Cref{lem-detQ} 
again,
\[
\sum_{j=1}^3 (e^{\lambda_{j+1} L } - e^{\lambda_{j} L }) e^{\lambda_{j+2} x} \mathop{=}^{\eqref{lem-detQ-cl2}} 0.
\]

From \eqref{pro-monotone-H'}, \eqref{pro-monotone-H},  \eqref{pro-monotone-H-*},  and \eqref{pro-monotone-BB},
we derive that
\begin{equation}\label{pro-monotone-B}
\left|\hu (z) \overline{\hu (z - p)} \int_0^L B(z, x) \diff x \right|  \le  C |\hw(z)|
|\hw (z-p)| \mbox{ for } z \in \mR.
\end{equation}

Note that, for $m \ge 1$,
\begin{multline*}
\left|\int_{|z| > m} \hu (z) \overline{\hu (z - p)} \int_0^L B(z, x) \diff x \diff z - E
|\alpha|^2
\int_{|z| > m}  \hw (z) \overline{\hw(z-p)} \diff z \right| \\[6pt]
\le \int_{|z| > m} \left| \hu (z) \overline{\hu (z - p)} \Big( \int_0^L B(z, x) \diff x
- E |z|^{-4/3} \Big)   \right| \diff z \\[6pt]
+ |E|  \int_{|z| > m} \left| |\alpha|^2 \hw (z)  \overline{\hw (z - p)} -  |z|^{-4/3} \hu
(z)
\overline{\hu (z - p)} \right| dz.
\end{multline*}
Using \eqref{pro-monotone-H} \eqref{pro-monotone-H-*}, and \Cref{lem-1,lem-B}, we
derive  that
\begin{multline*}
\left|\int_{|z| > m} \hu (z) \overline{\hu (z - p)} \int_0^L B(z, x) \diff x \diff z - E
|\alpha|^2
\int_{|z| > m} \hw (z) \overline{\hw(z-p)} \diff z \right| \\[6pt]
\le  C \int_{|z| > m} |\hw (z)| |\hw (z - p)|  |z|^{-1/3} dz.
\end{multline*}
We derive from  \eqref{pro-monotone-H'} and \eqref{pro-monotone-B} that
\begin{multline*}
\left|\int_{\mR} \hu (z) \overline{\hu (z - p)} \int_0^L B(z, x) \diff x \diff z  - E
|\alpha|^2
\int_{\mR} \hw (z) \overline{ \hw(z-p)} \diff z  \right| \\[6pt]
\le  C \int_{|z| \le m} |\hw (z)| |\overline{\hw (z - p)}| \diff  z + C m^{-1/3}
\int_{|z| > m} |\hw (z)| |\hw(z-p)| \diff z.
\end{multline*}
Since, for $z \in \mR$,
\[
|\hw(z)| \le C\| w\|_{L^1} = C \|w \|_{L^1(-T, T)} \le C T^{1/2} \|w \|_{L^2(\mR)},
\]
we derive that
\begin{equation*}
\left|\int_{\mR} \hu (z) \overline{\hu (z - p)} \int_0^L B(z, x) \diff x \diff z  - E
|\alpha|^2
\int_{\mR} \hw (z)  \overline{\hw(z-p)} \diff  z \right|
\le C \int_{-T}^{T} \Big( T m + m^{-1/3} \Big) |w|^2.
\end{equation*}
Using the fact
\[
\int_{\mR} \hw (z) \overline{ \hw(z-p) } \diff z  =  \int_{\mR} |w(t)|^2 e^{-itp} \diff t
=
 \int_{-T}^{T} |w(t)|^2 e^{-itp} \diff t,
\]
we obtain, by choosing $m = 1/ T^{3/4}$,
\[
\int_{\mR} \hu (z) \overline{\hu (z - p)} \int_0^L B(z, x) \diff x \diff z =  E
|\alpha|^2
\int_{-T}^T |w(t)|^2 (1 +  O(1) T^{1/4}) \diff  t.
\]
The conclusion follows by noting that
\[
 \int_{\mR} |w(t)|^2  =  \int_{\mR} |\hw (z)|^2 \diff  z \ge C \int_{\mR}
\frac{|\hu(z)|^2}{1 + |z|^{4/3}} \diff z,
\]
and  by normalizing $u$ such that $|\alpha| \| w \|_{L^2(\mR)} = 1$.
\end{proof}

\section{Useful estimates for the linear KdV equations}\label{sect-L-KdV}

In this section, we establish several results for the linear  KdV equations which will
be used in the proof of \Cref{thm-main}. Our study of  the inhomogeneous KdV equations
is based on three elements. The first one is
on the information of the KdV equations explored previously. The second one is a
connection between the KdV equations and the KdV-Burgers equations, as previously
suggested in \cite{Kato83, Bona09}. The third one is on estimates for the
KdV-Burgers equations with periodic boundary condition.   This section contains two
subsections. The first one is on  inhomogeneous KdV-Burgers equations with periodic
boundary condition and the second one is on the inhomogeneous KdV equations.

\subsection{On the linear KdV-Burgers equations}  \label{sect-KdVB}

In this section, we derive several estimates for the solutions of the linear
KdV-Burgers equations using low regular data information.  The main result of this
section is the following result:

\begin{lemma}\label{lem-kdvB} Let $L > 0$  and $f_1 \in L^1\big(\mR_+; L^1(0, L) \big)$
and $f_2 \in L^1\big(\mR_+; W^{1, 1} (0, L) \big)$
be such that
\begin{equation}\label{lem-kdvB-f1}
\int_0^L f_1(t, x) \diff x  = 0 \mbox{ for a.e. } t>0,
\end{equation}
and
\begin{equation}\label{lem-kdvB-f2}
f_2(t, 0) = f_2(t, L) = 0 \mbox{ for a.e. } t > 0.
\end{equation}
Set $f = f_1 + f_{2, x}$ and assume that $f \in L^1\big(\mR_+; L^2(0, L) \big)$.
Let $y$ be the unique solution in $C\big([0, + \infty); L^2(0, L)
\big) \cap L^2_{\loc}\big([0, + \infty); H^1(0, L) \big)$, which is periodic in
space, of the system
\begin{equation}\label{sys-kdvB}
y_t (t, x) + 4 y_x (t, x) + y_{xxx} (t, x) - 3 y_{xx} (t, x) = f(t, x) \mbox{ in } (0,
+\infty) \times (0, L),
\end{equation}
and
\begin{equation}\label{bdry-kdvB}
y(t = 0, \cdot)  = 0 \mbox{ in } (0, L).
\end{equation}
We have, for $x \in [0, L]$,
\begin{equation}\label{lem-kdvB-cl1}
\| y(\cdot, x) \|_{L^2 (\mR_+)} + \| y_x(\cdot, x) \|_{H^{-1/3}(\mR)}   \le C \|
f\|_{L^1(\mR_+ \times (0, L))},
\end{equation}
and
\begin{equation}\label{lem-kdvB-cl1-*}
\| y(\cdot, x) \|_{H^{-1/3} (\mR)} + \| y_x(\cdot, x) \|_{H^{-2/3}(\mR)} + \| y
\|_{L^2(\mR_+; H^{-1} (0, L))}  \le C \|
(f_1, f_2)\|_{L^1(\mR_+ \times (0, L))}.
\end{equation}
Assume that  $f(t, \cdot) = 0$ for $t > T$. We have, for all $\delta > 0$, and for all
$t \ge T + \delta$,
\begin{equation}\label{lem-kdvB-cl3}
| y_t (t, x)| + |y_x(t, x)|
\le C_\delta \| (f_1, f_2)\|_{L^1(\mR_+ \times (0, L))} \mbox{ for } x \in [0, L].
\end{equation}
Here $C$ (resp. $C_\delta$) denotes a positive constant depending only on $L$ (resp.
 $L$ and $\delta$).
 \end{lemma}

\begin{remark} Using the standard  energy method, as for the KdV equations,   one
can prove that if $f  \in L^1(\mR_+, L^2(0, L))$ with $\int_0^L f(t, x) \diff x  =0$ for
a.e. $t>0$ (this holds by \eqref{lem-kdvB-f1} and \eqref{lem-kdvB-f2}), then
\eqref{sys-kdvB}-\eqref{bdry-kdvB} has a unique solution in $C([0, + \infty); L^2(0,
L)) \cap L^2([0, + \infty); H^1(0, L))$  which is periodic in space.
\end{remark}

In the proof of \Cref{lem-kdvB}, we use the following elementary estimate, which has
its root in the work of Bourgain \cite{Bourgain93}.

\begin{lemma} \label{lem-Pre1} There exists a positive constant $C$ such that, for $j
= 0, 1$, and $z \in \mR$,\footnote{We recall that an absolutely convergent sum is nothing 
but the integral with the counting measure, which is $\sigma$-finite. In the 
following, we will often exchange sums and integrals without comments, the justification 
being one of Fubini's theorem.} 
\begin{equation}\label{lem-Pre1-cl}
\sum_{n\neq 0} \frac{|n|^j}{|z + 4 n - n^3| + n^2} \le  \frac{C \ln
(|z|+2)}{(|z| + 2)^{\frac{2-j}{3}}}.
\end{equation}
\end{lemma}

\begin{proof} For $z \in \mR$, let $k \in \mZ$ be such that $k^3 \le z <  (k+1)^3$. It
is clear that
\begin{equation}\label{lem-Pre1-p1}
\sum_{n\neq 0} \frac{|n|^j}{|z + 4 n - n^3| + n^2} =  \sum_{m+
k \neq 0} \frac{|m+k|^j}{|z + 4 (m+k) - (m+k)^3| + (m+k)^2}.
\end{equation}
We split the sum in two parts, one for $|m|\leq 2|k|+2$ and one for $|m|>2|k|+2$.  Since $k^3 \le z <  (k+1)^3$, one can check that, for $m \in \mZ$,  $m + k \neq 0$,  and $|m| \le 2 |k| + 2$, 
$$
|z+4(m+k)-(m+k)^3| + |m +k|^2  \ge C (|m| +1) (|k| + 2)^2, 
$$
and, for $|m| \ge 2 |k| + 2$, 
$$
|z+4(m+k)-(m+k)^3| + |m +k|^2 \ge C |m|^3 
$$
(by considering  $|k| \ge 10$ and $|k|  <  10$). 
We deduce that 
\begin{multline}\label{lem-Pre1-p2}
 \sum_{|m| \le 2 |k|  + 2, m + k \neq 0} \frac{|m+k|^j}{|z + 4 (m+k) - (m+k)^3| +
(m+k)^2} \\[6pt]
  \le C \sum_{|m| \le 2 |k| + 2}  \frac{1}{(|k| + 2)^{2-j} (|m|+1)} \le \frac{C
\ln (|k| + 2)}{(|k|+2)^{2-j}},
\end{multline}
and
\begin{equation}\label{lem-Pre1-p3}
 \sum_{|m| >  2 |k|  + 2} \frac{|m+k|^{j}}{|z + 4 (m+k) - (m+k)^3| + (m+k)^2}  \le
C \sum_{|m| > 2 |k| + 2} \frac{1}{|m|^{3-j}} \le \frac{C}{(|k|+2)^{2 -j}}.
\end{equation}
Combining \eqref{lem-Pre1-p1} - \eqref{lem-Pre1-p3}  yields \eqref{lem-Pre1-cl}.
\end{proof}

In what follows, for an appropriate function $\zeta$ defined in $\mR_+ \times (0,
L)$, we denote
\[
\hat{\hat \zeta }(z, n) = \frac{1}{L} \int_0^L \hat \zeta(z, x) e^{- \frac{i 2 \pi
n x}{L}} \diff x \mbox{ for } (z, n) \in \mR \times \mZ.
\]
Recall that to define $\hat \zeta(z, x)$, we extend $\zeta$ by $0$ for $t < 0$.

\begin{proof}[Proof of \Cref{lem-kdvB}] For simplicity of notations, we will assume that
$L =  2 \pi$. We
establish \eqref{lem-kdvB-cl1}, \eqref{lem-kdvB-cl1-*},   and \eqref{lem-kdvB-cl3} in
Steps 1, 2 and 3 below.

\step{Proof of  \eqref{lem-kdvB-cl1}}

We first estimate $\| y(\cdot, x) \|_{L^2 (\mR_+)}$ for $x \in [0, L]$.
From \eqref{sys-kdvB} and \eqref{bdry-kdvB},  we have
\begin{equation}\label{lem-kdvB-form-y1}
\hhy(z, n) =  \frac{ \hhf(z, n)}{i (z + 4 n - n^3 ) + 3 n^2} \quad \mbox{ for } (z, n)
\in \mR \times  (\mZ \setminus \{0 \}),
\end{equation}
and
\begin{equation}\label{lem-kdvB-form-y2}
\hhy(z, 0) =   0 \quad \mbox{ for }  z \in \mR
\end{equation}
since $\dsp \int_0^L f(t, x) \diff x  = 0$ for $t>0$ by \eqref{lem-kdvB-f1} and
\eqref{lem-kdvB-f2}.
By Plancherel's theorem, we obtain
\begin{equation}\label{lem-kdvB-id1}
\int_{\mR_+} |y(t, x)|^2 \diff t  = \int_{\mR} |\hy(z, x)|^2 \diff z
\le C  \int_{\mR} \left|\sum_{n\neq 0} \frac{|\hhf(z, n)|}{|z + 4 n -
n^3| +  n^2}   \right|^2  \diff z.
\end{equation}
Since
\begin{equation}\label{lem-kdvB-f}
|\hhf(z, n)| \le C \| f\|_{L^1(\mR_+ \times (0, L))},
\end{equation}
it follows from \eqref{lem-kdvB-id1} that
\begin{equation}\label{lem-kdvB-id1-1}
\int_{\mR_+} |y(t, x)|^2 \diff t
\le C \| f\|_{L^1(\mR_+ \times (0, L))}^2  \int_{\mR} \left|\sum_{n\neq 0} \frac{1}{|z + 
4 n - n^3| +  n^2}  \right|^2  \diff z.
\end{equation}
Applying  \Cref{lem-Pre1}  with $j=0$, we derive from \eqref{lem-kdvB-id1-1} that
\begin{equation*}
\int_{\mR_+} |y(t, x)|^2 \diff t   \le C \| f\|_{L^1(\mR_+ \times (0, L))}^2 \int_{\mR}
\frac{\ln^2 (|z| + 2)}{ (|z| + 2)^{4/3}} \diff z,
\end{equation*}
which yields
\begin{equation}\label{lem-kdvB-cl1-1}
\|y (\cdot, x) \|_{L^2}  \le C \| f\|_{L^1(\mR_+ \times (0, L))}.
\end{equation}

We next estimate $\| y_x(\cdot, x) \|_{H^{-1/3}(\mR_+)}$ for $x \in [0, L]$. We have,
by  \eqref{lem-kdvB-form-y1},  \eqref{lem-kdvB-form-y2},  and \eqref{lem-kdvB-f},
\begin{multline}\label{lem-kdvB-y-2-1}
\|  y_x (\cdot, x) \|_{H^{-1/3}(\mR_+)}^2 \\[6pt]
\le   C \| f\|_{L^1(\mR_+ \times (0, L))}^2   \int_{\mR} \frac{1}{(1 + |z|^2)^{1/3}}
\left| \sum_{n\neq 0} \frac{ |n| }{ |z + 4 n - n^3|  + n^2}
\right|^2 \diff  z.
\end{multline}
Applying  \Cref{lem-Pre1} with $j=1$, we derive from \eqref{lem-kdvB-y-2-1} that
\begin{equation*}
\| y_x (\cdot, x) \|_{H^{-1/3}(\mR_+)}^2    \le C \| f\|_{L^1(\mR_+ \times (0, L))}^2
\int_{\mR}  \frac{\ln^2 (|z| + 2)}{ (|z| + 2)^{4/3}} \diff z,
\end{equation*}
which yields
\begin{equation}\label{lem-kdvB-cl2-1}
\|y_x (\cdot, x) \|_{H^{-1/3}(\mR)}  \le C \| f\|_{L^1(\mR_+ \times (0, L))}.
\end{equation}
Assertion~\eqref{lem-kdvB-cl1} now follows from \eqref{lem-kdvB-cl1-1} and
\eqref{lem-kdvB-cl2-1}.

\step{Proof of  \eqref{lem-kdvB-cl1-*}} By Step 1, without loss of
generality, one might assume that $f_1 = 0$.  The proof of  the inequality $\| y(\cdot, x)
\|_{H^{-1/3}} \le C \| f_2\|_{L^1(\mR_+ \times (0, L))}$ is similar to the one of
\eqref{lem-kdvB-cl2-1} and is omitted.

To prove
\begin{equation}\label{lem-kdvB-cl1-*-Step2}
\| y_x(\cdot, x) \|_{H^{-2/3}(\mR)} \le C \| f_2\|_{L^1(\mR_+ \times (0, L))},
\end{equation}
 we proceed as follows. For $z \in \mR$, it holds
\begin{equation}\label{lem-kdvB-cl1-*-p1}
\hy_x(z, x) =  - \frac{1}{L} \int_0^{L} \hf_2 (z, \xi)  \sum_{n\neq 0}
\frac{n^2 e^{ i n (x - \xi)}}{i (z + 4n -n^3) + 3 n^2} \diff  \xi.
\end{equation}
We have, for some large positive constant $c$,
\begin{equation*}
\left| \sum_{|n| \ge c(|z| + 1)}   \frac{n^2 e^{ i n (x - \xi)}}{i
(z + 4n -n^3) + 3 n^2} +  \sum_{|n| \ge c(|z| + 1)}
\frac{e^{ i n (x - \xi)}}{i n} \right|
 \le C  \sum_{|n| \ge c(|z| + 1)} \frac{1}{|n|^2}  \le
\frac{C}{|z| + 1},
\end{equation*}
\[
\left| \sum_{0<|n| \le c (|z| + 1)}  \frac{e^{ i n (x - \xi)}}{i n} \right| \le  C \ln 
(|z| + 2),
\]
and, as in  \eqref{lem-Pre1-p2}  in the proof of \Cref{lem-Pre1},
\[
\left| \sum_{0<|n| \le c (|z| + 1)}   \frac{n^2 e^{ i n (x - \xi)}}{i
(z + 4n -n^3) + 3 n^2} \right| \le  C \ln (|z| + 2).
\]
It follows that
\begin{equation}\label{lem-kdvB-cl1-*-p2}
\left| \sum_{n\neq 0}   \frac{n^2 e^{ i n (x - \xi)}}{i (z + 4n -n^3) + 3
n^2} +  \sum_{n\neq 0}  \frac{e^{ i n (x - \xi)}}{i n}
\right|
\le \frac{C}{|z| + 1} + C \ln (|z| + 2).
\end{equation}
Since
\[
\sum_{n\neq 0}  \frac{e^{ i n \xi'}}{i n} = - \xi' + \pi  \mbox{ for } 
\xi' \in (0, 2 \pi),
\]
and
\[
\| y_x(\cdot, x)\|_{H^{-2/3}(\mR)}^2 =  \int_{\mR} \frac{|\hy_x(z, x)|^2}{(1 +
|z|^2)^{2/3}} \diff z,
\]
assertion \eqref{lem-kdvB-cl1-*-Step2} follows from \eqref{lem-kdvB-cl1-*-p1} and
\eqref{lem-kdvB-cl1-*-p2}.

We next deal with
\[
\| y \|_{L^2(\mR_+; H^{-1} (0, L))}  \le C \| f_2 \|_{L^1(\mR_+ \times (0, L))}.
\]
Since
\[
\| y \|_{L^2(\mR_+; H^{-1} (0, L))}^2  \le C \int_{\mR} \sum_{n\neq 0} \left|
\frac{\hhf_2(z, n)}{|i (z + 4n - n^3)| + 3 n^2}\right|^2 \diff z,
\]
the  estimate follows from \Cref{lem-Pre1}. The proof of Step 2 is complete.

\step{Proof of \eqref{lem-kdvB-cl3}}

For simplicity of the presentation, we will assume that $f_1 = 0$. We have the following
representation for the solution:
\begin{equation}\label{lem-kdvB-S3}
y(t, x) =   \sum_{n\neq 0}  e^{ in x} \int_0^t e^{- \big(i (4 n - n^3) + 3
n^2 \big) (t- \tau) }   \left( \frac{in}{L} \int_0^L f_2(\tau, \xi) e^{- in \xi} \diff
\xi \right) \diff  \tau.
\end{equation}
Let $\mathds{1}_{A}$ denote the characteristic function of a set $A$ in $\mR$.
Assertion \eqref{lem-kdvB-cl3} then follows easily from \eqref{lem-kdvB-S3} by noting
that, for $t \ge T+ \delta$
\[
 \sum_{n\neq 0}  \int_0^t |n|^{10} e^{- 3 n^2 (t-\tau) }
\mathds{1}_{\big\{\tau <
T\big\}}  \diff \tau < C_\delta.
\]
The proof is complete.
\end{proof}

\subsection{On the linear KdV equations}\label{sect-LKdV}

In this section, we derive various results on the linear KdV equations using low
regularity data information.
These will be used in the proof of \Cref{thm-main}. We begin with

\begin{lemma}\label{lem-kdv1} Let $h = (h_1, h_2, h_3) \in H^{1/3} (\mR_+) \times
H^{1/3}(\mR_+) \times L^2 (\mR_+)$,  and
let $y \in C\big([0, + \infty); L^2(0, L) \big) \cap L^2_{\loc}\big([0, + \infty); H^1(0,
L) \big) $ be the unique solution of the system
\begin{equation}\label{sys-y-LKdV}\left\{
\begin{array}{cl}
y_t (t, x) + y_x (t, x) + y_{xxx} (t, x) = 0 &  \mbox{ in } (0, +\infty) \times  (0, L),
\\[6pt]
y(t, x=0) = h_1(t),  \;  y(t, x=L) = h_2(t), \;  y_x(t , x= L)   = h_3(t)& \mbox{ in } (0,
+\infty),
\end{array}\right.
\end{equation}
and
\begin{equation}\label{IC-y-LKdV}
y(t = 0, \cdot)  = 0 \mbox{ in } (0, L).
\end{equation}
We have, for $T>0$,
\begin{equation}\label{lem-kdv1-cl2}
\| y\|_{L^2((0, T) \times (0, L))} \le C_{T, L} \Big( \| (h_1, h_2) \|_{L^2(\mR_+)} + \|
h_3 \|_{H^{-1/3}(\mR)}\Big),
\end{equation}
and
\begin{equation}\label{lem-kdv1-cl2-*}
\| y\|_{L^2((0, T); H^{-1} (0, L))} \le C_{T, L} \Big( \| (h_1, h_2) \|_{H^{-1/3}(\mR)} +
\| h_3 \|_{H^{-2/3}(\mR)}\Big),
\end{equation}
for some positive constant $C_{T, L}$  independent of $h$.
\end{lemma}

Here and in what follows,  $H^{-1}(0, L)$ is the dual space of $H^1_0(0, L)$ with the
corresponding norm.

\begin{proof} By the linearity and the uniqueness of the system, it suffices to
consider the three cases $(h_1, h_2, h_3) = (0, 0, h_3)$, $(h_1, h_2, h_3)  = (h_1, 0,
0)$, and $(h_1, h_2, h_3)  = (0, h_2, 0)$ separately.

We first consider the case $(h_1, h_2, h_3)  = (0, 0, h_3)$. Making a truncation, without
loss of generality, one might assume that $h_3 =0$ for $t > 2 T$. This fact is assumed
from now on. Let $g_{3} \in C^1(\mR)$ be such that $\supp g_{3} \subset [T, 3 T]$, and if
$z$ is a real solution of the equation $\det
Q(z)  \Xi(z) = 0$  of order $m$ then $z$ is also a real solution of order $m$ of $\hat h_3(z) -
\hat g_{3}(z)$, and
\[
\|g_3 \|_{H^{-1/3}(\mR)} \le C_{T, L} \| h_3\|_{H^{-2/3}(\mR)}.
\]
The construction of $g_3$, inspired by the moment method, see e.g. \cite{TT07},  can be
done as follows.
Set $\eta(t) = e^{-1/ (t^2 - (T)^2)} \mathds{1}_{|t| < T}$ for $t \in \mR$. Assume that
$z_1$, \dots, $z_k$ are real, distinct solutions of the equation $ \det Q(z)  \Xi(z)  =0$,
and $m_1$, \dots, $m_k$ are the corresponding orders (the number of real solutions of the
equation $\det Q(z)  \Xi(z)   =0$ is finite by \Cref{lem-detQ} and in fact they are simple; nevertheless,
we ignore this point and present a proof without using this information). Set, for $z \in
\mC$,
\[
\zeta(z) = \sum_{i=1}^k   \left(  \hat \eta(z - z_i) \mathop{\prod_{j=1}^k}_{j \neq i} (z
- z_j)^{m_j} \Big( \sum_{l=0}^{m_i} c_{i, l}  (z - z_i)^{l} \Big) \right),
\]
where $c_{i, l} \in \mC$ is chosen such that
\[
\frac{d^{l}}{dz^{l}} \Big( e^{2 i T z}\zeta(z)  \Big)_{z = z_i}= \frac{d^{l}}{dz^{l}} \hat
h_3 (z_i) \mbox{ for } 0 \le l \le m_i, \; 1 \le i \le k.
\]
This can be done since $\hat \eta (0) \neq 0$. Since
\[
|\hat \eta(z)| \le C e^{T |\Im (z)|},
\]
and, by \cite[Lemma 4.3]{TT07},
\[
|\hat \eta(z)| \le C_1 e^{- C_2 |z|^{1/2}} \mbox{ for } z \in \mR,
\]
using Paley-Wiener's theorem, one can prove that $\zeta$ is the Fourier transform of a
function $\psi$ of class $C^1$; moreover,  $\psi$ has the support in $[-T, T]$. Set, for
$z \in \mC$,
\[
g_3(t) = \psi(t + 2T).
\]
Using the fact $\hat g_3 (z) = e^{ i 2 T z} \zeta (z)$, one can check that $\hat g_3 -
\hat h_3$ has solutions $z_1$, \dots, $z_k$ with the corresponding orders $m_1$, \dots,
$m_k$.  One can check that
\[
\| \psi \|_{C^1} \le C_{T, L} \sum_{i=1}^k \sum_{l=0}^{m_i} \left|\frac{d^{l}}{dz^{l}}
\hat h_3 (z_i)\right|,
\]
which yields
\[
\| \psi \|_{C^1} \le C_{T, L}  \| h_3 \|_{H^{-2/3}(\mR)}.
\]
The required properties  of $g_3$ follow.

By considering the solution corresponding to $h_3 - g_3$, without loss of generality, one
might assume that  if $z$ is a real solution of order $m$ of the equation $\det Q(z)  \Xi(z) = 0$  then
$z$ is also a real solution of order $m$ of $\hat h_3(z)$. This fact is assumed from now
on.

We now establish \eqref{lem-kdv1-cl2}.   We have, by \Cref{lem-form-sol},
\begin{equation}\label{lem-kdv1-y}
\hy(z, x) = \frac{\hh_3(z) }{\det Q}  \sum_{j=1}^3 \big(e^{\lambda_{j+2} L } -
e^{\lambda_{j+1} L } \big) e^{\lambda_j x} \mbox{ for a.e. } x \in (0, L).
\end{equation}
From the assumption of $h_3$, we have, for $z \in \mR$ and $|z| \le \gamma$,
\begin{equation}\label{lem-kdv-1-p1}
\left| \frac{\hat h_3(z)}{\det Q(z)}  \sum_{j=1}^3 \big(e^{\lambda_{j+2} L } -
e^{\lambda_{j+1} L } \big) e^{\lambda_j x} \right| \le C_{T, \gamma} \|
h_3\|_{H^{-2/3}(\mR)},
\end{equation}
and,
by  \Cref{th:lambda_asym}, for $z \in \mR$, $|z| \ge \gamma$ with sufficiently large
$\gamma$,
\begin{equation}\label{lem-kdv-1-p2}
\left| \frac{1}{\det Q}  \sum_{j=1}^3 \big(e^{\lambda_{j+2} L } - e^{\lambda_{j+1} L }
\big) e^{\lambda_j x} \right| \le \frac{C}{ (1 + |z|)^{1/3}}.
\end{equation}
Combining \eqref{lem-kdv-1-p1} and \eqref{lem-kdv-1-p2} yields
\[
\|\hy\|_{L^2\big( \mR \times (0, L) \big)} \le C_T \| h_3 \|_{H^{-1/3}(\mR)},
\]
which  is \eqref{lem-kdv1-cl2} when $(h_1, h_2, h_3) = (0, 0, h_3)$.

We next deal with \eqref{lem-kdv1-cl2-*}. The proof of \eqref{lem-kdv1-cl2-*} is similar
to the one of \eqref{lem-kdv1-cl2}. One just notes that, instead of \eqref{lem-kdv-1-p2},
it holds, for $z \in \mR$, $|z| \ge \gamma$ with sufficiently large  $\gamma$,
\begin{equation}
\left\| \frac{1}{\det Q}  \sum_{j=1}^3  \big(e^{\lambda_{j+2} L } - e^{\lambda_{j+1} L }
\big) e^{\lambda_j x} \right\|_{H^{-1}(0, L)} \le \frac{C}{ (1 + |z|)^{2/3}}.
\end{equation}
The details are omitted.

The proof in the case $(h_1, h_2, h_3) = (h_1, 0, 0)$ or in the case $(h_1, h_2, h_3) =
(0, h_2, 0)$ is similar. We only mention here that  the solution corresponding to the
triple $(h_1, 0, 0)$ is given by 
\[
\hy(z, x) = \frac{\hh_1(z) }{\det Q}  \sum_{j=1}^3 (\lambda_{j+2} - \lambda_{j+1})
e^{\lambda_j (x- L)}
\mbox{ for a.e. } x \in (0, L),
\]
and the solution corresponding to the triple $(0, h_2, 0)$ is given by 
\[
\hy(z, x) = \frac{\hh_2(z) }{\det Q}  \sum_{j=1}^3 (\lambda_{j+1} e^{\lambda_{j+1} L } -
\lambda_{j+2} e^{\lambda_{j+2} L }) e^{\lambda_j x}
\mbox{ for a.e. } x \in (0, L).
\]
The details are left to the reader.
\end{proof}

\begin{remark} The estimates in \Cref{lem-kdv1} are in the spirit of the
well-posedness results due to Bona et al. in \cite{Bona09} (see also \cite{Bona03}) but
quite different. The setting of \Cref{lem-kdv1} is below the  limiting case in
\cite{Bona09},  which   was not investigated in their work.
\end{remark}

We next establish a variant of \Cref{lem-kdv1} for inhomogeneous KdV systems.

\begin{lemma}\label{lem-kdv3} Let $L > 0$ and $T>0$.  Let  $h = (h_1, h_2, h_3) \in
H^{1/3}(\mR_+) \times H^{1/3}(\mR_+) \times L^2(\mR_+)$, $f_1 \in L^1 \big((0, T) \times
(0, L) \big)$, and $f_2 \in L^1 \big((0, T); W^{1,1}(0, L) \big)$ with
\begin{equation}\label{lem-kdv3-cond-f2}
f_2 (t, 0) = f_2 (t, L) = 0.
\end{equation}
Set $f = f_1 + f_{2, x}$ and assume that $f \in L^1(\mR_+; L^2(0, L))$. Let  $y \in
C\big([0, + \infty); L^2(0, L) \big) \cap L^2_{\loc}\big([0, + \infty); H^1(0, L) \big) $
be the unique solution of  the system
\begin{equation}\label{lem-kdv3-S}\left\{
\begin{array}{cl}
y_t (t, x) + y_x (t, x) + y_{xxx} (t, x) = f(t, x) &  \mbox{ in } (0, +\infty) \times (0,
L), \\[6pt]
y(t, x=0) = h_1(t), \;  y(t, x=L) = h_2 (t), \;  y_x(t , x= L) = h_3(t) & \mbox{ in } (0,
+\infty),
\end{array}\right.
\end{equation}
and
\begin{equation*}
y(t = 0, \cdot)  = 0 \mbox{ in } (0, L).
\end{equation*}
We have
\begin{equation}\label{lem-kdv3-cl1}
\| y \|_{L^2 \big( (0, T) \times (0, L) \big)} \le C_T \Big( \| (h_1, h_2) \|_{L^2(\mR_+)}
+ \|h_3 \|_{H^{-1/3}(\mR)} + \| f\|_{L^1(\mR_+ \times (0, L))} \Big),
\end{equation}
and
\begin{equation}\label{lem-kdv3-cl1-*}
\| y \|_{L^2 \big( (0, T); H^{-1} (0, L) \big)} \le C_T \Big( \| (h_1, h_2)
\|_{H^{-1/3}(\mR)} + \|h_3 \|_{H^{-2/3}(\mR)} + \| (f_1, f_2)\|_{L^1(\mR_+ \times (0, L))}
\Big).
\end{equation}
Assume in addition that $h(t, \cdot) = 0$ and $f(t, \cdot) = 0$ for $t \ge T_1$ for some
$0 < T_1 < T$. Then,  for any $\delta > 0$ and for $T_1+ \delta \le t \le T$,  we have
\begin{equation}\label{lem-kdv3-cl2}
|y_t (t, x)| + |y_x (t, x)| \le C_{T, T_1, \delta} \Big( \| (h_1, h_2) \|_{H^{-1/3}(\mR)}
+ \|h_3 \|_{H^{-2/3}(\mR)} + \| (f_1, f_2) \|_{L^1(\mR_+ \times (0, L))} \Big).
\end{equation}
Here $C_T$ and $C_{T, T_1, \delta}$ denote  positive constants independent of $h$ and $f$.
\end{lemma}

\begin{proof}  The proof is based on a connection between  the  KdV equations and the
KdV-Burgers equations.  Set $v(t, x) = e^{-2 t + x} y (t, x)$, which is equivalent to
$y(t, x)  = e^{2t - x} v(t, x)$. Then
\[
y_t (t, x) = \big(2 v (t, x) + v_t (t, x) \big) e^{2 t - x}, \quad y_x (t, x) = \big(-
v(t, x) + v_x (t, x) \big) e^{2t -x},
\]
\[
y_{xxx} (t, x) = \big(v_{xxx} (t, x) - 3 v_{xx} (t, x) + 3 v_x (t, x) - v(t, x) \big) e^{2
t -x}.
\]
Hence, if $y$ satisfies the equation
\[
y_t (t, x) + y_x (t, x) + y_{xxx} (t, x) = f (t, x) \mbox{ in } \mR_+ \times  (0, L),
\]
then it holds
\[
v_t (t, x) + 4 v_x (t, x)  + v_{xxx} (t, x) - 3 v_{xx} (t, x) = f(t, x) e^{-2 t + x}
\mbox{ in } \mR_+ \times  (0, L).
\]

Set, in $\mR_+ \times (0, L)$,
\begin{equation}\label{lem-kdv3-def-hg}
\psi(t, x) =  \psi(t) : =  \frac{1}{L} \int_0^L f(t, \xi) e^{-2 t + \xi} \diff  \xi \quad
\mbox{ and } \quad
g(t, x) : = f(t, x) e^{-2 t + x} - \psi(t, x).
\end{equation}
Then
\[
\int_0^L g(t, x) \diff x = 0.
\]
Let $y_1 \in  C\big([0, + \infty); L^2(0, L) \big) \cap L^2_{\loc}\big([0, + \infty);
H^1(0, L) \big)$ be the unique solution which is periodic in space  of the system
\begin{equation}
y_{1, t} (t, x) + 4 y_{1, x} (t, x) + y_{1, xxx} (t, x) - 3 y_{1, xx} (t, x) = g(t, x)
\mbox{ in } (0, +\infty) \times (0, L),
\end{equation}
and
\begin{equation}
y_1(t = 0, \cdot)  = 0 \mbox{ in } (0, L).
\end{equation}
We have, by \eqref{lem-kdv3-cond-f2},
\begin{equation}\label{lem-kdv3-g}
g(t, x) = f_1(t, x) e^{-2t + x} + f_{2, x} (t, x) e^{-2t + x}  - \psi(t, x),
\end{equation}
and
\begin{equation}\label{lem-kdv3-psi}
\psi(t, x) =  \frac{1}{L}\int_0^L f_1(t, \xi) e^{-2 t + \xi} \diff  \xi - \frac{1}{L}
\int_0^L f_2(t, \xi) e^{-2 t + \xi} \diff  \xi.
\end{equation}

Applying \Cref{lem-kdvB}, we have
\begin{equation*}
\| y_{1}(\cdot, x)\|_{L^2(\mR_+)} + \| y_{1, x} (\cdot, x) \|_{H^{-1/3}(\mR)} \le C \|g
\|_{L^1(\mR_+ \times (0, L))}
\end{equation*}
which yields, by \eqref{lem-kdv3-def-hg},
\begin{equation}\label{lem-kdv3-y1}
\| y_{1}(\cdot, x)\|_{L^2(\mR_+)} + \| y_{1, x} (\cdot, x) \|_{H^{-1/3}(\mR)} \le C \| f
\|_{L^1(\mR_+ \times (0, L))}.
\end{equation}
Similarly, by noting $f_{2, x} (t, x) e^{-2t + x} = \big(f_2(t, x) e^{-2t + x} \big)_x -
f_{2} (t, x) e^{-2t + x} $,  we get
\begin{equation}\label{lem-kdv3-y1-*}
\| y_{1}(\cdot, x)\|_{H^{-1/3}(\mR)} + \| y_{1, x} (\cdot, x) \|_{H^{-2/3}(\mR)} \le C \|
(f_1, f_2) \|_{L^1(\mR_+ \times (0, L))}.
\end{equation}

Applying  \Cref{lem-kdvB} again,  we obtain
\begin{equation}\label{lem-kdv3-y1-2}
| y_{1, x}(t, x)| + | y_{1, t} (t, x) | \le C_{T, T_1, \delta}  \| (f_1, f_2)
\|_{L^1(\mR_+ \times (0, L))} \mbox{ for } T_1 + \delta/2 \le t \le T.
\end{equation}
if $f = 0$ for $t \ge T_1$.

Fix $\varphi \in C(\mR)$ such that $\varphi = 1$ for  $|t| \le T$ and $\varphi = 0$ for
$|t| > 2T$.  Let $y_2 \in C\big([0, + \infty); L^2(0, L) \big) \cap L^2_{\loc}\big([0, +
\infty); H^1(0, L) \big)$ be the unique solution of the system
\begin{equation*}\left\{
\begin{array}{cl}
y_{2, t} (t, x) + y_{2, x} (t, x) + y_{2, xxx} (t, x) =  \varphi(t) \psi(t, x) &  \mbox{
in } (0, +\infty) \times (0, L), \\[6pt]
y_2(t, x=0) = h_1(t)  - \varphi (t) e^{2t} y_1(t, 0)  & \mbox{ in } (0, +\infty), \\[6pt]
y_2(t, x=L) = h_2 (t) - \varphi (t)  e^{2t - L} y_1(t, L) & \mbox{ in } (0, +\infty),
\\[6pt]
y_{2, x}(t , x= L) = h_3(t) - \varphi (t) \big(e^{2t - \cdot} y_1(t, \cdot) \big)_x (t, L)
& \mbox{ in } (0, +\infty),
\end{array}\right.
\end{equation*}
and
\begin{equation*}
y_2(t = 0, \cdot)  = 0 \mbox{ in } (0, L).
\end{equation*}
Using \eqref{lem-kdv3-g}  and applying \Cref{lem-kdv1} to $y_2$, from \eqref{lem-kdv3-y1},
we have
\begin{equation}\label{lem-kdv3-y2}
\| y_2\|_{L^2\big((0, T) \times (0, L) \big)} \le C_T  \Big( \| (h_1, h_2) \|_{L^2(\mR_+)}
+ \|h_3 \|_{H^{-1/3}(\mR)} + \| f\|_{L^1(\mR_+ \times (0, L))} \Big),
\end{equation}
and from \eqref{lem-kdv3-y1-*}, we obtain
\begin{equation}\label{lem-kdv3-y2-*}
\| y_2\|_{L^2\big((0, T); H^{-1} (0, L) \big)} \le C_T  \Big( \| (h_1, h_2)
\|_{H^{-1/3}(\mR)} + \|h_3 \|_{H^{-2/3}(\mR)} + \| (f_1, f_2)\|_{L^1(\mR_+ \times (0, L))}
\Big).
\end{equation}

One can verify that $y_1 + y_2$ and $y$ satisfy the same system for $0 \le t \le T$ and
they are in the space $C([0, T]; L^2(0, L)) \cap L^2(0, T; H^1(0, L))$. By the
well-posedness of the KdV system, one has
\[
y = y_1 + y_2  \mbox{ in } (0, T) \times (0, L).
\]
Combining \eqref{lem-kdv3-y1} and \eqref{lem-kdv3-y2} yields \eqref{lem-kdv3-cl1}, and
combining \eqref{lem-kdv3-y1-*} and \eqref{lem-kdv3-y2-*} yields \eqref{lem-kdv3-cl1-*}.
Combining \eqref{lem-kdv3-y1-2} and \eqref{lem-kdv3-y2} yields, for some $T_1 +\delta/2
\le \tau \le T_1 + 3 \delta/4 $,
\begin{equation}\label{lem-kdv3-coucou}
\| y(\tau, \cdot)\|_{H^{-1}(0, L)} \le C_{T, T_1, \delta}  \Big( \| (h_1, h_2)
\|_{H^{-1/3}(\mR)} + \|h_3 \|_{H^{-2/3}(\mR)} + \| (f_1, f_2)\|_{L^1(\mR_+ \times (0, L))}
\Big),
\end{equation}
and assertion \eqref{lem-kdv3-cl2} follows by the standard $C^\infty$ smoothness property
of solutions of the linear KdV system \eqref{lem-kdv3-S}. The proof is complete.
\end{proof}

\begin{remark} One can check \eqref{lem-kdv3-coucou} by using a variant of
\eqref{lem-kdvB-cl3} in  \Cref{lem-kdvB} in which $f=0$ however,  a non-zero initial
condition is considered.
\end{remark}

\section{Small time local null-controllability  of the KdV system}\label{sect-NL-KdV}

The main result of this section is the following, which implies in particular
\Cref{thm-main}.

\begin{theorem}\label{thm-NL} Let $L > 0$, and  $k, l \in \N$. Set
\begin{equation}\label{def-p}
p = \frac{(2k + l)(k-l)(2 l + k)}{3 \sqrt{3}(k^2 + kl + l^2)^{3/2}}.
\end{equation}
Assume that
\begin{equation}\label{def-L}
L = 2 \pi \sqrt{\frac{k^2 + k l + l^2}{3}},
\end{equation}
and
\begin{equation}\label{cond-kl}
2 k + l \not \in 3 \N.
\end{equation}
Let $\Psi$ be defined in \eqref{def-Psi}, where
\begin{equation}\label{eta-kdv}
\eta_1 = - \frac{2 \pi i}{3 L} (2k + l), \quad \eta_2 = \eta_1 + \frac{2\pi i}{L} k, \quad
\eta_3 = \eta_2 + \frac{2\pi i}{L} l,
\end{equation}
and $E$ is given by \eqref{def-D}.  There exists $\eps_0 > 0$ such that for all $0< \eps <
\eps_0$,  for all  $0 < T < T_*/2$ \footnote{$T_*$ is the constant in \Cref{corollary-dir}
with $p$, $\eta_j$, and $L$ given previously. Note that $E \neq 0$ by \Cref{lem-E} below.}
and for all  solutions $y \in C\big([0, + \infty); H^2(0, L)\big) \cap L^2_{\loc}\big([0,
+ \infty); H^3(0, L) \big) $  of
\begin{equation}\label{sys-y-O}\left\{
\begin{array}{cl}
y_t (t, x) + y_x (t, x) + y_{xxx} (t, x) + y y_x (t, x) = 0 &  \mbox{ in } (0, + \infty)
\times  (0, L), \\[6pt]
y(t, x=0) = y(t, x=L) = 0 & \mbox{ in }  (0, + \infty), \\[6pt]
y_x(t , x= L) = u(t) & \mbox{ in } (0, \infty),  \\[6pt]
y(0, \cdot) = y_0 (x) : = \eps \Psi(0, \cdot),
\end{array}\right.
\end{equation}
with $(u \in H^{2/3}(\mR_+)$,  $\| u \|_{H^{2/3}(\mR)} < \eps_0$, $u(0) = 0$,  and $\supp
u \subset [0, T])$, we have
\[
y(T, \cdot) \neq 0.
\]
\end{theorem}

\begin{remark} \label{rem-Psi} With the choices of $p$ and $L$ in \Cref{thm-NL}, the
function $\Psi(t, x)$ given in \Cref{corollary-dir} satisfies
the linear KdV system as noted in \cite{Cerpa07}, i.e.,
\begin{equation}\label{rem-Psi1}
\Psi_t (t,x) + \Psi_{xxx} (t,x) + \Psi_x (t,x) = 0 \mbox{ in } \mR_+ \times (0, L),
\end{equation}
and
\begin{equation}\label{rem-Psi2}
\Psi(t, 0) = \Psi(t, L) = \Psi_x(t, 0) = \Psi_x(t, L) = 0 \mbox{ in } \mR_+.
\end{equation}
This property can be rechecked using the fact $\eta_1, \, \eta_2, \eta_3$ are the roots of
$\eta^3 + \eta - i p =0$.
\end{remark}

We first show that  $E$ defined by \eqref{def-D} with $\eta_j$ given in \eqref{eta-kdv}
and with  $p$ in \eqref{def-p} is not 0 if \eqref{cond-kl} holds. More precisely, we have

\begin{lemma} \label{lem-E} Let $k, \, l \in \N$ and let $E$ be given by \eqref{def-D}
with $\eta_j$  in \eqref{eta-kdv} and with  $p$ in \eqref{def-p}. Assume that
\eqref{def-L} holds.
We have
\begin{equation*}
E =  \frac{40 \pi^3}{3 L^3}  ( e^{\eta_1 L} - 1 ) i  kl (k+l).
\end{equation*}
Consequently,
\[
E \neq 0 \mbox{ provided that  \eqref{cond-kl} holds. }
\]
\end{lemma}

\begin{proof} With $\gamma_j = L \eta_j/ (2 \pi i)$, we have
\[
\gamma_1 = - \frac{2k + l}{3}, \quad \gamma_2 = \frac{k-l}{3}, \quad \gamma_3 = \frac{k+ 2
l}{3}.
\]
It follows that
\begin{align*}
\frac{L^3}{(2 \pi i)^3}\sum_{j=1}^3 \eta_{j+2}^2(\eta_{j+1}-\eta_{j})  =  &   \sum_{j=1}^3
\gamma_{j+2}^2(\gamma_{j+1}-\gamma_{j})
=   \gamma_3^2 k + \gamma_1^2 l  - \gamma_2^2 (k+l)
\\[6pt] = &(\gamma_3^2  - \gamma_2^2) k  - (\gamma_2^2 - \gamma_1^2)l =  (k+l) k l,
\end{align*}
which yields
\begin{equation*}
\sum_{j=1}^3 \eta_{j+2}^2(\eta_{j+1}-\eta_{j})  = - 8 \pi^3 i  k l (k+l) /L^3.
\end{equation*}
We also have
\begin{align*}
\sum_{j=1}^3  \frac{\eta_{j+1} - \eta_j}{\eta_{j+2}} =  &   \sum_{j=1}^3
\frac{\gamma_{j+1} - \gamma_j}{\gamma_{j+2}}
=  \frac{3k}{k + 2 l } - \frac{3 l }{2 k + l} - \frac{3 (k+l)}{k -l} = - \frac{2 7 k l
(k+l)}{(k+2l) (2 k + l) (k-l)}.
\end{align*}
We then have, by \eqref{def-D},
\begin{equation}\label{compute-E}
E = \frac{1}{3} (e^{\eta_1 L} - 1 )  \left( \frac{16 \pi^3 i }{3 L^3} kl (k+l) + \frac{2 7
i p kl (k+l)}{(k -l) (k+ 2l) (2l + k)} \right).
\end{equation}
From \eqref{def-p} and \eqref{def-L}, we have
\[
 \frac{p}{(k -l) (k+ 2l) (2l + k)} = \Big(\frac{2 \pi}{ 3 L} \Big)^3.
\]
We derive from \eqref{compute-E} that
\[
E = \frac{40 \pi^3}{3 L^3}  ( e^{\eta_1 L} - 1 ) i  kl (k+l).
\]
The proof is complete.
\end{proof}

Before giving the proof of \Cref{thm-NL}, we state and establish new estimates for the
nonlinear KdV system \eqref{intro-sys-KdV} and \eqref{intro-IC-KdV} which play a role in
the proof  of \Cref{thm-NL}.

\begin{lemma} \label{lem-kdvNL} Let $L > 0$ and $T>0$.  There exists a constant $\eps_0 >
0$  depending on $L$ and $T$ such that for $y_0 \in L^2(0, L)$ and for $u \in L^2(\mR_+)$
with
\[
\| y_0 \|_{L^2(0, L)} + \| u \|_{L^2(\mR_+)} \le \eps_0,
\]
then the unique solution $y \in C\big([0, + \infty); L^2(0, L) \big) \cap
L^2_{\loc}\big([0, + \infty); H^1(0, L) \big) $ of the system
\begin{equation*}\left\{
\begin{array}{cl}
y_t (t, x) + y_x (t, x) + y_{xxx} (t, x)  +  y (t, x) y_x (t, x)= 0 &  \mbox{ in } (0, + \infty)
\times  (0, L), \\[6pt]
y(t, x=0) = y(t, x=L) = 0 & \mbox{ in }  (0, + \infty), \\[6pt]
y_x(t , x= L) = u(t) & \mbox{ in } (0, \infty),
\end{array}\right.
\end{equation*}
with $y(0, \cdot) = y_0$,
satisfies
\begin{equation}\label{lem-kdvNL-p1}
\| y \|_{L^2\big((0, T) \times (0, L) \big)} \le C \Big( \| y_0 \|_{L^2(0, L)} + \| u
\|_{H^{-1/3}(\mR)} \Big),
\end{equation}
and
\begin{equation}\label{lem-kdvNL-p2}
\| y \|_{L^2\big((0, T); H^{-1} (0, L) \big)} \le C \Big( \| y_0 \|_{L^2(0, L)} + \| u
\|_{H^{-2/3}(\mR)} \Big),
\end{equation}
where $C$ is a positive constant depending only on  $T$ and  $L$.
\end{lemma}

\begin{proof}[Proof of \Cref{lem-kdvNL}] We have, see e.g. \cite[Proposition 14]{CC04} for
$\eps_0$ small,
\begin{equation*}
\| y_x \|_{L^2 \big( (0, T) \times (0, L) \big)} \le C_T \Big(\| y_0 \|_{L^2(0, L)} + \| u
\|_{L^2(\mR_+)} \Big),
\end{equation*}
which yields
\begin{equation}\label{lem-kdvNL-p1-cc}
\|y_x \|_{L^2 \big( (0, T) \times (0, L) \big)} \le C  \eps_0.
\end{equation}

Set
\[
f (t, x) =  - y (t, x) \partial_x y (t, x).
\]
The Cauchy–Schwarz inequality and \eqref{lem-kdvNL-p1-cc} yield
\[
\|f \|_{L^1(\mR_+ \times (0, L))} \le C \eps_0 \| y \|_{L^2\big(\mR_+ \times (0, L)
\big)}.
\]
Applying \Cref{lem-kdv3}, and more precisely~\eqref{lem-kdv3-cl1}, we have
\begin{equation*}
\| y \|_{L^2(\mR_+ \times (0, L))} \le C \eps_0 \| y \|_{L^2\big(\mR_+ \times (0, L)
\big)} + C \Big( \|y_0 \|_{L^2(0, L)} + \| u \|_{H^{-1/3}(\mR)} \Big).
\end{equation*}
By choosing $\eps_0$ sufficiently small, one can absorb the first term of the RHS by the
LHS and assertion \eqref{lem-kdvNL-p1} follows.

To prove \eqref{lem-kdvNL-p2}, one notes
\[
\|y^2 \|_{L^1\big( (0, T) \times (0, L) \big)} \le C \| y \|_{L^2\big( (0, T); H^{-1} (0,
L) \big)}  \| y \|_{L^2\big( (0, T); H^{1} (0, L) \big)}
\mathop{\le}^{\eqref{lem-kdvNL-p1-cc}} C \eps_0 \| y \|_{L^2\big( (0, T); H^{-1} (0, L)
\big)}.
\]
By  \Cref{lem-kdv3} (this time Eq.~\eqref{lem-kdv3-cl1-*}), we obtain
\begin{equation*}
\| y \|_{L^2( (0, T);  H^{-1} (0, L))} \le C \eps_0 \| y \|_{L^2\big((0, T); H^{-1}(0, L)
\big)} + C \Big( \|y_0 \|_{L^2(0, L)} + \| u \|_{H^{-2/3}(\mR)} \Big).
\end{equation*}
By choosing $\eps_0$ sufficiently small, one can absorb the first term of the RHS by the
LHS and assertion \eqref{lem-kdvNL-p2} follows.
\end{proof}

We are ready to give the

\begin{proof}[Proof of \Cref{thm-NL}] By \Cref{lem-E}, the constant $E$ is not $0$.  Let
$\eps_0$ be a small positive constant, which depends only on $k$ and $l$ and is determined
later.  We prove \Cref{thm-NL} by contradiction.    Assume that there exists a solution $y
\in C\big([0, + \infty); H^2(0, L) \big) \cap L^2_{\loc}\big([0, + \infty); H^3(0, L)
\big) $ of
\eqref{sys-y-O}  with $y(t, \cdot) =0$ for $t \ge T$,  for some $u \in H^{2/3}(0, +
\infty)$, for some $0< \eps < \eps_0$, and for some $0< T < T_*/2$  with   $\| u
\|_{H^{2/3}(\mR_+)} < \eps_0$, $u(0) = 0$,  and $\supp u \subset [0, T]$.

We have, for $\eps_0$ small, see e.g., \cite[Proposition 14]{CC04},
\begin{equation}\label{thm-NL-yyy}
\| y\|_{L^2 \big( (0, T); H^1(0, L) \big)} \le C \Big( \| y_0 \|_{L^2(0, L)} + \| u
\|_{L^2(\mR_+)} \Big).
\end{equation}

Set
\begin{equation}\label{thm-def-y1}
y_1 (t, x) = y  (t, x) -  c \int_{0}^L y(t, \eta) \Psi(t, \eta) \diff  \eta \,   \Psi(t,
x),
\end{equation}
with $\dsp c^{-1}  := \int_0^L |\Psi(0, \eta)|^2 \diff  \eta $. Since $y_0(x) = 
\epsilon \Psi(0,x)$, this choice of $c$ ensures that $y_1(0,
\cdot) = 0$ in $(0, L)$.  Then $y_1 \in C\big([0, + \infty); L^2(0, L) \big) \cap
L^2_{\loc}\big([0, + \infty); H^1(0, L) \big) $ is the solution of
\begin{equation*}\left\{
\begin{array}{cl}
y_{1, t} (t, x) + y_{1, x} (t, x) + y_{1, xxx} (t, x)  + f (t, x)= 0 &  \mbox{ in } (0, +
\infty) \times (0, L), \\[6pt]
y_1(t, x=0) = y_1(t, x=L) = 0 & \mbox{ in } (0, + \infty), \\[6pt]
y_{1, x} (t , x= L) = u(t) & \mbox{ in }  (0, +\infty), \\[6pt]
y_1(0, \cdot)  = 0,
\end{array}\right.
\end{equation*}
where
\[
f(t, x) = f_{1} (t, x) + f_{2, x} (t, x),
\]
with
\[
f_1(t, x) = - c \int_0^L y y_x (t, \eta) \Psi (t, \eta) \diff  \eta \,   \Psi(t, x) =
\frac{c}{2} \int_0^L y^2(t, \eta) \Psi_x (t, \eta) \diff  \eta \,   \Psi(t, x),
\]
and
\[
f_2(t, x) = \frac{1}{2} y^2 (t, x).
\]
By \Cref{lem-kdvNL}, we have
\begin{equation}\label{thm-NL-y}
\| y \|_{L^2\big( (0, T) \times (0, L) \big)}  \le C \Big(\| y_0\|_{L^2(0, L)} + \| u
\|_{H^{-1/3}(\mR)} \Big),
\end{equation}
and
\begin{equation}\label{thm-NL-yy1-p1}
\|y\|_{L^2\big( (0, T); H^{-1} (0, L) \big)} \le C\Big(\| y_0 \|_{L^2(0, L)} + \|
u\|_{H^{-2/3} (\mR)} \Big).
\end{equation}
From the definition of $y_1$ in \eqref{thm-def-y1},  and \eqref{thm-NL-yy1-p1}, after
applying \Cref{lem-kdv3} to $y - y_1$, we obtain
\begin{equation}\label{thm-NL-yy1-p2}
\|y_1\|_{L^2\big( (0, T); H^{-1} (0, L) \big)} \le C\Big(\| y_0 \|_{L^2(0, L)} + \|
u\|_{H^{-2/3} (\mR)} \Big).
\end{equation}

Let $y_2 \in C\big([0, + \infty); L^2(0, L) \big) \cap L^2_{\loc}\big([0, + \infty);
H^1(0, L) \big)$ be the unique solution of
\begin{equation*}\left\{
\begin{array}{cl}
y_{2, t} (t, x) + y_{2, x} (t, x) + y_{2, xxx} (t, x)  = - f (t, x)&  \mbox{ in } (0,
+\infty) \times  (0, L), \\[6pt]
y_2(t, x=0) = y_2(t, x=L) = 0 & \mbox{ in }  (0, +\infty), \\[6pt]
y_{2, x} (t , x= L) = 0& \mbox{ in } (0, +\infty), \\[6pt]
y_2(0, \cdot) = 0,
\end{array}\right.
\end{equation*}
and let $y_3 \in C\big([0, + \infty); L^2(0, L) \big) \cap L^2_{\loc}\big([0, + \infty);
H^1(0, L) \big)$ be the unique solution of
\begin{equation*}\left\{
\begin{array}{cl}
y_{3, t} (t, x) + y_{3, x} (t, x) + y_{3, xxx} (t, x)  = 0&  \mbox{ in }  (0, +\infty)
\times  (0, L), \\[6pt]
y_3(t, x=0) = y_3(t, x=L) = 0 & \mbox{ in } (0, +\infty), \\[6pt]
y_{3, x} (t , x= L) = u (t) & \mbox{ in } (0, +\infty), \\[6pt]
y_3(0, \cdot) = 0.
\end{array}\right.
\end{equation*}
Then
\[
y_1 = y_2 + y_3.
\]

There exists $u_4 \in L^2(0, +\infty)$  such that $\supp u_4 \subset [2T_*/3, T_*]$,
\begin{equation*}
\| u_4\|_{L^2(0, + \infty)} \le C \| y_3(2T_*/3, \cdot)\|_{L^2(2T_*/3, T_*)},
\end{equation*}
and
\[
y_4(T_*, \cdot ) = 0,
\]
where $y_4 \in C\big([0, + \infty); L^2(0, L) \big) \cap L^2_{\loc}\big([0, + \infty);
H^1(0, L) \big)$ is the unique solution of
\begin{equation*}\left\{
\begin{array}{cl}
y_{4, t} (t, x) + y_{4, x} (t, x) + y_{4, xxx}(t, x)   = 0&  \mbox{ in } (2T_*/3, +
\infty) \times (0, L), \\[6pt]
y_4 (t, x=0) = y_4(t, x=L) = 0 & \mbox{ in } (2T_*/3, +\infty), \\[6pt]
y_{4, x} (t , x= L) = u_4 (t) & \mbox{ in } (2T_*/3, +\infty), \\[6pt]
y_4(T_*/2, \cdot) =  y_3(2T_*/3, \cdot).
\end{array}\right.
\end{equation*}
Such an $u_4$ exists since $y_3(2T_*/3, \cdot)$ is generated from zero at time $0$, see
\cite{Rosier97}.

Since $y_2(t, \cdot) + y_3 (t, \cdot) = 0$ for $t \ge T_*/2$, we have
\begin{equation*}
\| u_4\|_{L^2(0, + \infty)} \le  C \| y_2(2T_*/3, \cdot )\|_{L^2(0, L)},
\end{equation*}
which yields
\begin{multline}\label{thm-NL-u4}
\| u_4\|_{L^2(0, + \infty)} \mathop{\le}^{ \Cref{lem-kdv3}}  C \| (f_1, f_2)\|_{L^1\big(
\mR_+ \times (0, L)\big)} \\[6pt]
\le C \min \Big\{ \| y \|_{L^2 \big( (0, T) \times (0, L)\big)}^2, \| y\|_{L^2\big( (0,
T); H^1(0, L) \big)}  \| y\|_{L^2\big( (0, T); H^{-1}(0, L) \big)} \Big\} \\[6pt]
\mathop{\le}^{\eqref{thm-NL-yyy}, \eqref{thm-NL-y}, \eqref{thm-NL-yy1-p1}} C  \min \Big\{
\Big(\| y_0\|_{L^2(0, L)} + \| u \|_{H^{-1/3}(\mR)} \Big)^2, \eps_0 \Big(\| y_0\|_{L^2(0,
L)} + \| u \|_{H^{-2/3}(\mR)} \Big)  \Big\}.
\end{multline}

Let $\ty \in C\big([0, + \infty); L^2(0, L) \big) \cap L^2_{\loc}\big([0, + \infty);
H^1(0, L) \big)$ be the unique solution of
\begin{equation*}\left\{
\begin{array}{cl}
\ty_{t} (t, x) + \ty_{x} (t, x) + \ty_{xxx} (t, x)  = 0&  \mbox{ in } (0, +\infty) \times
(0, L), \\[6pt]
\ty(t, x=0) = \ty(t, x=L) = 0 & \mbox{ in } (0, +\infty), \\[6pt]
\ty_{x} (t , x= L) = u (t) + u_4 (t)& \mbox{ in } (0, +\infty), \\[6pt]
\ty(0, \cdot) = 0,
\end{array}\right.
\end{equation*}
Then, by the choice of $u_4$,
\[
\ty(t, \cdot) = 0 \mbox{ for } t \ge T_*.
\]

Multiplying the equation of $y$ with $\Psi(t, x)$, integrating by parts on $[0, L]$, and
using
\eqref{rem-Psi1} and \eqref{rem-Psi2}, we have
\begin{equation}\label{thm-NL-B}
\frac{d}{dt} \int_{0}^L y (t, x) \Psi(t, x)\diff x - \frac{1}{2} \int_0^L y^2 (t, x)
\Psi_x (t, x) \diff x = 0.
\end{equation}
Integrating \eqref{thm-NL-B} from 0 to $T$ and using the fact $y(T, \cdot) = 0$ yield
\begin{equation}\label{thm-NL-id}
\int_{0}^L y_0 (x)  \Psi(0, x) \diff x  + \frac{1}{2} \int_0^T \int_0^L y^2 (t, x)
\Psi_x(t, x) \diff x \diff t  = 0.
\end{equation}

It is clear that
\begin{multline}\label{thm-NL-yty}
\left| \int_{0}^{T} \int_0^L y^2 (t, x) \Psi_x (t, x)  \diff x \diff t  -
\int_{0}^{+\infty} \int_0^L \ty^2 (t, x) \Psi_x (t, x)  \diff x \diff t  \right| \\[6pt]
\le \left| \int_{0}^{T} \int_0^L y^2 (t, x) \Psi_x (t, x)  \diff x \diff t  - \int_{0}^{T}
\int_0^L y_1^2 (t, x) \Psi_x (t, x)  \diff x \diff t  \right| \\[6pt]
+ \left| \int_{0}^{+\infty} \int_0^L y_1^2 (t, x) \Psi_x (t, x)  \diff x \diff t  -
\int_{0}^{+\infty} \int_0^L \ty^2 (t, x) \Psi_x (t, x)  \diff x \diff t  \right|.
\end{multline}
We next estimate the two terms of the RHS.

We begin with the first term. We have
\begin{multline}\label{thm-NL-yy1-p3-0}
\left| \int_{0}^{T} \int_0^L y^2 (t, x) \Psi_x (t, x)  \diff x \diff t  - \int_{0}^{T}
\int_0^L y_1^2 (t, x) \Psi_x (t, x)  \diff x \diff t  \right|  \\[6pt]
\le C \| y - y_1 \|_{L^2\big( (0, T); H^1(0, L) \big)}  \| (y, y_1) \|_{L^2\big( (0, T);
H^{-1} (0, L) \big)}.
\end{multline}
By considering the system of $y-y_1$, we obtain
\begin{multline}\label{thm-NL-yy1-p3}
\| y - y_1 \|_{L^2\big( (0, T); H^{1} (0, L) \big)} \le C \Big( \| y_0 \|_{L^2(0, L)} + \|
f_1\|_{L^1\big( (0, T); L^2(0, L)  \big)}\Big) \\[6pt]
\le C \| y_0 \|_{L^2(0, L)}  +
C \| y \|_{L^2 \big((0, T) \times (0, L) \big)}^2  \mathop{\le}^{\eqref{thm-NL-y}} C \|
y_0 \|_{L^2(0, L)} + C\Big(\| y_0 \|_{L^2(0, L)} + \| u\|_{H^{-1/3} (\mR)} \Big)^2.
\end{multline}
Combining \eqref{thm-NL-yy1-p1}, \eqref{thm-NL-yy1-p2}, and \eqref{thm-NL-yy1-p3}, we
derive from \eqref{thm-NL-yy1-p3-0} that
\begin{multline}\label{thm-NL-yy1}
 \left| \int_{0}^{T} \int_0^L y^2 (t, x) \Psi_x (t, x)  \diff x \diff t  - \int_{0}^{T}
\int_0^L y_1^2 (t, x) \Psi_x (t, x)  \diff x \diff t  \right|  \\[6pt]
 \le C \eps_0 \| y_0 \|_{L^2(0, L)}  + C \Big(\| y_0 \|_{L^2(0, L)} + \| u\|_{H^{-2/3}
(\mR)} \Big)  \Big(\| y_0 \|_{L^2(0, L)} + \| u\|_{H^{-1/3} (\mR)} \Big)^2 .
\end{multline}

We next estimate the second term of the RHS of \eqref{thm-NL-yty}.
It is clear that
\begin{multline}\label{thm-NL-y1ty-p0}
 \left| \int_{0}^{+ \infty} \int_0^L y_1^2 (t, x) \Psi_x (t, x)  \diff x \diff t  -
\int_{0}^{+\infty} \int_0^L \ty^2 (t, x) \Psi_x (t, x)  \diff x \diff t  \right|  \\[6pt]
 \le C \| y_1 - \ty \|_{L^2\big((0, T_*); H^1(0, L) \big)} \Big(
 \| y_1 \|_{L^2\big((0, T_*); H^{-1}(0, L) \big)} + \| \ty \|_{L^2\big((0, T_*); H^{-1}(0,
L) \big)}  \Big).
\end{multline}
Consider the systems of $y_1 -y$ and $\ty$. We have
\begin{align}\label{thm-NL-y1ty-p0-1}
\| y_1 - \ty \|_{L^2\big((0, T_*); H^1(0, L) \big)}  \le &  C \Big( \| f \|_{L^1\big((0,
T); L^2(0, L) \big)} + \| u_4 \|_{L^2(0, T)} \Big)  \\[6pt]
 \mathop{\le}^{\eqref{thm-NL-u4}} & C \| y y_x \|_{L^1\big((0, T); L^2(0, L)\big)}  + C
\Big(\| y_0\|_{L^2(0, L)} + \| u \|_{H^{-1/3}(\mR)} \Big)^2 \nonumber \\[6pt]
 \mathop{\le}^{\eqref{thm-NL-yyy}} & C \Big( \| y_0 \|_{L^2(0, L)} + \| u\|_{L^2(\mR_+)}
\Big)^2 \nonumber,
\end{align}
and, by \Cref{lem-kdv3} and \eqref{thm-NL-u4},
\begin{equation}\label{thm-NL-y1ty-p0-2}
 \| \ty \|_{L^2\big((0, T_*); H^{-1}(0, L) \big)}
 \le C \| (u, u_4) \|_{H^{-2/3}(\mR)}
 \le C \Big( \| y_0 \|_{L^2(0, L)} + \| u\|_{H^{-2/3}(\mR)} \Big).
\end{equation}
Using  \eqref{thm-NL-yy1-p2}, \eqref{thm-NL-y1ty-p0-1}, and \eqref{thm-NL-y1ty-p0-2}, we
derive from
\eqref{thm-NL-y1ty-p0} that
\begin{multline}\label{thm-NL-y1ty}
 \left| \int_{0}^{+ \infty} \int_0^L y_1^2 (t, x) \Psi_x (t, x)  \diff x \diff t  -
\int_{0}^{+\infty} \int_0^L \ty^2 (t, x) \Psi_x (t, x)  \diff x \diff t  \right|  \\[6pt]
 \le C \Big( \| y_0 \|_{L^2(0, L)} + \| u\|_{L^2(\mR_+)} \Big)^2 \Big( \| y_0 \|_{L^2(0,
L)} + \| u\|_{H^{-2/3}(\mR)} \Big).
\end{multline}

Combining   \eqref{thm-NL-yty}, \eqref{thm-NL-yy1}, and \eqref{thm-NL-y1ty} yields
\begin{multline}\label{thm-NL-p4}
\left| \int_{0}^{T} \int_0^L y^2 (t, x) \Psi_x (t, x)  \diff x \diff t  -
\int_{0}^{+\infty} \int_0^L \ty^2 (t, x) \Psi_x (t, x)  \diff x \diff t  \right| \\[6pt]
\le C \eps_0 \| y_0 \|_{L^2(0, L)}  + C \Big(\| y_0 \|_{L^2(0, L)} + \| u\|_{H^{-2/3}
(\mR)} \Big) \Big(\| y_0 \|_{L^2(0, L)} + \| u\|_{L^2 (\mR_+)} \Big)^2.
\end{multline}

On the other hand, from \Cref{corollary-dir} and the choice of $y_0$, we have
\begin{multline}\label{thm-NL-p5}
 \int_{0}^L y_0 (x)  \Psi(0, x) \diff x + \frac{1}{2}  \int_{0}^{+\infty} \int_0^L \ty^2
(t, x) \Psi_x (t, x)  \diff x \diff t \\[6pt]
\ge C \Big( \|y_0\|_{L^2(0, L)} +  \| u + u_4\|_{H^{-2/3}(\mR)}^2 \Big).
\end{multline}
Using the fact
\[
\| u + u_4\|_{H^{-2/3}(\mR)}^2 \ge C  \| u\|_{H^{-2/3}(\mR)}^2   - C \| u_4
\|_{L^2(\mR)}^2
\mathop{\ge}^{\eqref{thm-NL-u4}} C   \| u\|_{H^{-2/3}(\mR)}^2-  C \Big(\| y_0 \|_{L^2(0,
L)} + \| u\|_{H^{-1/3} (\mR)} \Big)^4,
\]
 we derive from  \eqref{thm-NL-p5} that, for small $\eps_0$,
\begin{multline}\label{thm-NL-p6}
 \int_{0}^L y_0 (x)  \Psi(0, x) \diff x + \frac{1}{2}  \int_{0}^{\infty} \int_0^L \ty^2
(t, x) \Psi_x (t, x)  \diff x \diff t \\[6pt]
 \ge C \left( \|y_0\|_{L^2(0, L)} +  \| u\|_{H^{-2/3}(\mR)}^2 \right) - C \| u\|_{H^{-1/3}
(\mR)}^4 .
\end{multline}
Combining \eqref{thm-NL-id}, \eqref{thm-NL-p4}, and \eqref{thm-NL-p6} yields
\begin{align}
C \eps_0 \| y_0 \|_{L^2(0, L)} & + C \Big(\| y_0 \|_{L^2(0, L)} + \| u\|_{H^{-2/3} (\mR)}
\Big) \Big(\| y_0 \|_{L^2(0, L)} + \| u\|_{L^2 (\mR_+)} \Big)^2 \\[6pt]
& \mathop{\geq}^{\eqref{thm-NL-p4}} \left|\int_{0}^{T} \int_0^L y^2 (t, x) \Psi_x (t, x)  \diff x \diff t  -
\int_{0}^{+\infty} \int_0^L \ty^2 (t, x) \Psi_x (t, x)  \diff x \diff t  \right|  \nonumber\\[6pt]
& \mathop{\geq}^{\eqref{thm-NL-id}} \int_{0}^L y_0 (x)  \Psi(0, x) \diff x + \frac{1}{2}  \int_{0}^{\infty} \int_0^L 
\ty^2
(t, x) \Psi_x (t, x)  \diff x \diff t  \nonumber\\[6pt] 
& \mathop{\ge}^{\eqref{thm-NL-p6}} C \Big( \|y_0\|_{L^2(0, L)} +  \| u\|_{H^{-2/3}(\mR)}^2 - C \| u\|_{H^{-1/3}(\mR)}^4
\Big). \nonumber
\end{align}
It follows that, if $\eps_0$ is fixed but sufficiently small,
\begin{equation}\label{thm-NL-coucou}
 \| u\|_{H^{-1/3} (\mR)}^4 + \| u\|_{H^{-2/3} (\mR)} \| u\|_{L^2 (\mR_+)}^2  \ge C \|
u\|_{H^{-2/3}(\mR)}^2.
\end{equation}

We have
\begin{equation}\label{thm-NL-interpolation-0}
\| u \|_{H^{-1/3}(\mR)}^2 \le C  \| u\|_{L^2(\mR)} \| u\|_{H^{-2/3}(\mR)} \le C \eps_0 \|
u\|_{H^{-2/3}(\mR)},
\end{equation}
and
\begin{equation}\label{thm-NL-interpolation-2-1}
\| u \|_{L^{2}(\mR)}^2 \le C  \| u\|_{H^{-2/3}(\mR)} \| u\|_{H^{2/3}(\mR)},
\end{equation}
(recall that  we extended $u$ by 0 for $t < 0$).  Let $U$ be the even extension of $u \Big|_{\mR_+}$ in
$\mR$.
Applying the Hardy inequality for fractional Sobolev space $H^{2/3}(\mR)$ for $U$ after
noting that $U(0) = 0$, see e.g.  \cite[Theorem 1.1]{NgSq18} \footnote{We here apply
\cite[ii) of Theorem 1.1]{NgSq18} with $\gamma = - 2/3$, $\tau = p = 2$, $s = 2/3$, $a=1$,
$\alpha = 0$.}, we derive that
\[
\| |\cdot |^{-2/3} U (\cdot) \|_{L^2(\mR)} \le C \| U \|_{H^{2/3}(\mR)}.
\]
We have
\[
\| U \|_{H^{2/3}(\mR)} \le C \| u\|_{H^{2/3}(\mR_+)}.
\]
since $U$ is an even extension of $u$, and
\[
| U |_{H^{2/3}(\mR)}^2 \sim \int_{\mR} \int_{\mR} \frac{|U(s) - U(t)|^2}{|s - t|^{1 +
4/3}} \diff  s \diff t,\quad
| u |_{H^{2/3}(\mR)}^2 \sim \int_{\mR_+} \int_{\mR_+} \frac{|u(s) - u(t)|^2}{|s - t|^{1 +
4/3}} \diff  s \diff t.
\]
We derive that
\[
\| |\cdot|^{-2/3} u (\cdot) \|_{L^2(\mR)} \le C \| u \|_{H^{2/3}(\mR_+)}.
\]
Since
\begin{multline*}
| u |_{H^{2/3}(\mR)}^2 \sim  \int_{\mR} \int_{\mR} \frac{|u(s) - u(t)|^2}{|s - t|^{1 +
4/3}} \diff  s \diff t \\[6pt]
\mathop{\le}^{u(s) = 0, \;  s < 0}  \int_{\mR_+} \int_{\mR_+} \frac{|u(s) - u(t)|^2}{|s -
t|^{1 + 4/3}} \diff  x \diff y +  C \int_{\mR_+} \frac{|u(t)|^2}{t^{4/3}} \diff t \\[6pt]
\le C \| u\|_{H^{2/3}(\mR_+)}^2 +  C \int_{\mR_+} \frac{|u(t)|^2}{t^{4/3}} \diff t,
\end{multline*}
it follows that
\begin{equation}\label{thm-NL-interpolation-2-2}
\| u \|_{H^{2/3}(\mR)} \le C \| u \|_{H^{2/3}(\mR_+)}.
\end{equation}
Here we also used the fact $u = 0 $ in $\mR_-$.  Combining
\eqref{thm-NL-interpolation-2-1} and \eqref{thm-NL-interpolation-2-2} yields
\begin{equation}\label{thm-NL-interpolation-2}
\| u \|_{L^{2}(\mR)}^2 \le C \eps_0 \| u\|_{H^{-2/3}(\mR)}. 
\end{equation}

Using \eqref{thm-NL-interpolation-0} and \eqref{thm-NL-interpolation-2}, we derive from
\eqref{thm-NL-coucou} that, $\|u\|_{H^{-2/3}}^2 \leq C\epsilon_0^2 \|u\|_{H^{-2/3}}^2 + 
C\epsilon_0 \|u\|_{H^{-2/3}}^2$. So, for fixed sufficiently small $\eps_0$,
\[
u =0.
\]
As a consequence, we obtain
\[
\|y(t, \cdot) - \eps \Psi(T_*/2, \cdot) \|_{L^2(0, L)} \le C \eps^2. 
\]
One has a contradiction if $\eps_0$ is sufficiently small.  The proof is complete.
\end{proof}

\begin{remark} Viewing the proof of \Cref{thm-NL}, it is natural  to ask whether or not
one needs to derive estimates for the (linear and nonlinear) KdV systems using low regular
data. In fact, without using these estimates, one might require that $\| u\|_{H^2(0, T)}$
or even $\| u\|_{H^3(0, T)}$ is small.
\end{remark}

\section{Controllability of the KdV system with controls in \texorpdfstring{$H^1$}{H¹}}
\label{sect-CP}

For $T>0$, set
\[
X = C\big([0, T]; Y \big) \cap L^2\big((0, T); H^4([0, L]) \big)
\]
with the corresponding norm. Here we denote
\[
Y = H^3(0, L) \cap H^1_0(0, L),
\]
which is a Hilbert space with the corresponding scalar product.

\medskip
In this section, we prove the following local  controllability of the KdV system
\eqref{intro-sys-KdV} and \eqref{intro-IC-KdV}:

\begin{theorem}\label{thm-CP} Let $L > 0$, and  $k, l \in \N$. Let $p$ be defined by
\eqref{def-p}. Assume  that  \eqref{def-L} holds,   $2k + l \not \in 3 \N$,  and the
dimension of $\M$ is 2.  Given  $T> \pi/ p$,  there exists $\eps_0 > 0$ such that for
$y_0, y_T \in Y$ with
\[
\| (y_0, y_T) \|_{Y} \le \eps_0,
\]
there exists $u \in H^{1}(0, T)$ such that  $u(0) = y_0'(L)$,
\[
\| u\|_{H^{1}(0, T)} \le C  \| (y_0, y_T) \|_{Y}^{1/2},
\]
and the corresponding solution  $y \in X$ of the nonlinear system \eqref{intro-sys-KdV}
with $y(t = 0, \cdot) = y_0$ satisfies $y(t = T, \cdot) = y_T$.
\end{theorem}

\medskip
We recall a result in \cite{Bona03} (\cite[Lemma 3.3]{Bona03} applied to $s=3$) on the
well-posedness and the stability of the linearized system of \eqref{intro-sys-KdV}.

\begin{lemma}\label{lem-Bona03} Let $L >0$ and $T>0$. For $y_0 \in H^3(0, L) \cap H^1_0(0,
L)$,   $f \in W^{1,1}\big([0, T]; L^2(0, L) \big)$,  and $u \in H^1(0, T)$ with $u(0) =
y_0'(L)$. There exists a unique solution $y \in X$ of the system
\begin{equation}\label{Sys-y}\left\{
\begin{array}{cl}
y_{t} (t, x) + y_{x} (t, x) + y_{xxx} (t, x) = f (t, x) &  \mbox{ for } t \in
(0, T), \, x \in (0, L), \\[6pt]
y(t, x=0) = y(t, x=L) = 0 & \mbox{ for } t \in (0, T), \\[6pt]
y_{x}(t , x= L) = u(t) & \mbox{ for } t \in (0, T),\\[6pt]
y(t = 0 , \cdot) = y_0 & \mbox{ for } x \in (0, L).
\end{array}\right.
\end{equation}
Moreover,
\[
\| y \|_{X} \le C \Big( \|f \|_{W^{1,1}\big([0, T]; L^2(0, L) \big)} + \| u\|_{H^1(0, 1)}
\Big),
\]
for some positive constant $C$ depending only on $L$ and $T$.
\end{lemma}

\begin{remark} \label{rem-Bona03} By the same method,  the  conclusion also holds for
the non-linear KdV equations if
$
\| f\|_{W^{1, 1} \big( (0, T); L^2(0, L) \big)} + \| u_0\|_{H^1(0, L)}
$
is small.
\end{remark}

\medskip
In what follows in this section,  $\M^\perp$ denotes all elements of $Y$ orthogonal to
$\M$ with respect to  $L^2(0, L)$-scalar product. We also denote  $P_{\M}$ and
$P_{\M^\perp}$  the projections into $\M$ and $\M^{\perp}$ with respect to $L^2(0, L)$-scalar product.  Before giving the proof of \Cref{thm-CP}, let us establish two lemmas
used in its proof. The first one is a consequence of the Hilbert Uniqueness Method for
controls in $H^1$ and  solutions in $X$.

\begin{lemma}\label{lem-HUM} Let $L \in \cN$ and $T>0$.  There is a continuous linear map
$\cL : \M^\perp \to H^1(0, T)$ such that for $\varphi \in \M^\perp$  and $u =
\cL(\varphi)$, then $u(0) = 0$, and  the unique solution  $y \in X$ of
\begin{equation}\label{Sys-KdV-LN}\left\{
\begin{array}{cl}
y_t (t, x) + y_x (t, x) + y_{xxx} (t, x) = 0 &  \mbox{ for } t
\in (0, T), \, x \in (0, L), \\[6pt]
y(t, x=0) = y(t, x=L) = 0 & \mbox{ for } t \in (0, T), \\[6pt]
y_x(t , x= L) = u(t) & \mbox{ for } t \in (0, T), \\[6pt]
y(t=0, \cdot) = 0,
\end{array}\right.
\end{equation}
satisfies $y(T, \cdot) = \varphi$.
\end{lemma}

\begin{proof}  Set
\[
\M_1^\perp = \Big\{w \in   \M^\perp ; w_x(0) = 0 \Big\}.
\]
For $\psi \in \M^\perp_1$, by \Cref{lem-Bona03}, there exists a unique solution  $y^* \in
X$  of the backward KdV system
\begin{equation}\label{Sys-KdV-BW}\left\{
\begin{array}{cl}
y^*_t (t, x) + y^*_x (t, x) + y^*_{xxx} (t, x) = 0 &  \mbox{ for } t
\in (0, T), \, x \in (0, L), \\[6pt]
y^*(t, x=0) = y^*(t, x=L) = 0 & \mbox{ for } t \in (0, T), \\[6pt]
y^*_x(t , x= 0) = 0 & \mbox{ for } t \in (0, T),\\[6pt]
y^*(T, \cdot) = \psi.
\end{array}\right.
\end{equation}
Applying the observability inequality to $y^*$ and $y^*_t$ (see e.g. \cite[Theorem 2.4]{Cerpa07} and also \cite[the proof of Proposition 3.9]{Rosier97}),  we have, for $\gamma \ge 1$,
\[
\int_{T/2}^T \gamma |y^*_x(t, L) \eta |^2  + |y^*_{tx} (t, L)|^2  \diff t \ge C \int_0^L
\gamma |y^*(T, x)|^2 + |y^*_t(T, x)|^2 \diff x,
\]
where in the last inequality, we used the fact that  if $\psi \in \M^\perp$ then $\psi''' +
\psi' $ is also in $\M^\perp$ (this can be proved through integration by part arguments; recall
that $\M^\perp$ is defined via $L^2(0, L)$-scalar product). In other words,
\begin{equation}\label{lem-HUM-p1}
\int_{T/2}^T \gamma  |y^*_x(t, L)  |^2  + |y^*_{tx} (t, L)|^2  \diff t \ge C \int_0^L
\gamma |\psi|^2 + |\psi''' + \psi'|^2 \diff x.
\end{equation}

Fix  a non-negative function $\eta \in C^1([0, T])$ such that $\eta = 1$ in $[T/2, T]$ and $\eta = 0$ in $[0,
T/3]$. Since
\[
\int_0^L \gamma |\psi|^2 + |\psi''' + \psi'|^2 \diff x = \int_0^L \gamma |\psi|^2 +
|\psi'''|^2  + |\psi'|^2 + 2 \psi''' \psi' \diff x,
\]
and, for all $\eps > 0$,
\[
\int_{0}^L |\psi'|^2 \, dx \le \int_{0}^L   \eps |\psi'''|^2 + C_\eps |\psi|^2 \, d x,
\]
it follows that, for large $\gamma$,
\begin{equation}\label{lem-HUM-p1-1}
\int_0^L \gamma |\psi|^2 + |\psi''' + \psi'|^2 \diff x \ge  C \|\psi \|_{H^3(0, L)}^2.
\end{equation}
We have
\[
\int_0^T |y^*_x (t, L) y^*_{tx} (t, L)| \, dt  \le  \int_0^T \eps^{-1}|y^*_x|^2 + \eps |
y^*_{tx}|^2 \, dt  \le C \int_0^L  \eps^{-1}|\psi|^2 + \eps |\psi''' + \psi'|^2 \, dx. 
\]
Here in the last inequalitiy, we applied \cite[(58) in the proof of Proposition
3.7]{Rosier97}  (see also \cite[Proposition 2]{Cerpa07}) to $y^*$ and $y^*_t$.  It follows
from \eqref{lem-HUM-p1} and \eqref{lem-HUM-p1-1},   for $\gamma$ large enough, that
\begin{equation}\label{lem-HUM-p1-2}
\int_{0}^T \gamma \eta (t)  |y^*_x(t, L)|^2  + y^*_{tx} (t, L) \big( \eta y^*_x(t, L)
\big)_t    \diff t \ge C_\gamma \|\psi \|_{H^3(0, L)}^2.
\end{equation}

For a given $\varphi \in \M^\perp_1$, by the Lax-Milgram's theorem and \eqref{lem-HUM-p1-2}, there exists a unique $\Phi \in \M^\perp_1$ such that
\begin{equation}\label{lem-HUM-p1-3}
\int_{0}^L \gamma \varphi \psi + (\varphi''' + \varphi') (\psi''' + \psi') \, dx 
= \int_0^T  \gamma y^*_x  \eta  Y^*_x + y^*_{tx} ( \eta Y^*_x)_t  \, dt  \quad \forall \,  \psi \in \M^\perp_1,
\end{equation}
where $Y^*$ is the solution of \eqref{Sys-KdV-BW} with $\psi = \Phi$. 

Let $y \in X$ be the solution of \eqref{Sys-KdV-LN} with  $u (\cdot) = \cL_1(\varphi) = \eta  (\cdot) Y^*_x(\cdot, L)$. Then, by integration by parts,
\begin{equation}\label{lem-HUM-p2}
\int_{0}^L \gamma \psi y(T, \cdot) + (\psi''' + \psi') \big(y_{xxx}(T, \cdot) + y_x(T, \cdot) \big)  \, dx \\[6pt]
= \int_0^T  \gamma y^*_x  \eta  Y^*_x + y^*_{tx} ( \eta Y^*_x)_t  \, dt  \quad \forall \,  \psi \in \M^\perp_1.
\end{equation}
From \eqref{lem-HUM-p1-3} and \eqref{lem-HUM-p2}, we obtain
\[
\int_{0}^L \gamma \varphi \psi + (\varphi''' + \varphi') (\psi''' + \psi') = \int_{0}^L \gamma \psi y(T, \cdot) + (\psi''' + \psi') (y_{xxx}(T, \cdot) + y_x(T, \cdot)) \quad \forall \,  \psi \in \M^\perp_1.
\]
Since $y$ and $Y^*$  satisfies  system \eqref{Sys-KdV-LN} with the same $u$ for $t \in [T/2, T]$, it follows that $y(t, \cdot) - Y^*(t, \cdot) \in \M$ for  $t \in [T/2, T]$. In particular, $y(T, \cdot) \in \M^\perp_1$ since $Y^*(T, \cdot) \in \M^\perp_1$. 
Combining this with the fact that $\varphi \in  \M^\perp_1$, 
we then derive from \eqref{lem-HUM-p1-1} that
\[
y (T, \cdot) = \varphi.
\]

The conclusion for $2T$ (instead of $T$) is now  as follows. Fix $\zeta \in C^1([0,  2 T])$ with $\zeta(2 T) = 1$ and $\zeta(t) = 0$ for $t \le 5T/4$. For  $\varphi \in \M^\perp$, let $\ty^*$ be the unique solution of
\begin{equation*}\left\{
\begin{array}{cl}
\ty^*_t (t, x) + \ty^*_x (t, x) + \ty^*_{xxx} (t, x) = 0 &  \mbox{ for } t
\in (T, 2T), \, x \in (0, L), \\[6pt]
\ty^*(t, x=0) = \ty^*(t, x=L) = 0 & \mbox{ for } t \in (T, 2T), \\[6pt]
\ty^*_x(t , x= 0) =  \varphi_x(2T, 0) \zeta (t) & \mbox{ for } t \in (T, 2T),\\[6pt]
\ty^*(2T, \cdot) = \varphi.
\end{array}\right.
\end{equation*}
One can check that $\ty^*(T, \cdot) \in \M^\perp_1$. Set
\begin{equation}
\cL(\varphi) (t) = \left\{\begin{array}{cl} \ty^*_x(t, L) & \mbox{ for }  t \in (T, 2 T), \\[6pt]
\cL_1(\ty^*(T, \cdot)) (t) & \mbox{ for } t \in (0, T).
\end{array} \right.
\end{equation}
It is clear that $\cL(\varphi)  \in H^1(0, 2 T)$ since $\ty_x(\cdot, L) \in H^1(T, 2 T)$, $\cL_1(\ty^*(T, \cdot)) \in H^1(0, T)$, and $\cL_1(\ty^*(T, \cdot)) (T) = \ty^*_x (T, L)$,  and  that the corresponding solution at the time $2T$ is $\varphi$. The proof is complete.
\end{proof}

For $r > 0$ and an element $e \in Y$,  we denote $B_r(e)$ the ball in $Y$ centered at $e$ with radius $r$, and $\overline{B_r(e)}$ its closure in $Y$.  The second lemma is a consequence of the power series method and the information derived in \Cref{sect-dir,sect-NL-KdV}.

\begin{lemma} \label{lem-CP} Let $L > 0$, and  $k, l \in \N$. Let $p$ be defined by \eqref{def-p}. Assume  that  \eqref{def-L} holds,   $2k + l \not \in 3 \N$,  and the dimension of $\M$ is 2.  Let $T> \pi/ p$ and $0 < c_1 < c_2$. Fix $\varphi \in \M$ with $c_1 \le \| \varphi \|_{Y}  \le c_2$. There exist  a constant $0< c_3 < c_1/2$, and two maps $U_1: B_{c_3}(\varphi) \to H^1(0, T)$ and $U_2: B_{c_3}(\varphi) \to H^1(0, T)$ such that for $\psi \in B_{c_3}(\varphi) $,  $U_1(\varphi)(0) = U_2(\varphi)(0) = 0$, and  the unique solutions $y_1$ and $y_2$ in $X$ of the following two systems, with $u_1 = U_1(\varphi)$ and $u_2 = U_2(\varphi)$,
\begin{equation}\label{Sys-y1}\left\{
\begin{array}{cl}
y_{1, t} (t, x) + y_{1, x} (t, x) + y_{1, xxx} (t, x) = 0 &  \mbox{ for } t \in
(0, T), \, x \in (0, L), \\[6pt]
y_1(t, x=0) = y_1(t, x=L) = 0 & \mbox{ for } t \in (0, T), \\[6pt]
y_{1, x}(t , x= L) = u_1(t) & \mbox{ for } t \in (0, T),\\[6pt]
y_{1}(t = 0 , \cdot) = 0 & \mbox{ for } t \in (0, T),
\end{array}\right.
\end{equation}
\begin{equation}\label{Sys-y2} \left\{
\begin{array}{cl}
y_{2, t} (t, x) + y_{2, x} (t, x) + y_{2, xxx} (t, x) + y_1 (t,x) y_{1, x} (t, x)  = 0 &
\mbox{ for } t \in (0, T), \, x \in (0, L), \\[6pt]
y_2(t, x=0) = y_2(t, x=L)  = 0 & \mbox{ for } t \in (0, T), \\[6pt]
y_{2, x}(t , x= L) = u_2(t) & \mbox{ for } t \in (0, T), \\[6pt]
y_{1}(t = 0 , \cdot) = 0 & \mbox{ for } t \in (0, T),
\end{array}\right.
\end{equation}
satisfy
\[
y_1(T, \cdot) = 0 \quad \mbox{ and } \quad y_2(T, \cdot) = \psi.
\]
Moreover, for $\psi, \,  \widetilde \psi \in B_{c_3}(\varphi)$,
\begin{equation}\label{lem-CP-U1}
\| U_1(\psi) - U_1(\widetilde \psi) \|_{H^1(0, T)} \le C \|\psi - \widetilde \psi \|_{Y}
\end{equation}
and
\begin{equation}\label{lem-CP-U2}
\| U_2(\psi) - U_2(\widetilde \psi) \|_{H^1(0, T)} \le C \|\psi - \widetilde \psi \|_{Y},
\end{equation}
for some positive constant $C$ depending only on $L$, $T$,  $c_1$, and $c_2$.
\end{lemma}

\begin{proof} By \Cref{lem-E} and \Cref{corollary-dir}, for all $\tau > 0$, there exists  $v_1 \in H^2_0(0, \tau)$ such that if $y_1 \in X$ is the solution of \eqref{Sys-y1} with $u_1 = v_1$ and $y_2 \in X$ is the solution of \eqref{Sys-y2} with $u_2 = 0$ then
\[
y_2 (\tau, \cdot) \in \M \setminus \{0 \}.
\]
Since $c_3$ is small, $\dim \M = 2$, and $v_1 \in H^2_0(0, L)$, by using rotations (see also \cite[the proof of Proposition 13]{Cerpa07}) there exists $U_1(\psi)$ with $U_1(\psi) (0) = 0$ satisfying \eqref{lem-CP-U1} such that if
$y_1 \in X$ is the solution of \eqref{Sys-y1} with $u_1= U_1(\psi)$ and $\hy_2 \in X$ is the solution of \eqref{Sys-y2} with $u_2 = 0$ then
\[
\hy_2 = P_{\M} \psi.
\]
We then choose
\[
u_2 = \cL (\hy_2 - P_{\M} \psi) ,
\]
where $\cL$ is a map given by \Cref{lem-HUM}.
\end{proof}

We are ready to give the

\begin{proof}[Proof of \Cref{thm-CP}]

Fix $y_0, y_T \in Y$ with small norms. For simplicity of the presentation, we will assume that $\| y_0 \|_Y \le \| y_T \|_Y$ (the other case also follows from this case by e.g. reversing the time: $t \to T - t$ and noting that $y_x(\cdot, 0)$ is in $H^1(0, T)$; this can be derived by considering  the equation for $y_t$ \footnote{The compatibility condition is automatic.}).
Set $\rho = \| y_T\|_Y$ and assume that $\rho >0$ otherwise, one just takes the zero control and the conclusion  follows.

Let $w_0$ be the state at the time $T$ of the solution of the linear system \eqref{Sys-KdV-LN} with the zero control starting from $P_{\M} y_0$ at the time 0.  We first consider the case where
\begin{equation}\label{CP-rem}
\|P_{\M} y_T -  w_0 \|_{H^2(0, L)} \ge 2 c   \rho,
\end{equation}
for some small constant $c$ independent of $\rho$ and defined later.

Set
\[
\begin{array}{cccc}
 \hcG  \colon & Y \cap B_{c \rho} (y_T) & \to & H^{1}(0, T) \\[6pt]
& \varphi & \mapsto & \rho \bu_0 + \rho^{1/2} u_1 + \rho u_2.
\end{array}
\]
Here we decompose $\varphi$ as
\[
\varphi = P_{\M^\perp} \varphi  + P_{\M} \varphi,
\]
$\bu_0  \in H^1(0, T)$ is a control for which the corresponding solution $\by_0$ in $X$ of the linear system \eqref{Sys-KdV-LN} starting from $P_{\M^\perp} y_0 / \rho$ at $0$ and arriving  $P_{\M^\perp} \varphi / \rho$ at the time $T$, and  $u_1$ and $u_2$ are  controls for which the solutions $y_1 \in X$ and $y_2 \in X$ of the system \eqref{Sys-y1} and \big(\eqref{Sys-y2} with the initial data $P_{\M} y_0/ \rho$ instead of $0$\big) satisfies $y_1(T, \cdot) = 0$ and $y_2(T, \cdot) = P_{\M} \varphi / \rho $. Moreover, by \Cref{lem-HUM}, one can choose $\bu_0$ in such a way that $\bu_0 = \bu_0 (\varphi)$
is a Lipschitz function of $\varphi$ with  the Lipschitz constant  bounded by a positive constant independent of $\rho$,  and by \Cref{lem-CP}
one can choose $u_1 = u_1(\varphi)$ and  $u_2 = u_2 (\varphi)$ as  Lipschitz functions of $P_{\M} \varphi / \rho$ with  the Lipschitz constants  bounded by positive constants independent of $\rho$.

Set
\[
\begin{array}{cccc}
\cP\colon  &  \Big\{ w \in H^{1}(0, T); w(0) = y_0'(L) \Big\}   & \to & H^3(0, L) \\[6pt]
& w & \mapsto & y(T, \cdot),
\end{array}
\]
where $y \in X$ is the unique solution of the nonlinear system \eqref{intro-sys-KdV} with $u = w$ starting from $y_0$ at time $0$. Consider the map
\[
\begin{array}{cccc}
\Lambda\colon  &  Y \cap \overline{B_{c \rho} (y_T)} & \to & Y \\[6pt]
& \varphi & \mapsto & \varphi - \cP \circ \hcG (\varphi) +  y_T. 
\end{array}
\]
We will prove that
\begin{equation}\label{CP-claim1}
\Lambda (\varphi) \in \overline{B_{c \rho} (y_T)},
\end{equation}
and
\begin{equation}\label{CP-claim2}
\|\Lambda (\varphi) - \Lambda (\phi) \|_{Y} \le \lambda \| \varphi - \phi\|_Y,
\end{equation}
for some $\lambda \in (0, 1)$. Assuming this, one derives from the contraction mapping  theorem that there exists a unique $\varphi_0 \in Y \cap \overline{B_{c \rho} (y_T)} $ such that $\Lambda (\varphi_0) = \varphi_0$.  As a consequence,
\[
y_T = \cP \circ \hcG (\varphi_0),
\]
and $\hcG(\varphi_0)$ is hence a required control.

We next establish \eqref{CP-claim1} and \eqref{CP-claim2}. Indeed, assertion \eqref{CP-claim1} follows from the fact
\[
\|\varphi -  \cP \circ   \hcG (\varphi) \|_{Y} \le C \| \varphi \|_{Y}^{3/2} \mbox{ for } Y \cap \overline{B_{\rho/2} (y_T)}.
\]
This can be proved using the approximation via the power series method as follows. Set \footnote{The index $a$ stands the
approximation.}

\[
u = \rho \bu_{0} + \rho^{1/2} u_{1} + \rho u_{2}  \quad \mbox{ and } \quad y_a = \rho \by_0 + \rho^{1/2} y_1 + \rho y_2.
\]
Let $y \in X$ be the  solution of the nonlinear KdV system \eqref{intro-sys-KdV} with $y(t=0, \cdot) = y_0$ and with $u$ defined above. Then
\[
(y - y_a)_t + (y - y_a)_x + (y - y_a)_{xxx} + y y_x - y_a y_{a, x} =  f(t, x),
\]
where
\[
- f(t, x) = \rho^{3/2}  (y_1 y_2)_x   + \rho^2  y_2 y_{2, x}  + \rho^2 \by_0 \by_{0, x}  + \rho^{3/2} \Big( \by_0 (y_1 + \rho^{1/2} y_2) \Big)_x.
\]
Since
\[
y y_x - y_a y_{a, x} = (y- y_a) y_x + y_a (y_x - y_{a, x}),
\]
applying \Cref{lem-Bona03}, we obtain, for small $\rho$,
\begin{equation}\label{CP-yya}
\|y- y_a\|_X \le C \| f \|_{W^{1, 1} \big( (0, T); L^2(0, L) \big)} \le  C \rho^{3/2}.
\end{equation}
Assertion~\eqref{CP-claim1} follows since $y(T, \cdot) = \cP \circ \hcG (\varphi)$ and $y_a(T, \cdot) = \varphi$.

\medskip
We next establish \eqref{CP-claim2}. To this end, we estimate
\[
 \Big( \varphi -  \cP \circ \hcG (\varphi) \Big) - \Big( \tvarphi -  \cP \circ \hcG (\tvarphi) \Big).
\]
Denote $\tbu_{0}, \tu_{1}, \tu_{2}, \tu$ and $\tby_{0}, \ty_{1}, \ty_{2}, \ty_a, \ty$ the  functions corresponding to $\tvarphi$ which are defined in the same way as the functions  $\bu_{0}, u_{1}, u_{2}, u$ and $\by_{0}, y_{1}, y_{2}, y_a, y$ defined  for  $\varphi$.

We have
\[
(y - \ty)_{t} + (y - \ty)_{x} + (y - \ty)_{xxx} + y y_x - \ty \ty_x = 0,
\]
\[
(y_a - \ty_a)_{t} + (y_a - \ty_a)_{x} + (y_a - \ty_a)_{xxx} + y_a y_{a, x} - \ty_a \ty_{a,x} =  g (t, x),
\]
where
\begin{multline}
 g (t, x) = \rho^{3/2}  \Big( (y_1 y_2)_x - (\ty_1 \ty_2)_x \Big)  + \rho^2 \Big( y_2 y_{2, x} - \ty_2 \ty_{2, x} \Big) + \rho^2 \Big(\by_0 \by_{0, x} - \tby_0 \tby_{0, x} \Big) \\[6pt]
+ \rho^{3/2} \Big(\by_0 (y_1 + \rho^{1/2} y_2) - \tby_0 (\ty_1 + \rho^{1/2} \ty_2) \Big)_x.
\end{multline}
This implies
\begin{align*}
(y - & y_a - \ty + \ty_a)_{t} +  (y - y_a - \ty + \ty_a)_{x} + (y - y_a - \ty + \ty_a)_{xxx}  \\[6pt] =
& - \Big(  (y-y_a) y_x + y_a (y - y_a)_x - (\ty-\ty_a) \ty_x - \ty_a (\ty - \ty_a)_x + g(t, x) \Big)\\[6pt]
= & - \Big( (y - y_a - \ty + \ty_a) y_x  +  (y_x - \ty_x)(\ty - \ty_a)  + y_a (y - y_a - \ty + \ty_a)_x \\[6pt]
& \quad \quad \quad + (y_a - \ty_a)(\ty - \ty_a)_x  + g(t, x) \Big) \\[6pt]
= & - \Big( (y - y_a - \ty + \ty_a) y_x   + y_a (y - y_a - \ty + \ty_a)_x   + (y_x - y_{a, x} - \ty_x  + \ty_{a, x} )(\ty - \ty_a) 
+ h(t, x) \Big) , 
\end{align*}
where
\[
h(t, x) = g(t, x) + (y_{a, x} - \ty_{a, x})(\ty - \ty_a) +  (y_a - \ty_a)(\ty - \ty_a)_x. 
\]

Using \Cref{lem-Bona03}, we  derive that, for $\rho $ small,
\begin{equation}\label{CP-coucou}
\| y - y_a - \ty + \ty_a \|_{X}  \le C \| h(t, x) \|_{W^{1,1}\big((0, T); L^2(0, L) \big)}.
\end{equation}
We have
\begin{equation*}
\| (y-y_a, \ty - \ty_a) \|_X \mathop{\le}^{\eqref{CP-yya}} C \rho^{3/2},  \; \; 
\|y_a - \ty_a\|_X \le C \rho^{-1/2} \|\varphi - \tvarphi \|_Y, 
\end{equation*}
and
\[
\| g(t, x) \|_{W^{1,1}\big((0, T); L^2(0, L) \big)}  \le C \rho^{1/2} \| \varphi - \tvarphi \|_Y.
\]
It follows that
\begin{equation}
\| h(t, x) \|_{W^{1,1}\big((0, T); L^2(0, L) \big)} \le C \rho^{1/2} \| \varphi - \phi \|_Y, 
\end{equation}
which yields, by \eqref{CP-coucou}, 
\[
\| (y - y_a - \ty + \ty_a) (T, \cdot) \|_{Y} \le C \rho^{1/2} \| \varphi - \phi \|_Y. 
\]
Assertion~\eqref{CP-claim2} follows.

We next consider the case $\|P_{\M} y_T - w_0 \|_{H^3(0, L)} \le 2 c  \| y_T \|_{H^3(0, L)} $.  In fact,  one can bring this case to the previous case as follows. Fix $\eps > 0$ small. By \Cref{lem-E} and \Cref{corollary-dir},  there exists  $v_1 \in H^2_0(0, \eps)$ such that if $y_1 \in X$ (with $T = \eps$) is the solution of \eqref{Sys-y1} with $u_1 = v_1$ and $y_2 \in X$ is the solution of \eqref{Sys-y2} with $u_2 = 0$ then
\[
y_2 (\eps, \cdot) \in \M \setminus \{0 \}.
\]
Let $u_{0, T}$, $u_{1, T}$, $u_{2, T}$ be such that $u_{0, T}$ is 
 a control for which the corresponding solution in $X$ of the linear system \eqref{Sys-KdV-LN} starting from $y_T(L - \cdot) / \rho$ at $0$ and arriving  $0$ at the time $\eps$,  $u_{1, T} = \gamma v_1$, $u_{2, T} = \gamma^2 v_2$ for some $\gamma > 0$ defined later. Let $\by$ be the unique solution of the nonlinear KdV system in the time interval $[T, T+ \eps]$ using the control 
 $$
 \rho u_0 ( \cdot - T) + \rho^{1/2} u_1 (\cdot - T) + \rho u_2 (\cdot - T), 
 $$
with $\by (T, \cdot) = y_T(L  - \cdot )$. By choosing $\gamma$ large enough, $y_0$ and $\by (T + \eps, L - \cdot )$  satisfy the setting of the previous case for the time interval $[0, T + \eps]$ (instead of [0, T]). One now considers the control (for the nonlinear KdV system) in the time interval $[0, T + 2\eps]$ which is equal to the one which brings $y_0$ at the time 0 to $\by(T + \eps, L - \cdot)$ at the time  $T+ \eps$ obtained in the previous case in the time interval $[0, T+\eps]$,  and is equal to $- \by_x\big(2(T + \eps) - t, 0 \big)$ for $t \in [T+ \eps, T+ 2 \eps]$. It is clear that the solution of 
the nonlinear KdV system  at the time $T + 2 \eps$ is $y_T$.  The proof is complete by changing $T + 2 \eps$ to $T$. 
\end{proof}

\begin{remark}\label{Rem-H2/3} Similar result as the one in \Cref{thm-CP} also holds for  $y_0, y_T \in
H^2(0, L) \cap H^1_0(0, L)$ and $u \in H^{2/3}(0, T)$. More precisely, one has the
following result. Let $L > 0$, and  $k, l \in \N$. Let $p$ be defined by \eqref{def-p}.
Assume  that  \eqref{def-L} holds,   $2k + l \not \in 3 \N$,  and the dimension of $\M$ is
2.  Given  $T> \pi/ p$,  there exists $\eps_0 > 0$ such that for $y_0, y_T \in H^2(0, L)
\cap H^1_0(0, L)$ with
\[
\| (y_0, y_T) \|_{H^2(0, L)} \le \eps_0,
\]
there exists $u \in H^{2/3}(0, T)$ such that  $u(0) = y_0'(L)$,
\[
\| u\|_{H^{2/3}(0, T)} \le C  \| (y_0, y_T) \|_{H^2}^{1/2},
\]
and the corresponding solution  $y \in C\big([0, T]; H^2 (0, L) \big) \cap L^2\big((0, T); H^3[0, L]) \big) $ of the nonlinear system \eqref{intro-sys-KdV} with $y(t = 0, \cdot) = y_0$ satisfies $y(t = T, \cdot) = y_T$. This is complementary  to \Cref{thm-NL}. The only important modification in comparison with the proof of \Cref{thm-CP} is \Cref{lem-HUM}. Nevertheless, the method presented in its proof can be extended to cover the setting mentioned here (initial and final datum in $H^2(0, L) \cap H^1_0(0, 1)$ and controls in $H^{2/3}(0, T)$). We also have 
\begin{equation}\label{Rem-H2/3-p1}
\| y_x(\cdot, 0)\|_{H^{2/3}(0, T)} \le C \Big(  \|y(0, \cdot) \|_{H^2(0, L)} + \| y_x(\cdot, L) \|_{H^{2/3}(0, T)} \Big), 
\end{equation}
for solutions $y \in C\big([0, T]; H^2 (0, L) \big) \cap L^2\big((0, T); H^3[0, L]) \big)$ of \eqref{intro-sys-KdV} with small norm. Assertion~\eqref{Rem-H2/3-p1} would follow from \cite{Bona03} applied to $s=2$. Here is another way to see it. 
Split  $y$ into two parts $y_1$ and $y_2$ where $y_1$ is the solution of the linearized system with zero initial data and $y_{1, x} (\cdot, L) = y_x(\cdot, L)$.  As in the proof of \Cref{lem-kdv1}, one can prove 
\begin{equation}\label{Rem-H2/3-p2}
\| y_{1, x} (\cdot, 0)\|_{H^{2/3}(0, T)} \le C  \| y_x(\cdot, L) \|_{H^{2/3}(0, T)}. 
\end{equation}
Concerning $y_2$, by considering $y y_x$ as a source term, similar to the proof of \Cref{lem-kdv3},  one can prove 
\begin{equation}\label{Rem-H2/3-p3}
\| y_{2, x} (\cdot, 0)\|_{H^{2/3}(0, T)} \le C \Big(  \|y(0, \cdot) \|_{H^2(0, L)} +  \| y y_x \|_{L^2\big((0,T); H^2(0, L) \big)} \Big).  
\end{equation}
Since 
\begin{multline*}
\| y y_x \|_{L^2\big((0,T); H^2(0, L) \big)} \le C \| y \|_{C\big([0, T]; H^2 (0, L) \big) \cap L^2\big((0, T); H^3[0, L]) \big)}^2  \quad (\mbox{by the embedding theorem}) \\[6pt]
\le  C \Big(  \|y(0, \cdot) \|_{H^2(0, L)} + \| y_x(\cdot, L) \|_{H^{2/3}(0, T)} \Big)^2 \quad (\mbox{by \cite[Theorem 3.4]{Bona03} applied to $s=2$}), 
\end{multline*}
assertion \eqref{Rem-H2/3-p1} follows from  \eqref{Rem-H2/3-p2} and  \eqref{Rem-H2/3-p3}. Therefore, the arguments using the backward  systems also work in this case. 
\end{remark}

\begin{remark} The proof given in \Cref{thm-CP} can be extended easily to the case $L
\not \in \cN$ to yield the small-time local controllability of \eqref{intro-sys-KdV} with
initial final and initial datum in $H^3(0, L) \cap H^1_0(0, L)$ (resp. $H^{2}(0, L) \cap
H^1_0(0, L)$)
 and controls in $H^1(0, T)$ (resp. $H^{2/3}(0, T)$).
\end{remark}

\begin{remark}  Let $L \in \cN$. Assume that
$\dim \M$ is pair and for all $(k, l) \in \N^2$ with $k > l \ge 1$ and $L = \frac{1}{2 \pi}\sqrt{\frac{k^2 + l^2 + k l}{3}}$, it holds $2 k + l \not \in  3\N $. Then,
using the same method in the proof of \Cref{thm-CP}, and involving the ideas in \cite{CC09}, one can prove that the system \eqref{intro-sys-KdV} and \eqref{intro-IC-KdV}  is controllable at the time given in \cite{CC09}.
\end{remark}

\begin{remark} The mappings $\hcG$ and  $\Lambda$ have their roots in \cite{CC04} (see
also \cite{Cerpa07}). 
\end{remark}

\begin{remark}
\Cref{lem-HUM} is motivated by the Hilbert Uniqueness Method and inspired by the construction of smooth controls (for different contexts, e.g. the context of the wave equation) in \cite{EZ10}.  The function $\eta$ used there is inspired from  \cite{EZ10}. Nevertheless, we cannot take $\eta = 0$ near $T$ as in \cite{EZ10}.  We also add a large parameter $\lambda$ in the proof.
\end{remark}

\begin{remark}
In the proof of \Cref{lem-CP}, we use essentially the fact that  for all $\tau > 0$, there exists  $v_1 \in H^2_0(0, \tau)$ such that if $y_1 \in X$ is the solution of \eqref{Sys-y1} with $u_1 = v_1$ and $y_2 \in X$ is the solution of \eqref{Sys-y2} with $u_2 = 0$ then
\[
y_2 (\tau, \cdot) \in \M \setminus \{0 \}.
\]
This is a consequence of \Cref{lem-E} and \Cref{corollary-dir}. It is not clear for us how to use a contradiction argument as in \cite{CC04, Cerpa07, CC09} to obtain such a function $v_1$. This is why we cannot implement the strategy in \cite{CC04, Cerpa07, CC09} to derive the local controllability for initial and final datum in $H^3(0,L) \cap H^1_0(0, L)$ with controls in $H^1(0, T)$ for all critical lengths and for small time when $\dim \M = 1$ and for finite time otherwise.
\end{remark}

\begin{remark}  We emphasize that the way to implement the fixed point argument for
$\Lambda$ presented in this paper is somehow different from the one in \cite{Cerpa07}. We only apply  the fixed point arguments once  instead of twice, first for
$P_{\M^\perp} \Lambda$ and then for $P_{\M} \Lambda$ as in \cite{Cerpa07}. The Brouwer fixed point theorem is not required in our analysis.
\end{remark}

\appendix

\section{On symmetric functions of the roots of a
polynomial}\label{sec:symmetric_holomorphic}
This is standard for people knowing algebraic functions~\cite[Ch.~8 \S2]{Ahlfors78}, but
for the sake of completeness, we justify that an analytic symmetric function of the roots
$\lambda_j(z)$ of $\lambda^3 + \lambda + iz = 0$ is an entire function.
\begin{lemma}\label{pro-S}
  Let $(\lambda_1(z),\lambda_2(z),\lambda_3(z))$ be the three roots of $\lambda^3 + \lambda +
i z = 0$. Let $F\colon \C^3 \to \C$ be holomorphic in $\C^3$ and
symmetric, i.e., for every permutation $\sigma\in \mathfrak S_3$,
$F(z_{\sigma(1)},z_{\sigma(2)},z_{\sigma(3)}) = F(z_1,z_2,z_3)$. Then, the function
$G\colon z\in\C\mapsto F(\lambda_1(z),\lambda_2(z),\lambda_3(z))$ is entire.
\end{lemma}

Note that the ordering $\lambda_1(z),\lambda_2(z),\lambda_3(z)$ is not unique (and we
could prove that we cannot chose an ordering that makes any of the $\lambda_j$ entire), but
since $F$ is symmetric, the value $F(\lambda_1(z),\lambda_2(z),\lambda_3(z))$ does not
depend on the ordering.

\begin{proof} Note that, for $z_0 \neq  \pm 2/ (3 \sqrt{3})$, the discriminant of
$X^3+X+iz$ is nonzero, and thus the roots of $X^3+X+i z_0$ are  simple. By the
implicit function theorem,  there exists some complex neighborhood $U$ of $z_0$,  some
neighborhood $V_j$ of $\lambda_j(z_0)$ ($1 \le j \le 3$),  and three
holomorphic functions $\mu_j\colon U\to V_j$ such that $\mu_1(z), \, \mu_2(z), \mu_3(z)$
are the three distinct roots. Since $F$ is symmetric, it follows that  $G(z) =
F(\mu_1(z),\mu_2(z),\mu_3(z))$ and  is therefore analytic  in $U$. Consequently, $G$ is
analytic in $\mC \setminus \{ \pm 2/ (3 \sqrt{3}) \}$.

It suffices then to prove that $G$ is continuous at $\pm 2/3\sqrt 3$. The roots
$\lambda_j(z)$ are continuous, even around at $\pm\sqrt{4/27}$, in the sense that for every $\epsilon>0$, there exists $\delta>0$ such that for every $|z-z_0|<\delta$, there exists some ordering of the $\lambda_{k_j}(z)$ such that
$|\lambda_{k_1}(z) - \lambda_1(z_0)| + \dots + |\lambda_{k_3}(z) -
\lambda_3(z_0)|<\epsilon$ (this can be seen e.g.\ thanks to Cardano's formula). Thus
$G(z)$ is continuous at $z_0 = \pm \sqrt{4/27}$ and $\pm \sqrt{4/27}$. 
\end{proof}


\begin{remark} A variant of \Cref{pro-S} still holds for more general polynomial
equations
$P(z,\lambda) = 0$, but we wanted to avoid some technicalities of such a general
equation. The general case would be a consequence of the fact that the solutions of
$P(z,\lambda)=0$ define a finite number of algebraic functions, see~\cite[Ch.~8
\S2]{Ahlfors78}.
\end{remark}

\section{On the real roots of \texorpdfstring{$H$,  the common roots of $G$ and $H$,  and
the behavior  of $|\det
Q|$}{H, the common roots of G and H, and the beavior of |det Q|}}\label{B}

We begin with

\begin{lemma}\label{lem-detQ} Let $z \in \mR$. We have 
\begin{enumerate}
\item[1)] if $z \neq  \pm 2 / (3 \sqrt{3})$ and $H(z) = 0$,  then, for some $k, l \in \N$ with $1 \le l \le k$,
$L = 2 \pi \sqrt{\frac{k^2 + k l + l^2}{3}}$, and
\begin{equation}\label{lem-detQ-cl1}
z = - \frac{(2k + l)(k-l)(2 l + k)}{3 \sqrt{3}(k^2 + kl + l^2)^{3/2}}.
\end{equation}
Moreover,
\begin{equation}\label{lem-detQ-cl2}
\lambda_1(z) = - \frac{2 \pi i}{3 L} (2k + l), \quad \lambda_2(z) = \lambda_1(z) + \frac{2\pi i}{L} k, \quad \lambda_3 (z) = \lambda_2(z) + \frac{2\pi i}{L} l,
\end{equation}
and $z$ is a simple zero of the equation $H$.

\item[2)] if $z = \pm 2/ (3 \sqrt{3})$ then
\begin{equation}\label{lem-detQ-cl3}
 \lambda_1(z) =
 \mp \frac{i}{\sqrt 3},  \quad
 \lambda_2(z) =
 \mp \frac{i}{\sqrt 3}, \quad
 \lambda_3(z) = \pm \frac{2i}{\sqrt 3},
\end{equation}
 $z$ is not a zero of $H$, and $z$ is a simple solution of the equation $\det Q(z) \Xi 
(z) = 0$. 
\end{enumerate}
\end{lemma}

\begin{proof} We begin with 1). By \Cref{rem-detQ-realroots}, assertion~\eqref{lem-detQ-cl1} holds. Assertion~\eqref{lem-detQ-cl2} then follows from \cite{Rosier97}. To prove that $z$ is then a simple root of the equation $H(z) = 0$ in the case $z \neq \pm 2 / (3 \sqrt{3})$, we proceed as follows. We have
$$
\lambda_j(z + \eps) = \lambda_j (z) - \frac{i \eps}{3 \lambda_j^2 + 1} + O(\eps^2).
$$
It follows that
\begin{multline*}
\det Q(z+\eps)  =  \sum_{j=1}^3 \big(\lambda_{j+1} (z + \eps)  - \lambda_j(z + \eps) \big) e^{-\lambda_{j+2} (z + \eps) L } \\[6pt]
=   \sum_{j=1}^3 \Big(\lambda_{j+1} (z)  - \lambda_j(z)  - \frac{i \eps}{3 \lambda_{j+1}^2 + 1} + \frac{i \eps}{3 \lambda_j^2 + 1} + O(\eps^2)\Big) e^{-\lambda_{j+2} (z) L } \Big( 1 +  \frac{i \eps L }{3 \lambda_{j+2}^2 + 1} + O(\eps^2) \Big).
\end{multline*}
Since
$$
e^{-\lambda_{1} (z) L } = e^{-\lambda_{2} (z) L }  = e^{-\lambda_{3} (z) L },
$$
we derive that
\begin{equation}\label{lem-detQ-p1}
\det Q(z+\eps) = i \eps L e^{-\lambda_{1} (z) L } \sum_{j=1}^3 \frac{ \lambda_{j+1} (z) - \lambda_j (z) }{ 3 \lambda_{j+2}^2 (z) + 1} + O(\eps^2).
\end{equation}
In what follows, for notational ease, we denote $\lambda_j(z)$ by $\lambda_j$. We have
\begin{multline}\label{lem-detQ-p2}
 \sum_{j=1}^3 \frac{ \lambda_{j+1} - \lambda_j  }{ 3 \lambda_{j+2}^2 + 1}  = \frac{2 \pi i }{ L } \left(  \frac{k}{3 \lambda_3^2 + 1} + \frac{l}{3 \lambda_1^2 + 1}  - \frac{k+l}{3 \lambda_2^2 + 1}\right) \\[6pt]
 = \frac{2 \pi i }{ L } \left(  \frac{3k (\lambda_2^2 - \lambda_3^2)}{(3 \lambda_3^2 + 1)(3 \lambda_2^2 + 1)} +   \frac{3l (\lambda_2^2 - \lambda_1^2) }{(3 \lambda_1^2 + 1)(3 \lambda_2^2 + 1)}\right) \\[6pt]
=  \left( \frac{2 \pi i }{ L } \right)^2 \left(  - \frac{3kl  (\lambda_2 + \lambda_3) }{(3 \lambda_3^2 + 1)(3 \lambda_2^2 + 1)} +   \frac{3 kl (\lambda_2+ \lambda_1) }{(3 \lambda_1^2 + 1)(3 \lambda_2^2 + 1)}\right).
\end{multline}
Note that
\begin{multline}\label{lem-detQ-p3}
(\lambda_2 + \lambda_1) (3 \lambda_3^2 + 1) - (\lambda_2 + \lambda_3) (3 \lambda_1^2 + 1) \\[6pt]
= (\lambda_1 - \lambda_3) + 3 (\lambda_3 - \lambda_1) (\lambda_1 \lambda_2 + \lambda_1 \lambda_3 + \lambda_2 \lambda_3)  = 2 (\lambda_3 - \lambda_1),
\end{multline}
since $\lambda_1 \lambda_2 + \lambda_1 \lambda_3 + \lambda_2 \lambda_3 = 1$. From \eqref{lem-detQ-p1},  \eqref{lem-detQ-p2}, and  \eqref{lem-detQ-p3}, we derive that $z$ is a simple root of $H(z)$.

\medskip
We next consider 2). We only consider the case $z = 2/ (3 \sqrt{3})$, the other case follows similarly. By \eqref{lem-hol-lambda} in the proof of \Cref{lem-hol}, we have
\begin{equation}\label{lem-detQ-cc}
 \lambda_1(z+ \eps) =
 - \frac{i}{\sqrt 3} + \frac{\sqrt{ - i}}{3^{1/4}} \sqrt \epsilon + O(\epsilon),  \;
 \lambda_2(z + \eps) =
 - \frac{i}{\sqrt 3}  - \frac{\sqrt{- i}}{3^{1/4}} \sqrt \epsilon + O(\epsilon),
 \;  \lambda_3(z+\epsilon) =  \frac{2i}{\sqrt 3} + O(\epsilon).
\end{equation}
It follows that
$$
\det Q(z + \eps) = - \frac{2 L i}{\sqrt{3}}  \frac{\sqrt{-i} }{3^{1/4}} \sqrt{\eps} + O(\eps).
$$
Since $\Xi(z + \eps) = c_+ \sqrt{\eps}$  for some $c_+ \neq 0$ by \eqref{lem-detQ-cc}, $z = 2/ (3 \sqrt{3})$ is not a root of the equation $H(z) = 0$ and $z$  is a simple root of the equation $\det Q(z) \Xi (z) = 0$.
The proof is complete.
\end{proof}

\begin{lemma}\label{lem-HG} Let $z \in \mC$ be such that  $z \neq  \pm 2/ (3 \sqrt{3})$. Assume that $H(z) = G(z) = 0$. Then, for some $k, l \in \N$ with $k \ge l \ge 1$, we have
\begin{equation}
L = 2 \pi \sqrt{\frac{k^2 + k l + l^2}{3}},
\end{equation}
and
\begin{equation}
z = - \frac{(2k + l)(k-l)(2 l + k)}{3 \sqrt{3}(k^2 + kl + l^2)^{3/2}}.
\end{equation}
\end{lemma}

\begin{proof} By \Cref{rem-detQ-realroots} (see also \Cref{lem-detQ}), it suffices to prove that if $z \in \mC$ is such that  $z \neq  \pm  2 / (3  \sqrt{3})$, and
$H(z) = G(z) = 0$, then $z$ is real. Indeed, note that
\begin{equation*}
\det Q (z)  = (\lambda_1 - \lambda_3) (e^{-\lambda_2 L } - e^{-\lambda_3 L}) + (\lambda_3 - \lambda_2) (e^{-\lambda_1 L } - e^{-\lambda_3 L }),
\end{equation*}
and
$$
- P(z) =
 (\lambda_1 - \lambda_3) (e^{\lambda_2 L } - e^{\lambda_3 L}) + (\lambda_3 - \lambda_2) (e^{\lambda_1 L } - e^{\lambda_3 L }).
$$
It follows that
\begin{equation}\label{eq-lambda-3}
|\det Q (z)| = 0 \mbox{ if and only if } (\lambda_ 3- \lambda_1) (e^{(\lambda_3 - \lambda_2) L } - 1) =  (\lambda_3 - \lambda_2) (e^{(\lambda_3 - \lambda_1) L }  -1),
\end{equation}
and
\begin{equation}\label{eq-lambda-4}
|P(z)|= 0 \mbox{ if and only if } (\lambda_ 3- \lambda_1) (e^{-(\lambda_3 - \lambda_2) L } - 1) =  (\lambda_3 - \lambda_2) (e^{-(\lambda_3 - \lambda_1) L }  -1).
\end{equation}

Solving the system
\begin{equation}\left\{
\begin{array}{c}
\sum_{j=1}^3 \lambda_j = 0, \\[6pt]
\sum_{j=1}^3 \lambda_j \lambda_{j+1} = 1,
\end{array}\right.
\end{equation}
in which $\lambda_3$ is a parameter, one has, with $\Delta = - 3 \lambda_3^2 - 4$,
\begin{equation*}
\lambda_1 = \frac{-\lambda_3 + \sqrt{\Delta}}{2} \quad \mbox{ and } \quad \lambda_2 = \frac{-\lambda_3 - \sqrt{\Delta}}{2}.
\end{equation*}
This implies
\begin{equation}\label{ap-p1-*}
\alpha = \alpha (\lambda_3) = \lambda_3 - \lambda_1 =  \frac{3 \lambda_3 - \sqrt{\Delta}}{2} \quad \mbox{ and } \quad  \beta = \beta(\lambda_3) = \lambda_3 - \lambda_2 =  \frac{3 \lambda_3 + \sqrt{\Delta}}{2}.
\end{equation}

Thus, if $z$ is a common root of $|\det Q|$ and $|P|$ and $\lambda_i(z) \neq \lambda_j(z)$ for $i \neq j$ ($1 \le i, j \le 3$), then, by \eqref{eq-lambda-3} and \eqref{eq-lambda-4},
\begin{equation*}
(e^{\alpha L } - 1) (e^{- \beta L }  -1) = (e^{- \alpha L } - 1)  (e^{\beta L }  -1),
\end{equation*}
which is equivalent to
\begin{equation*}
(e^{ \alpha L} -  e^{ \beta L} ) (e^{ \alpha L } - 1) (e^{ \beta L }  -1) = 0.
\end{equation*}
This implies that  either $e^{ \alpha L} =  e^{ \beta L}$, or  $e^{\alpha L } = 1$, or $e^{\beta L } = 1$. Since $\lambda_1, \lambda_2, \lambda_3$ are distinct, it follows from \eqref{eq-lambda-3} and \eqref{eq-lambda-4} that
\begin{equation}
e^{\alpha L } =  e^{\beta L }  = 1.
\end{equation}
We derive from \eqref{ap-p1-*} that
$$
3 \lambda_3 \in 2 \pi i \mZ/L.
$$
Since
$$
\lambda_3^3 + \lambda_3 = - i z,
$$
it follows that $z$ is real. The proof is complete.
\end{proof}

We finally establish
\begin{lemma}\label{lem-bh-detQ} There exist $c,  \, C>0$ and  $m_0 \in \N$ such that
\begin{enumerate}
\item[1)] for $m \in \mZ$ with $|m| \ge m_0$, we have
\[
\lvert\det Q(z)\rvert \ge C e^{ - c |z|^{1/3}} \mbox{ if } \Im (z) =  \Big((2 m + 1) \pi
/ (\sqrt{3} L)  \Big)^3.
\]
\item[2)] for $z \in \mC$ with $|z| \ge m_0$ and $|\Re(z) | \ge c |z|^{1/3}$, we have
\[
\lvert\det Q(z)\rvert \ge C e^{ - c |z|^{1/3}}.
\]
\end{enumerate}
\end{lemma}

\begin{proof} For $z \in \mC$ with large $|z|$, denote $\lambda_1, \lambda_2, \lambda_3$ be the three roots of the equation
\[
\lambda^3 + \lambda = - iz,
\]
with the convention $ \Re (\lambda_3) \ge \max \big\{\Re (\lambda_1), \Re (\lambda_2) \big\}$, and, with $\Delta = - 3 \lambda_3^2 - 4$,
\[
\lambda_1 = \frac{-\lambda_3 + \sqrt{\Delta}}{2}, \quad \mbox{ and } \quad \lambda_2 = \frac{-\lambda_3 - \sqrt{\Delta}}{2}.
\]
This is possible since
\begin{equation*}\left\{
\begin{array}{c}
\lambda_1 + \lambda_2 =  - \lambda_3, \\[6pt]
\lambda_1 \lambda_2 = 1 + \lambda_3^2. 
\end{array}\right.
\end{equation*}

We have
\[
|\lambda_3^{-1} \det Q(z) e^{\lambda_3 L}|  = |f (\lambda_3)|,
\]
where
\begin{equation}\label{eq-lambda-3-1}
 f (\lambda_3) : = \frac{3 \lambda_3 - \sqrt{\Delta}}{2 \lambda_3} (e^{\frac{3 \lambda_3 + \sqrt{\Delta}}{2} L } - 1) - \frac{3 \lambda_3 + \sqrt{\Delta}}{2 \lambda_3} (e^{\frac{3 \lambda_3 - \sqrt{\Delta}}{2} L } - 1).
\end{equation}
Since $\lambda_3$ is large, we have
\begin{multline}\label{f-lambda}
\left(\frac{3 - i \sqrt{3}}{2} \right)^{-1}f(\lambda_3) = [1 + O(\lambda_3^{-2}) ](e^{\frac{3 + i \sqrt{3}}{2} \lambda_3 L + O(\lambda_3^{-1}) } - 1) \\[6pt]
 -  [1 + O(\lambda_3^{-2}) ] e^{ i \varphi_0} (e^{\frac{3 - i \sqrt{3}}{2} \lambda_3 L  + O(\lambda_3^{-1})} - 1),
\end{multline}
where $\varphi_0 = \pi/3$ since $ \frac{3 + i \sqrt{3}}{2} / \frac{3 - i \sqrt{3}}{2}  = e^{ i \varphi_0}$.

We begin with 1).  It suffices to prove, for $z \in \mC$ with $\Im (z) =  \Big((2 m + 1) \pi / (\sqrt{3} L)  \Big)^3$ with large $|m|$ ($m \in \mZ$), that
\begin{equation}\label{claim-detQ}
|\lambda_3^{-1} \det Q(z) e^{\lambda_3 L}| \ge 1.
\end{equation}
Assume that \eqref{claim-detQ} does not hold. Then for some $m \in \mZ$ with large modulus  and for some $z \in \mC$ with $\Im (z) =  \Big((2 m + 1) \pi / (\sqrt{3} L)  \Big)^3$, we have
\[
|f(\lambda_3) | \le 1.
\]
Since $\Re(\lambda_3) > 0$ and is large, it follows that
\[
|e^{\frac{3 + i \sqrt{3}}{2} \lambda_3 L }| = (1 + O(\lambda_3^{-1})) |e^{\frac{3 - i \sqrt{3}}{2} \lambda_3 L }|.
\]
One  derives that, if $ \lambda_3  = a  + i b$ with $a, b \in \mR$,
\begin{equation}\label{ccc}
a \mbox{ is  large  and } |b|  = O(\lambda_3^{-1}).
\end{equation}
It follows that
\begin{equation*}
e^{\frac{3 + i \sqrt{3}}{2} \lambda_3 L } = e^{\frac{3 a L}{2}} e^{i \frac{\sqrt{3} a L }{2}  } e^{O(\lambda_3^{-1})}
\end{equation*}
and
\begin{equation*}
e^{\frac{3 - i \sqrt{3}}{2} \lambda_3 L } = e^{  \frac{3 a L }{2}}  e^{ - i  \frac{\sqrt{3} a L }{2}} e^{O(\lambda_3^{-1})}.
\end{equation*}
Using \eqref{f-lambda}, and the fact $|f(\lambda_3)| \leq 1$ and  $\Im (z) =  \Big((2 m + 1) \pi / (\sqrt{3} L)  \Big)^3$, we obtain a contradiction.  Hence \eqref{claim-detQ} holds. The proof of 1) is complete.

To establish 2), it suffices to prove \eqref{claim-detQ} for $z \in \mC$ with $|z| \ge m_0$ and $|\Re(z)| \ge c|z|^{1/3}$ for some $c >0$.  This indeed follows from the fact if $|z|$ is large and $|f(\lambda_3)| \le 1$, then \eqref{ccc} holds.  The proof is complete.
\end{proof}

\medskip
\noindent \textbf{Acknowledgments.} The authors were partially supported by  ANR Finite4SoS ANR-15-CE23-0007. 
A. Koenig is supported by a public grant overseen by the French National Research Agency (ANR) as part of the ``Investissements d'Avenir'''s program of the Idex PSL reference ANR-10-IDEX-0001-02 PSL.  H.-M. Nguyen thanks Fondation des Sciences Math\'ematiques de Paris (FSMP) for the Chaire d'excellence which allows him to visit  Laboratoire Jacques Louis Lions  and Mines ParisTech. This work has been done during this visit. 

\providecommand{\bysame}{\leavevmode\hbox to3em{\hrulefill}\thinspace}
\providecommand{\MR}{\relax\ifhmode\unskip\space\fi MR }
\providecommand{\MRhref}[2]{%
  \href{http://www.ams.org/mathscinet-getitem?mr=#1}{#2}
}
\providecommand{\href}[2]{#2}

\end{document}